\documentclass[usenames,dvipsnames]{article}
\usepackage{blindtext,inputenc,graphicx,overpic}
\usepackage{amsmath,amssymb,amsthm,bm,mathrsfs,mathtools,enumitem}

\usepackage{multicol}
\usepackage{esint,braket}
\usepackage{fullpage,setspace}
\usepackage{authblk}
\usepackage{verbatim} 
\usepackage{caption}
\usepackage{subcaption}

\usepackage{color,soul}

\usepackage[
    backend=biber,
    backref=true,
    style=alphabetic,
    maxbibnames=99,
  ]{biblatex}
\DefineBibliographyStrings{english}{
  backrefpage={p.},
  backrefpages={p.}
}

\usepackage{tikz}
\usetikzlibrary{arrows.meta,decorations.markings,decorations.pathmorphing}
\usepackage{pgfplots}
\usepgfplotslibrary{polar}

\usepackage{hyperref}
\hypersetup{
    colorlinks=true,
    linkcolor=BrickRed,
    filecolor=red,      
    urlcolor=Periwinkle,
    citecolor=teal,
}
\usepackage{framed}
\definecolor{shadecolor}{rgb}{1.0,0.97,0.0}

\theoremstyle{definition}
\newtheorem{defn}{Definition}[section]
\newtheorem{theorem}{Theorem}[section]
\newtheorem{lemma}{Lemma}[section]
\newtheorem{remark}{Remark}[section]
\newtheorem{prop}{Proposition}[section]
\newtheorem{cor}{Corollary}[section]

\newtheorem{conjecture}{Conjecture}
\newtheorem{problem}{Riemann-Hilbert Problem}[section]

\DeclareMathOperator{\tr}{tr}

\DeclareMathOperator{\CC}{\mathbb{C}}
\DeclareMathOperator{\RR}{\mathbb{R}}

\DeclareMathOperator{\ZZ}{\mathbb{Z}}

\DeclareMathOperator{\OO}{\mathcal{O}}

\DeclareMathOperator{\Ai}{Ai}

\DeclareMathOperator*{\Ress}{Res}
\newcommand{\Res}{\displaystyle\Ress}

\numberwithin{equation}{section}
\numberwithin{figure}{section}

\usepackage{todonotes}
\usepackage{lineno}

\usepackage{titlesec}
\titleformat{\section}{\centering\normalfont\scshape}{\thesection .}{1em}{}
\titleformat{\subsection}{\normalfont\scshape}{\thesubsection .}{1em}{}
\titleformat{\subsubsection}{\normalfont\scshape}{\thesubsubsection .}{1em}{}
\usepackage{abstract}

\title{\bf Asymptotic Properties of a Special Solution to the (3,4) String Equation}

\date{\today}

\author[1]{\scshape Nathan Hayford\thanks{\href{mailto:nhayford@kth.se}{nhayford@kth.se}}}
\affil[1]{\small\textit{Department of Mathematics, Royal Institute of Technology (KTH), Stockholm, Sweden}}

\begin{document}

\maketitle
\vspace{8mm}
\begin{abstract}
    We analyze the asymptotic properties a special solution of the $(3,4)$ string equation, which appears in the study of the multicritical 
    quartic $2$-matrix model. In particular, we show that in a certain parameter regime, the corresponding $\tau$-function has an asymptotic
    expansion which is `topological' in nature. Consequently, we show that this solution
    to the string equation with a specific set of Stokes data exists, at least asymptotically. We also demonstrate that, along specific curves 
    in the parameter space, this $\tau$-function degenerates to the $\tau$-function for a tritronqu\'{e}e solution of Painlev\'{e} I 
    (which appears in the critical quartic $1$-matrix model), indicating that there is a `renormalization group flow' between these 
    critical points. This confirms a conjecture from \cite{CGM}.
\end{abstract}

\tableofcontents

\section{Introduction.}
In this work, we will study the asymptotic properties of a special solution to the \textit{$(3,4)$ string equation}:
    \begin{equation} \label{string-equation}
        \begin{cases}
             0 = \frac{1}{2}V'' - \frac{3}{2}UV + \frac{5}{2}t_5 V + t_2,\\
             0 = \frac{1}{12} U^{(4)} -\frac{3}{4}U''U -\frac{3}{8}(U')^2+\frac{3}{2}V^2 + \frac{1}{2}U^3 - \frac{5}{12}t_5\left(3U^2 - U''\right) + t_1.
        \end{cases}
    \end{equation}
In the above, $U = U(t_5,t_2,t_1)$, $V = V(t_5,t_2,t_1)$ and $' = \frac{\partial}{\partial t_1}$. The above is an ordinary differential 
equation depending parametrically on the variables $t_5,t_2$. This equation first appeared in the study of the critical $2$-matrix 
model \cite{CGM,Douglas1}. Essentially, the main message of these works was that the critical partition function for the $2$-matrix model 
under study converges in a multi-scaling limit to a $\tau$-function for \eqref{string-equation}. This $\tau$-function then can be interpreted as a 
partition function for a theory of topological gravity coupled to the Ising conformal field theory \cite{Krichever,DKK}.  A rigorous proof of this 
statement was finally made in \cite{DHL3}. However, the properties of this solution (even its existence!) have yet to be studied. 

One property of the string equation \eqref{string-equation} is that it has the \textit{Painlev\'{e} property}: its meromorphic solutions have 
as possible singularities only movable poles and/or fixed branch points. This property is the defining feature of the original $6$ Painlev\'{e} equations, and
is also shared by solutions to the Painlev\'{e} hierarchies. Painlev\'{e} I -- VI all admit a 
(nonautonomous) Hamiltonian  formulation \cite{Okamoto3}, as well as a representation as the isomonodromic deformations of linear differential equations
with rational coefficients \cite{JMU2,FIKN}. The same is true for both the Painlev\'{e} I and II hierarchies (cf. \cite{Takasaki,MM,CJM} and 
references therein). As a consequence of this fact and the seminal works \cite{JMU1,JMU2,JMU3}, we can associate to each solution of these equations an \textit{isomonodromic $\tau$-function}, an object which will play an important role in the present work.
Furthermore, Painlev\'{e} transcendents appear frequently in the context of random matrix theory. Special solutions to the Painlev\'{e} I and II hierarchies appear in the universal expressions for critical eigenvalue correlation kernels in random matrix theory \cite{IB,CV}, and the partition functions for
random matrix models converge (after appropriate normalization) to $\tau$-functions of Painlev\'{e} transcendents (see \cite{BleherDeano}, and further
conjectures of this claim from physics \cite{GGPZ,DFGZJ}). Painlev\'{e} equations have also found numerous applications in integrable systems \cite{Dubrovin2,CV2} and combinatorics \cite{BDJ,BleherDeano}, among other areas of mathematics and mathematical physics.

All of the aforementioned examples of Painlev\'{e}-type equations are examples of \textit{rank-2 isomonodromic systems}: they can be realized as isomonodromic
deformations of some $2\times 2$ matrix-valued linear differential equation with rational coefficients. Aside from some general theory \cite{JMU1,BHH}, not 
much is known about higher-rank isomonodromic systems, although historically many such systems are of interest to the random matrix theory community. For example, the so-called $(q,p)$ string equations are conjectured to play a role in the classification of universality classes of critical phenomenona in multi-matrix models \cite{Douglas1,GM,Kazakov3}. These equations arise from rank $q$ isomonodromic systems, for arbitrary $q\geq 2$. This can be contrasted to the
much thinner spectrum of critical phenomena in the $1$-matrix model, which are well understood to be indexed by the Painlev\'{e} I and II hierarchies, which correspond to the  $(2,2g+1)$ and $(2,2g)$ families of string equations, respectively.

The first nontrivial instance of such a higher rank system (i.e., one relevant to
random matrix theory\footnote{The connection of this equation to
random matrix theory is explained in Appendix \ref{Appendix:perturbative-expansion}.}, which is not reducible to a rank $2$ isomonodromic system) is the one studied in the present work. In other words,
\eqref{string-equation} can be realized as the equation governing the isomonodromic deformations of a rank $3$ system. 
Part of the subject of \cite{DHL2} was to show that although this system is not of rank $2$, essentially all of the properties enjoyed by the Painlev\'{e}
equations are shared by the string equation \eqref{string-equation}. Let us discuss some of these properties here.

First, we note that the $(3,4)$ string equation admits a (non-autonomous) Hamiltonian formulation \cite{DHL2}, not only in the variable $t_1$, but also
in the variables $t_2$, $t_5$. In other words, there exist functions $P_U,P_V,P_W, Q_U,Q_V,Q_W$ of $t_1,t_2, t_5$, and polynomials $H_1,H_2,H_5$ in the variables $\{P_a,Q_a\}_{a\in \{U,V,W\}}$ and $t_1,t_2,t_5$, such that the family of equations
    \begin{equation}
        \frac{\partial Q_a}{\partial t_k} = \frac{\partial H_1}{\partial P_a}, \qquad \frac{\partial P_a}{\partial t_k} = -\frac{\partial H_1}{\partial Q_a}, \qquad \qquad a\in\{U,V,W\}, \quad k \in \{1,2,5\},
    \end{equation}
are equivalent to the $(3,4)$ string equation \eqref{string-equation}, along with its compatible flows along the $t_2$ and $t_5$ variables. The exact form of these Darboux coordinates and Hamiltonians are given in Appendix \ref{Appendix:Hamiltonian}.

Furthermore, as was demonstrated in \cite{DHL2}, the $(3,4)$ string equation carries a \textit{Lax pair}:
\begin{align}
        &\mathcal{Q}(\lambda;t_{5},t_{2},t_{1}) := 
            \begin{psmallmatrix}
                0 & 0 & 1\\
                0 & 0 & 0\\
                0 & 0 & 0
            \end{psmallmatrix}\lambda + 
            \begin{psmallmatrix}
                0 & \frac{3}{4}U & -\frac{3}{2}V\\
                1 & 0 & \frac{3}{4}U\\
                0 & 1 & 0
            \end{psmallmatrix},\\
        &\mathcal{P}(\lambda;t_{5},t_{2},t_{1})  := 
            \begin{psmallmatrix}
                0 & 0 & 1\\
                0 & 0 & 0\\
                0 & 0 & 0
            \end{psmallmatrix}\lambda^2 + 
                \begin{psmallmatrix}
                    0 & \frac{5}{3}t_{5} + \frac{1}{4}U & -V\\
                    1 & 0 & \frac{5}{3}t_{5}+\frac{1}{4}U\\
                    0 & 1 & 0
                \end{psmallmatrix}\lambda  \label{spectral-lax-operator}\\
             &+  \begin{psmallmatrix}
                    \frac{1}{2}V' -\frac{1}{12}U'' + \frac{1}{8}U^2 -\frac{5}{12}t_{5} U & \frac{1}{12} U''' - \frac{7}{16}UU'-\frac{3}{8}UV +\frac{5}{12}t_{5} U' + t_{2} & \frac{1}{16}(U')^2 -\frac{1}{8}U U'' + \frac{7}{32}U^3 + \frac{3}{4}V^2-\frac{5}{12}t_{5} U^2 + t_{1}\\
                    \frac{1}{2}V-\frac{1}{4}U & \frac{1}{6}U'' - \frac{1}{4}U^2 + \frac{5}{6}t_{5} U & -\frac{1}{12} U''' + \frac{7}{16}U U'-\frac{3}{8}UV -\frac{5}{12}t_{5} U' + t_{2}\\
                    \frac{5}{3}t_{5}-\frac{1}{2}U &\frac{1}{2}V+\frac{1}{4}U & -\frac{1}{2}V' -\frac{1}{12}U'' + \frac{1}{8}U^2 -\frac{5}{12}t_{5} U
                 \end{psmallmatrix}, \nonumber
    \end{align}
with the Lax equation
    \begin{equation}
        \frac{\partial \mathcal{P}}{\partial t_1} - \frac{\partial \mathcal{Q}}{\partial \lambda} + [\mathcal{P},\mathcal{Q}] = 0
    \end{equation}
reproducing the string equation. This makes the $(3,4)$ string equation integrable in the sense of Lax, similarly to the rest of the Painlev\'{e} transcendents and their associated hierarchies \cite{FIKN}.

Finally, as was demonstrated in \cite{DHL2}, the string equation is equivalent to the isomonodromic deformation equations of a linear differential equation
with rational coefficients. We state this Riemann-Hilbert problem (RHP) here:

\begin{problem} \label{prob:MAIN-RHP}
    Let $\omega := e^{\frac{2\pi i}{3}}$, and define contours
\begin{align*}
        \Gamma_{\pm k} := \left\{\lambda \big| \arg \lambda = \pm \frac{\pi}{14} \pm \frac{\pi}{7}(k-1) \right\}, \qquad k = 1,...,7,
    \end{align*}
and $\RR_- := (-\infty,0)$. Find a $3\times 3$ sectionally analytic function $\Psi(\zeta;t_5,t_2,t_1)$ such that:
\begin{equation}\label{psi-asymptotics}
            \begin{cases}
                \Psi_{+}(\zeta;t_5,t_2,t_1) =  \Psi_{-}(\zeta;t_5,t_2,t_1) S_k, & \zeta \in \Gamma_k,\qquad k = \pm 1, ...,\pm 7,\\
                \Psi_{+}(\zeta;t_5,t_2,t_1) =  \Psi_{-}(\zeta;t_5,t_2,t_1) \mathcal{S}, & \zeta \in \RR_-,\\
                \Psi(\zeta;t_5,t_2,t_1) = f(\zeta)\left[\mathbb{I} + \frac{\Psi_1}{\zeta^{1/3}} + \frac{\Psi_2}{\zeta^{2/3}} + \OO(\zeta^{-1})\right]e^{\Theta(\zeta;t_5,t_2,t_1)}, & \zeta \to \infty,
            \end{cases}
        \end{equation}
    where $f(\zeta), \Theta(\zeta;t_5,t_2,t_1)$ are given by
\begin{equation}\label{gauge-matrix}
        f(\zeta) := 
        \frac{i}{\sqrt{3}}\underbrace{\begin{pmatrix}
            \zeta^{1/3} & 0 & 0\\
            0 & 1 & 0\\
            0 & 0 & \zeta^{-1/3}
        \end{pmatrix}}_{\zeta^{\Delta}}
        \underbrace{\begin{pmatrix}
            1 & \omega & \omega^2\\
            1 & 1 & 1\\
            1 & \omega^2 & \omega
        \end{pmatrix}}_{-i\sqrt{3}\mathcal{U}},
    \end{equation}
    \begin{equation}
        \Theta(\zeta;t_5,t_2,t_1) := \text{diag }(\vartheta_1(\zeta;t_5,t_2,t_1),\vartheta_2(\zeta;t_5,t_2,t_1),\vartheta_3(\zeta;t_5,t_2,t_1)),
    \end{equation}
with $\vartheta_j(\zeta;t_5,t_2,t_1) = \frac{3}{7}\omega^{j-1}\zeta^{7/3} + \omega^{1-j}t_5\zeta^{5/3} + \omega^{1-j}t_2\zeta^{2/3} + \omega^{j-1}t_1\zeta^{1/3}$, the jump matrices $S_k$ are given in Figure \eqref{fig:43-equation-jumps}, and the matrix $\mathcal{S}$ is
    \begin{equation}
        \mathcal{S} := 
        \begin{psmallmatrix}
                0 & 1 & 0\\
                0 & 0 & 1\\
                1 & 0 & 0
            \end{psmallmatrix}.
    \end{equation}
These matrices must satisfy the following constraint equation (Stokes equation):
    \begin{equation}\label{Stokes-Equation}
        S_{-7}\cdots S_{-1}S_{1}\cdots S_{7} = \mathcal{S}^T.
    \end{equation}
The matrices $\Psi_j := \Psi_j(t_5,t_2,t_1)$, $j=1,2$, are given by
    \begin{align}
        \Psi_1(t_5,t_2,t_1) &= 
            \begin{psmallmatrix}
                H_1 & 0 & 0\\
                0 & \omega^2 H_1 & 0\\
                0 & 0 & \omega H_1
            \end{psmallmatrix}\\
        \Psi_2(t_5,t_2,t_1) &=
        \begin{psmallmatrix}
            \frac{1}{2}(H_1)^2+\frac{1}{2}H_2 & -\frac{i\omega^2\sqrt{3}}{12} U & \frac{i\omega\sqrt{3}}{12} U \\
            \frac{i\omega^2\sqrt{3}}{12} U& \omega \left(\frac{1}{2}(H_1)^2+\frac{1}{2}H_2\right) & -\frac{i\sqrt{3}}{12} U\\
            -\frac{i\omega\sqrt{3}}{12}  U &  \frac{i\sqrt{3}}{12} U& \omega^2 \left(\frac{1}{2}(H_1)^2+\frac{1}{2}H_2\right)
        \end{psmallmatrix},
    \end{align}
where $H_1$, $H_2$ are the Hamiltonians given above, and $U,V$ are solutions to the string equation. The above jump conditions and 
asymptotics, along with the determination of $\Psi_j$, $j=1,2$, determine the solution to the above Riemann-Hilbert problem uniquely. We remark that the
coefficients $\Psi_j(t_5,t_2,t_1)$ carry the symmetry
    \begin{equation}\label{psi-symmetry}
        \Psi_j(t_5,t_2,t_1) = \omega^{-j}\mathcal{S}^T \Psi_j(t_5,t_2,t_1) \mathcal{S}.
    \end{equation}
\end{problem}

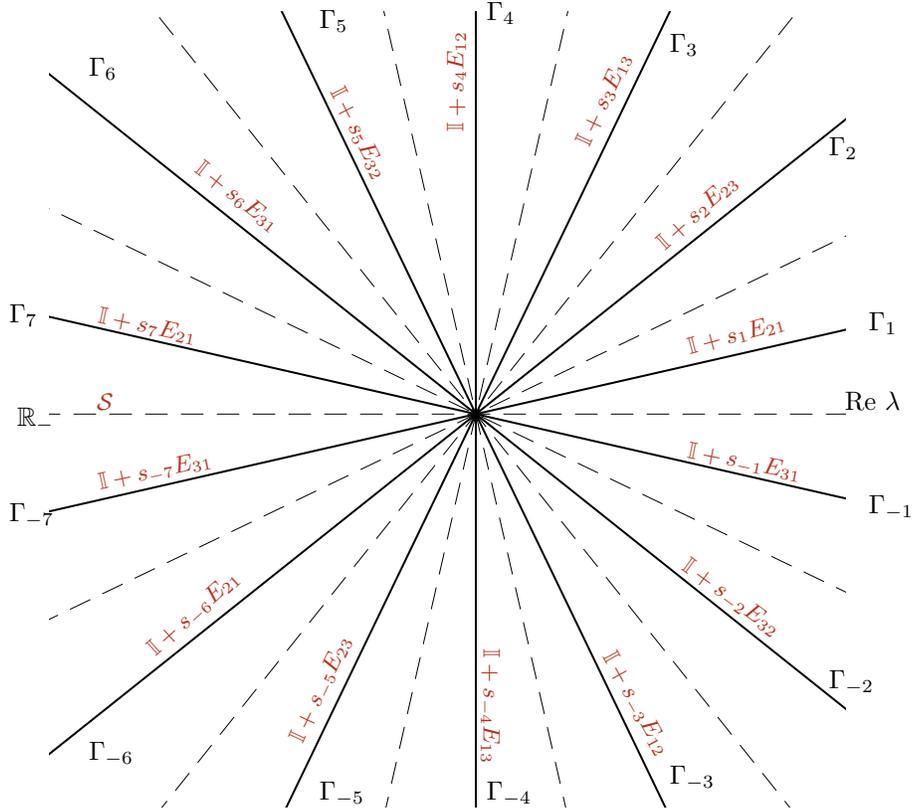
\begin{figure}
    \begin{center}
    \begin{overpic}[scale=.5]{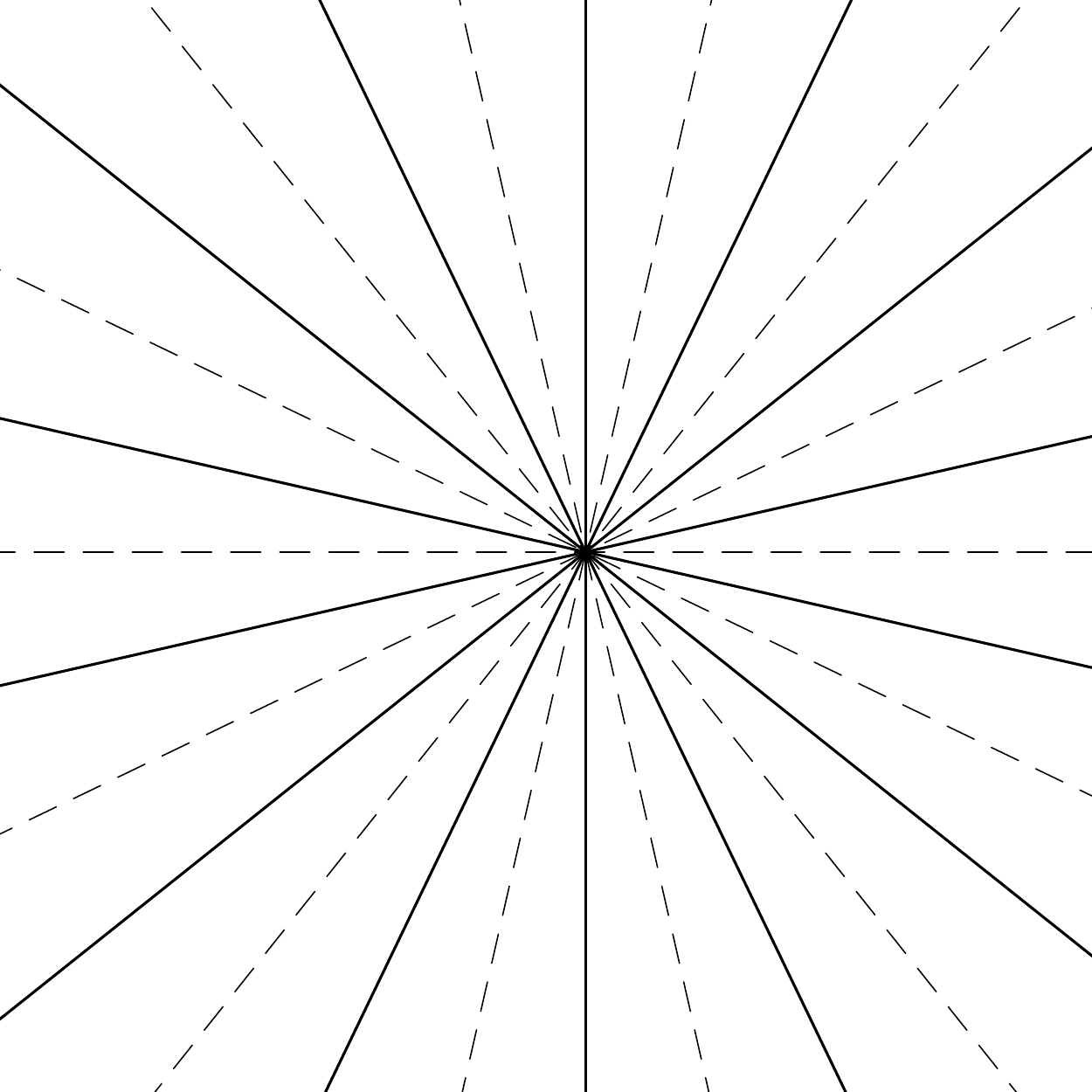}
                \put (100,50) {$\text{Re } \lambda$}
          	\put (103,60) {$\Gamma_1$}
      		  \put (103,37) {$\Gamma_{-1}$}
     		  \put (98,82) {$\Gamma_2$}
      		  \put (98,15) {$\Gamma_{-2}$}
                \put (78,95) {$\Gamma_3$}
      		  \put (78,3) {$\Gamma_{-3}$}
                \put (55,99) {$\Gamma_4$}
      		  \put (55,1) {$\Gamma_{-4}$}
                \put (34,98) {$\Gamma_5$}
      		  \put (34,1) {$\Gamma_{-5}$}
                \put (5,92) {$\Gamma_6$}
      		  \put (5,6) {$\Gamma_{-6}$}
                \put (-5,61) {$\Gamma_7$}
      		  \put (-5,36) {$\Gamma_{-7}$}
                \put (-4,48) {$\RR_-$}
                \put (80,57){\small \rotatebox{15}{\textcolor{BrickRed}{$\mathbb{I} + s_1E_{21}$}} } 
                \put (80,44){\small \rotatebox{-11}{\textcolor{BrickRed}{$\mathbb{I} + s_{-1}E_{31}$}} } 
                \put (76,70){\small \rotatebox{42}{\textcolor{BrickRed}{$\mathbb{I} + s_2E_{23}$}}  } 
                \put (79,30){\small \rotatebox{-36}{\textcolor{BrickRed}{$\mathbb{I} + s_{-2}E_{32}$}}  } 
                \put (66,83){\small \rotatebox{65}{\textcolor{BrickRed}{$\mathbb{I} + s_{3}E_{13}$}}  } 
                \put (69,19){\small \rotatebox{-61}{\textcolor{BrickRed}{$\mathbb{I} + s_{-3}E_{12}$}}  } 
                \put (50,85){\small \rotatebox{90}{\textcolor{BrickRed}{$\mathbb{I} + s_{4}E_{12}$}}  } 
                \put (54,20){\small \rotatebox{-90}{\textcolor{BrickRed}{$\mathbb{I} + s_{-4}E_{13}$}}  } 
                \put (35,90){\small \rotatebox{-63}{\textcolor{BrickRed}{$\mathbb{I} + s_{5}E_{32}$}}  } 
                \put (30,8){\small \rotatebox{65}{\textcolor{BrickRed}{$\mathbb{I} + s_{-5}E_{23}$}}  } 
                \put (18,80){\small \rotatebox{-38}{\textcolor{BrickRed}{$\mathbb{I} + s_{6}E_{31}$}}  } 
                \put (12,19){\small \rotatebox{42}{\textcolor{BrickRed}{$\mathbb{I} + s_{-6}E_{21}$}}  } 
                \put (6,61){\small \rotatebox{-12}{\textcolor{BrickRed}{$\mathbb{I} + s_{7}E_{21}$}}  } 
                \put (6,40){\small \rotatebox{12}{\textcolor{BrickRed}{$\mathbb{I} + s_{-7}E_{31}$}}  } 
                \put (6,50){\small \rotatebox{0}{\textcolor{BrickRed}{$\mathcal{S}$}}  } 
    		\end{overpic}
    \end{center}
    \caption{The Stokes lines $\Gamma_j$ for the Riemann-Hilbert problem for $\Psi(\zeta;t_5,t_2,t_1)$. Each of the Stokes sectors is bisected by an anti-Stokes line, depicted by a dashed line. All contours are oriented \textbf{\textit{outwards}} from the origin. The anti-Stokes line $(-\infty,0]$ is labeled by $\RR_-$. The Stokes matrix $S_k$ is the matrix associated to the parameter $s_k$; these parameters are not all independent, and must satisfy the equation
    $S_{-7}\cdots S_{-1}S_{1}\cdots S_{7} = \mathcal{S}^T$.}
    \label{fig:43-equation-jumps}
\end{figure}

Note that the structure of this RHP (and the appearance of $\omega$) is similar to that which appears in the recent works \cite{CLW,WZZ}, which study the ``good'' Boussinesq equation and its modified versions; one should note that the solution $U$ to the string equation solves the Boussinesq equation in the variables $t_2=t,t_1=x$ (cf. \cite{DHL2}, equation A.28). In \cite{DHL2}, it was shown that a solution to the above Riemann-Hilbert problem exists if and only if a corresponding solution to the string
equation exists, provided $(t_5,t_2,t_1)$ is not a singularity of this solution. However, existence of any particular solution for given Stokes data was 
left open.
In this work, we are interested in a very specific version of this Riemann-Hilbert problem, namely, the one with Stokes parameters
    \begin{align}\label{STOKES_TRUNCATED}
        s_1 &= 0,\qquad s_2 = -1,\qquad s_3 = 0,\qquad s_4 = 0,\qquad s_5 = 1,\qquad s_6 = -1, \qquad s_7 = 0,\nonumber\\
        s_{-1} &= 0,\qquad s_{-2} = 1,\qquad s_{-3} = -1,\qquad s_{-4} = 0,\qquad s_{-5} = 0,\qquad s_{-6} = 1,\qquad s_{-7} = 0.
    \end{align}
This Stokes data is the same as what appears in the multicritical quartic $2$-matrix model \cite{DHL3}, and so directly applies to this 
situation. \textbf{\textit{From here on, we will work exclusively with the Riemann-Hilbert Problem \ref{prob:MAIN-RHP} with Stokes data given by Equation \eqref{STOKES_TRUNCATED}. }}

\begin{remark} \textit{Symmetry properties of the Stokes data.}
One can readily check that the above set of parameters is indeed a solution to the constraint equation \eqref{Stokes-Equation}.
It is easy to see that the above parameters carry the symmetry 
    \begin{equation*}
        s_k = -s_{k+8}, \qquad\qquad k=-7,...,-1,
    \end{equation*}
and consequentially (cf. \cite{DHL2}, Section 4.3), the functions $U,H_1,H_5$ are even functions of $t_2$, and $V, H_2$ are odd functions of $t_2$. 
\end{remark}

The main theorems of this paper pertain to a large-parameter Deift-Zhou analysis of the Riemann-Hilbert problem \ref{prob:MAIN-RHP} with Stokes data \eqref{STOKES_TRUNCATED} in various scaling and double-scaling regimes. The establishment of these asymptotics guarantees existence of a
solution to Problem \ref{prob:MAIN-RHP} (with data \eqref{STOKES_TRUNCATED}) for sufficiently large values of its parameters, thus providing a
partial solution to the problem of existence. We will also study some double-scaling regimes in which solutions to the above equation limit to 
solutions to Painlev\'{e} I.

\subsection{Isomonodromic $\tau$-function.}
Our main theorems will be stated in terms of the so-called \textit{isomonodromic $\tau$-function} of Jimbo, Miwa, and Ueno \cite{JMU1}. 
This is a holomorphic function which is defined in terms of the local data of the RHP \ref{prob:MAIN-RHP}. Usually, the zero locus of a 
$\tau$-function is precisely the set on which its corresponding Riemann-Hilbert problem is not solvable \cite{Palmer,Malgrange}. Furthermore,
solutions to the string equation \eqref{string-equation} can be written as derivatives of the $\tau$-function, and in this sense the 
$\tau$-function is a fundamental object. However, as pointed out in \cite{DHL2}, the definition used in \cite{JMU1} (which we will refer to as the JMU $\tau$ function, and its differential as the JMU \textit{differential}) does not directly
apply to the present situation, as the linear differential equation with rational coefficients associated to RHP \ref{prob:MAIN-RHP} does not
have a diagonalizable leading term. Construction of $\tau$-functions for systems with nondiagonalizable leading terms resonant singularities is a 
problem that has been studied before \cite{BM}. In \cite{DHL2}, an alternative (albeit less general) formula for such a $\tau$-function was 
introduced. The advantage of this formula is that it is amenable to steepest descent analysis, as it involves only data derivable from the 
Riemann-Hilbert problem. 

We recall this expression from \cite{DHL2} in the present context here. 
\begin{defn}
The isomonodromic $\tau$-function associated to Problem \eqref{prob:MAIN-RHP} is given by
        \begin{equation}\label{tau-function-definition}
        {\bf d}\log \tau(t_5,t_2,t_1) = -\sum_{\ell}\bigg(\left\langle \mathfrak{G}^{-1}(\zeta;{\bf t})\mathfrak{G}'(\zeta;{\bf t}) \frac{\partial \Theta(\zeta;{\bf t})}{\partial t_{\ell}}\right\rangle - \left\langle\frac{\Delta}{\zeta}\frac{\partial \mathfrak{G}}{\partial t_{\ell}}(\zeta;{\bf t})\mathfrak{G}^{-1}(\zeta;{\bf t})\right\rangle \bigg)dt_{\ell} ,
    \end{equation}
where $'=\frac{\partial}{\partial \zeta}$, $\langle M(\zeta)\rangle := \Res_{\zeta = \infty}\tr M(\zeta)$, $\mathfrak{G}(\zeta;{\bf t})$ is the subexponential part of the asymptotic expansion of $\Psi(\zeta;{\bf t})$ as defined in \eqref{psi-asymptotics}:
    \begin{equation}
        \mathfrak{G}(\zeta;{\bf t}) = \Psi(\zeta;t_{5},t_{2},t_{1}) e^{-\Theta(\zeta;t_5,t_2,t_1)},
    \end{equation}
$\Delta := \text{diag }(\frac{1}{3}, 0, -\frac{1}{3})$ is as defined in \eqref{psi-asymptotics}, and ${\bf d}$ is the differential in the parameters ${\bf t} = (t_{5},t_{2},t_{1})$. 
\end{defn}

Note that the first term defines the usual JMU differential; the second term, which arises from the resonance of the singularity is a ``correction'' term. 
A theorem of \cite{DHL2} states that this differential is closed, and thus can be integrated (up to an overall constant factor independent of ${\bf t}$) to a 
unique function, up to an additive constant. We shall refer to this class of functions (with the constant of integration left ambiguous) simply as the $\tau$-function, by a slight abuse of notation.

The functions $U$, $V$ are then expressible as derivatives of $\log \tau(t_5,t_2,t_1)$:
    \begin{equation}
        U = -\frac{\partial^2}{\partial t_1^2} \log \tau(t_5,t_2,t_1),\qquad\qquad V = \frac{1}{2}\frac{\partial^2}{\partial t_1 \partial t_2} \log \tau(t_5,t_2,t_1).
    \end{equation}

\subsection{Statement of results.}
    \begin{figure}
        \centering
        \begin{overpic}[scale=.5]{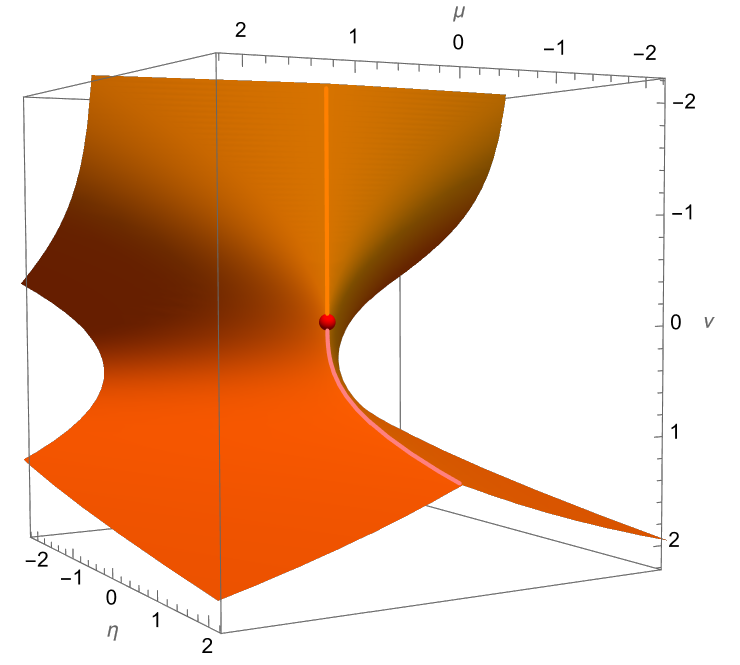}
            \put (45,65) {\small $\gamma_-$}
            \put (48,25) {\small $\gamma_+$}
        \end{overpic}
        \caption{Critical surface in the $\eta,\mu,\nu$ -parameter space. Theorem \ref{MainTheorem1} holds for $(\eta,\mu,\nu)$ below this surface. On this surface, the two curves $\gamma_{-}$ and $\gamma_+$ are shown in orange and pink, respectively; the origin $(0,0,0)$ is depicted in red.}
        \label{fig:43-critical surface}
    \end{figure}

Here, we state the main results of this work. We have four theorems that we wish to state.

Our first result involves the large-parameter asymptotics of the Riemann-Hilbert problem
for $\Psi$: we can say something about the form of the tau function when $t_5,t_2,t_1$ tend to infinity at a certain rate.
This result may be stated explicitly in terms of a particular solution to the following degree $5$ equation:
        \begin{equation}\label{sigma-eq}
            \mathcal{P}(\varsigma;\eta,\mu,\nu) := \nu + \frac{1}{2}\varsigma^3-\frac{5}{4}\eta\varsigma^2 + \frac{6\mu^2}{(5\eta-3\varsigma)^2}= 0.
        \end{equation}
This equation arises naturally in the study of the string equation \eqref{string-equation} under the rescaling \eqref{parameter-rescaling}. This
rescaling is explained in Appendix \ref{Appendix:perturbative-expansion}. Note that
    \begin{equation*}
        \frac{\partial \mathcal{P}}{\partial \varsigma}\bigg|_{\sigma = \frac{5}{2}\eta, \mu = 0,\nu = 0} = \frac{25}{8}\eta^2,
    \end{equation*}
so provided $\eta > 0$, a unique solution $\varsigma(\eta,\mu,\nu)$ exists in a neighborhood of $(\eta,0,0)$.

\begin{defn}\label{Domain-D-Definition}
    We define $D$ to be the connected component of $\RR^3$ containing the ray $\{(\eta,0,0) | \eta >0\}$, and where $\varsigma(\eta,\mu,\nu)$
    as defined above is a simple root of Equation \eqref{sigma-eq}.
\end{defn}
By construction, the region $D$ is connected, and $\varsigma(\eta,\mu,\nu)$ becomes a multiple root of Equation \eqref{sigma-eq} (i.e., the boundary of $D$ is a subset of the zero locus of the discriminant of this equation). This region and its boundary surface are depicted in Figure
\ref{fig:43-critical surface}. This surface has an implicit representation, which is given in Appendix \ref{Appendix:ImplicitCurves}, along with a representation
of the spectral curve. In particular, this surface contains two curves which we shall study in more detail later:
    \begin{equation}\label{gamma-pm}
        \gamma_+:=\left\{\left(\eta,0,\frac{125}{108}\eta^3\right)\big| \eta > 0\right\}, \qquad\qquad \gamma_-:=\left\{\left(\eta,0,0\right)\big| \eta < 0\right\}.
    \end{equation}

With this in place, we can now state our first main result.

\begin{theorem}\label{MainTheorem1}
    For $(\eta,\mu,\nu) \in D$, the $\tau$-function for the $(3,4)$ string equation with Stokes data \eqref{STOKES_TRUNCATED},
    as defined by \eqref{tau-function-definition} admits the $\hbar\to 0$ asymptotic
    expansion
        \begin{equation}\label{genus-0-tau-function}
                \tau(\hbar^{-2/7}\eta,\hbar^{-5/7}\mu,\hbar^{-6/7}\nu) = \frac{C(\hbar)}{\chi(\eta,\mu,\nu)^{1/24}}e^{\hbar^{-2} \varpi_0(\eta,\mu,\nu)}[1 + \OO(\hbar^{2})].
            \end{equation}
        where $C(\hbar)$ is a constant independent of $\eta,\mu,\nu$,
        \begin{align}
            \varpi_0(\eta,\mu,\nu) &:= -\frac{\varsigma^5}{1344}(54\varsigma^2-245\eta\varsigma + 280\eta^2) - \frac{\mu^2 \varsigma^2(50\eta^2-80\eta\varsigma+27\varsigma^2)}{8(5\eta-3\varsigma)^2}+ \frac{\mu^4(25\eta-24\varsigma)}{(5\eta-3\varsigma)^4} ,\\
            \chi(\eta,\mu,\nu) &:= \varsigma(5\eta-3\varsigma)^2-\frac{72\mu^2}{(5\eta-3\varsigma)^2} = -2(5\eta-3\varsigma)\frac{\partial \mathcal{P}}{\partial \varsigma},
        \end{align}
        and $\varsigma$ is the unique solution to the $5^{th}$ order equation \eqref{sigma-eq} on $D$ which is specified in Definition \ref{Domain-D-Definition}.
\end{theorem}
In other words, the formal expansion of the $\tau$-function we obtained directly from the string equation is a true asymptotic expansion of the
$\tau$-function, provided $(\eta,\mu,\nu)$ lie in $D$. Note that the branching behavior of the $\tau$-function becomes singular
as $(\eta,\mu,\nu)$ tend to a point on the critical surface, since $\varsigma$ is a solution to both $\mathcal{P}(\varsigma,\eta,\mu,\nu) = 0$
and $\frac{\partial \mathcal{P}}{\partial \varsigma}(\varsigma,\eta,\mu,\nu) = 0$ there. We will also show in Section \ref{section:g-function} that the 
particular solution $\varsigma(\eta,\mu,\nu) > \frac{5}{3}\eta$ in $D$, so that the branching behavior defined by $\chi(\eta,\mu,\nu)$ 
is always present in $D$.

Our next theorem is that
    \begin{theorem}\label{Theorem:topologicalexpansion}
        For $(\eta,\mu,\nu)\in D$, the $\tau$-differential admits the $\hbar\to 0$ asymptotic (`topological') expansion
            \begin{equation}
                \hbar^2 {\bf d}\log \tau(\hbar^{-2/7}\eta,\hbar^{-5/7}\mu,\hbar^{-6/7}\nu) \sim \sum_{g=0}^{\infty} {\bf d} \log \tau_g(\eta,\mu,\nu) \hbar^{2g},
            \end{equation}
        where $\tau_g(\eta,\mu,\nu)$ are real analytic functions on $D$ which can be determined iteratively.
    \end{theorem}
The interpretation of this expansion as one of topological type comes from the relation of this $\tau$-function to the partition function of the multicritical
quartic $2$-matrix model; this relation is explained further in Appendix \ref{Appendix:perturbative-expansion}, see Remark \ref{topological-expansion-remark}.
The second main result of this work regards a critical double-scaling limit, and demonstrates the degeneration of this solution to the 
$(3,4)$ string equation to the \textit{tritronqu\'{e}e} solution of Painlev\'{e} I that appears in the critical $1$-matrix model \cite{FIK1,FIK2,DK0,BleherDeano}. These are special solutions to Painlev\'{e} I, first studied by Boutroux \cite{Boutroux}, and are pole-free outside of a sector of angle opening $\frac{2\pi}{5}$  (see also \cite{JK} for a more modern treatment). As observed by Crnkovi\'{c}, Ginsparg and Moore \cite{CGM}, if we make the formal rescaling of variables and take a limit as $T\to +\infty$, \footnote{We put $t_2 = 0$ so that $V\equiv 0$ here for simplicity. We have no reason to doubt that this conjecture (and our proof of it) extend to the $t_2\neq 0$ case, if one is willing to work through the technicalities.}
    \begin{equation}\label{q-limit}
        q(x) := \lim_{T\to +\infty} \frac{1}{2}T^{2/5}U(-6T/5,0,T^{1/5}x),
    \end{equation}
then the function $q(x)$ satisfies the Painlev\'{e} I equation:
    \begin{equation}
        q''(x) = 6q^2(x) + x.
    \end{equation}
The exact nature of this convergence is until now conjecture.
Our second main result concerns the resolution (and clarification) of this conjecture. We consider two double-scaling limits of the 
isomonodromic $\tau$-function for the string equation: these cases correspond to double-scaling around points on $\gamma_+$, $\gamma_-$, 
respectively. These two double-scaling limits result in the degeneration of the $\tau$-function for the $(3,4)$ string equation with the
special choice of Stokes data \eqref{STOKES_TRUNCATED} to the $\tau$-function for a tritronque\'{e} solution of the Painlev\'{e} I 
equation. In particular, we shall see that the double-scaling limit on $\gamma_-$ indeed resolves the
nature of the limit \eqref{q-limit}, and demonstrates convergence of $U$ to a tritronqu\'{e}e solution of Painlev\'{e} I.

In order to state our results, we need to introduce the following function, which acts as a normalizing factor for the critical $\tau$-differential.
Given $\eta_0>0$, define
    \begin{equation}\label{tauhat0-def}
        \hat{\tau}_0(\eta,0,\nu) := 
                \exp \left[-\frac{5\eta_0}{6\hbar^2}\left(\nu^2+\frac{125}{108}\eta_0^3\nu-\frac{125}{54}\eta_0^2\eta\nu +\frac{3125}{1296}\eta_0^4\eta^2-\frac{15625}{5832}\eta_0^5\eta \right)\right].
    \end{equation}
This function has the property that
    \begin{equation*}
        \lim_{(\eta,\nu)\to (\eta_0,\nu_0)\in \gamma_+}\hat{\tau}_0(\eta,0,\nu) = \lim_{(\eta,\nu)\to (\eta_0,\nu_0)\in \gamma_+}\exp\left[\frac{1}{\hbar^2}\varpi_0(\eta,0,\nu)\right],
    \end{equation*}
where $\varpi_0(\eta,0,\nu)$ is as defined in Theorem \ref{MainTheorem1}. Let us also introduce the notation
    \begin{equation}
        \tau(\eta,\mu,\nu|\hbar) := \tau(\hbar^{-2/7}\eta,\hbar^{-5/7}\mu,\hbar^{-6/7}\nu).
    \end{equation}

We can now state our next theorem:

\begin{theorem}\label{MainTheorem2}
    Let $(\eta_0,\nu_0) \in \gamma_+$, $\vec{n} = \langle n_{\eta},n_{\nu}\rangle$ be any vector based at $(\eta_0,\nu_0)$
    which lies below the tangent line of the critical curve $\nu = \frac{125}{108}\eta^3$ there, and let $x\in \RR$ be a real variable which is not a pole of the tritronqu\'{e}e Painlev\'{e} I transcendent. 
    Then, considered as a differential in $x$, and for some explicit constant $C = C(\eta_0,\vec{n})>0$,
        \begin{equation}
            \lim_{\hbar\to 0} {\bf d}\log \frac{\tau\left(\eta_0 - C n_{\eta} x\hbar^{4/5}, 0 ,\nu_0 - C n_{\nu}x\hbar^{4/5}|\hbar\right)}{\hat{\tau}_0\left(\eta_0 - C n_{\eta} x\hbar^{4/5}, 0 ,\nu_0 - C n_{\nu}x\hbar^{4/5}\right)} = -\mathcal{H}(x)dx,
        \end{equation}
    where $\mathcal{H}(x) = \frac{1}{2}[q'(x)]^2 -2q(x)^3 - xq(x)$ is the Hamiltonian for a tritronqu\'{e}e solution $q(x)$ of Painlev\'{e} I.
\end{theorem}
In other words, after an appropriate normalization, the $\tau$-function for the $(3,4)$ string equation converges in this limit to the 
$\tau$-function for Painlev\'{e} I. Since $U(t_5,t_2,t_1) = -\frac{\partial^2}{\partial t_1^2}\log \tau(t_5,t_2,t_1)$, by defining the large parameter 
$T :=\frac{10\eta_0}{3}\hbar^{-2/7} \to +\infty$, and choosing the direction $\vec{n} = (0,-1)$,

\begin{cor}
    Let $U(t_5,0,t_1)$ be the solution to the string equation \eqref{string-equation} with Stokes data \eqref{STOKES_TRUNCATED} and $t_2 = 0$, 
    and $x\in \RR$ not a pole of the tritronqu\'{e}e Painlev\'{e} I transcendent. Then, 
        \begin{equation}
            \lim_{T\to +\infty} \left[-\frac{1}{8}T^{7/5} + \frac{1}{4}T^{2/5}U\left(\frac{3T}{10},0,\frac{1}{32}T^3+xT^{1/5}\right)\right] = q(x),
        \end{equation}
    where $q(x)$ is a tritronqu\'{e}e solution of Painlev\'{e} I.
\end{cor}
The factor of $-\frac{1}{8}T^{7/5}$ in the above limit is residual from the normalization factor $\hat{\tau}_0(\eta,0,\nu)$. It is straightforward to see
that at the formal level that this statement holds; to the best of our knowledge, this fact has not been observed in the literature before.

We now state our final theorem, which confirms the conjecture made in \cite{CGM}:
\begin{theorem}\label{MainTheorem3}
    Let $(\eta_0,0) \in \gamma_-$, and let $x\in \RR$ be a real variable. 
    Then, considered as a differential in $x$, and for some explicit constant $C = C(\eta_0)>0$,
        \begin{equation}
            \lim_{\hbar\to 0} {\bf d}\log \tau\left(\eta_0, 0 ,C x\hbar^{4/5}|\hbar\right) = -\mathcal{H}(x)dx,
        \end{equation}
    where $\mathcal{H}(x) = \frac{1}{2}[q'(x)]^2 -2q(x)^3 - xq(x)$ is the Hamiltonian for a solution $q(x)$ of Painlev\'{e} I.
\end{theorem}
It is worth noting here that the above limit does not require a normalization factor, as was the case in Theorem \ref{MainTheorem2}.
Defining the large parameter $T:=-\frac{5\eta_0}{6}\hbar^{-2/7} \to +\infty$, we obtain as a corollary
    \begin{cor}
        Let $U(t_5,0,t_1)$ be the solution to the string equation \eqref{string-equation} with Stokes data \eqref{STOKES_TRUNCATED} and $t_2 = 0$, 
        and $x\in \RR$ not a pole of the corresponding Painlev\'{e} I transcendent. Then, 
            \begin{equation}
                \lim_{T\to +\infty} \frac{1}{2}T^{2/5}U(-6T/5,0,T^{1/5}x) = q(x),
            \end{equation}
        where $q(x)$ is a solution to Painlev\'{e} I.
    \end{cor}

\begin{remark}
    We believe the solution to Painlev\'{e} I appearing in Theorem \ref{MainTheorem3} is indeed the same tritronqu\'{e}e solution that appears in the previous theorems. However, in order to prove this, one must study monodromy map from the space of Stokes parameters for the $2\times 2$ Riemann-Hilbert problem for Painlev\'{e} I to the $3\times 3$ problem. This is in essence straightforward, and has been performed for similar systems (for Painlev\'{e} II, this type of problem is studied in \cite{LW}), and involves a classical steepest descent analysis. However, this calculation is rather involved, and so we do not study it in this work.
\end{remark}

\subsection{Notations.}
Throughout, we shall make use of the following notations without further comment:
\begin{itemize}
    \item $\omega := e^{\frac{2\pi i}{3}} = -\frac{1}{2}+ i\frac{\sqrt{3}}{2}$ denotes the principal third root of unity,
    \item Unless otherwise specified, for $p\geq 2$ the root $\lambda^{1/p}$ denotes the principal branch, i.e. the branch which is positive 
    for $\lambda>0$, and with $-\pi < \arg \lambda <\pi$,
    \item $E_{ij}$ denotes the elementary matrix of size $3\times 3$: $(E_{ij})_{k\ell} = \delta_{ik}\delta_{j\ell}$,
    \item We let
        \begin{equation*}
            \sigma_1 :=
            \begin{psmallmatrix}
                0 & 1\\
                1 & 0
            \end{psmallmatrix},\qquad 
            \sigma_2 :=
            \begin{psmallmatrix}
                0 & -i\\
                i & 0
            \end{psmallmatrix},\qquad
            \sigma_3 :=
            \begin{psmallmatrix}
                1 & 0\\
                0 & -1
            \end{psmallmatrix},
        \end{equation*}
        denote the usual Pauli matrices.
        \item For given diagonal matrix $A=\text{diag }(a_1,a_2,a_3)$, we set
            \begin{equation*}
                \lambda^{A} = 
                \begin{psmallmatrix}
                    \lambda^{a_1} & 0 & 0\\
                    0 & \lambda^{a_2} & 0\\
                    0 & 0 & \lambda^{a_3}
                \end{psmallmatrix}.
            \end{equation*}
    \item If $A$ is an $n\times n$ matrix, and $B$ is an $m\times m$ matrix, we denote by $A\oplus B$ the $(n+m)\times (n+m)$ matrix
        \begin{equation*}
            A\oplus B = 
            \begin{pmatrix}
                A & 0\\
                0 & B
            \end{pmatrix},
        \end{equation*}
    where $0$ denotes a rectangular matrix comprised of all zeros.
\end{itemize}

\subsection{Acknowledgments.}
This research was partially supported by the European Research Council (ERC), Grant Agreement No. 101002013. We would like to thank Great Bay University for their warm hospitality, where this work was initiated. 
The author also acknowledges support from the Royal Swedish Academy of Sciences and special program `Random Matrices and Scaling Limits' hosted at the Mittag–Leffler Institute, where part of this work was completed. We would also like to thank the anonymous referees, whose valuable comments improved the overall presentation of this work.

\section{Preliminary Transformations.}
An issue that arises in the asymptotic analysis of the RHP \ref{prob:MAIN-RHP} is that its solution cannot be determined uniquely without reference
to a certain number of subleading terms in its asymptotic expansion. This causes problems in that these subleading terms involve functions we are
trying to find asymptotics for in the first place, and so we are seemingly stuck in a redundant loop. However, a certain property of the
$\tau$-function \ref{tau-function-definition} allows us to circumvent this issue.

A key property of the above modified $\tau$-function is that it is \textit{gauge-invariant}: we can multiply $\Psi(\zeta;{\bf t})$ on the left
by any upper-triangular matrix-valued function $\mathfrak{h}({\bf t})$ with $1's$ on the diagonal, and the $\tau$-function for this RHP and the one for $\Psi(\zeta;{\bf t})$ are identical.
In other words, if we let $\tau_{\mathfrak{h}}({\bf t})$ denote the expression where $\mathfrak{G}(\zeta;{\bf t})$ is replaced by 
$\mathfrak{h}({\bf t})\mathfrak{G}(\zeta;{\bf t})$ in Equation \eqref{tau-function-definition}, then we have the equality
    \begin{equation*}
        \tau_{\mathfrak{h}}({\bf t}) = \tau({\bf t}).
    \end{equation*}
Define the gauge matrix
    \begin{equation*}
        \mathfrak{h}(t) := 
        \begin{psmallmatrix}
            1 & \frac{1}{2}H_1 & \frac{1}{4}H_2 + \frac{1}{8}H_1^2\\
            0 & 1 & \frac{1}{2}H_1\\
            0 & 0 & 1
        \end{psmallmatrix},
    \end{equation*}
and let 
    \begin{equation}
        \hat{\Psi}(\zeta;{\bf t}) := \mathfrak{h}(t)\Psi(\zeta;{\bf t}).
    \end{equation}
Then $\hat{\Psi}(\zeta;{\bf t})$ has the same jumps as $\Psi(\zeta;{\bf t})$, and at infinity behaves as
    \begin{equation*}
        \hat{\Psi}(\zeta;{\bf t}) =  \mathfrak{h}(t) f(\zeta)\left[\mathbb{I} + \OO(\zeta^{-1/3})\right]f^{-1}(\zeta)f(\zeta)e^{\Theta(\zeta)} = \left[\mathbb{I} + \OO(\zeta^{-1})\right]f(\zeta)e^{\Theta(\zeta)},
    \end{equation*}
where we have used the symmetry \eqref{psi-symmetry} to deduce that $f(\zeta)\left[\mathbb{I} + \OO(\zeta^{-1/3})\right]f^{-1}(\zeta) = \mathfrak{h}^{-1}({\bf t}) + \OO(\zeta^{-1})$ admits a regular expansion at infinity.
In other words, $\hat{\Psi}(\zeta;{\bf t})$ satisfies the RHP
    \begin{equation}
            \begin{cases}
                \hat{\Psi}_{+}(\zeta;t_5,t_2,t_1) =  \hat{\Psi}_{-}(\zeta;t_5,t_2,t_1) S_k, & \zeta \in \Gamma_k,\qquad k = \pm 1, ...,\pm 7,\\
                \hat{\Psi}_{+}(\zeta;t_5,t_2,t_1) =  \hat{\Psi}_{-}(\zeta;t_5,t_2,t_1) \mathcal{S}, & \zeta \in \RR_-,\\
                \hat{\Psi}(\zeta;t_5,t_2,t_1) = \left[\mathbb{I}+ \OO(\zeta^{-1})\right]f(\zeta)e^{\Theta(\zeta;t_5,t_2,t_1)}, & \zeta \to \infty,
            \end{cases}
        \end{equation}

The benefit of this choice of gauge is that the Riemann-Hilbert problem for
$\hat{\Psi}(\zeta;{\bf t})$ admits a unique solution \textit{without the need to specify any subleading terms}. For this reason, we shall
mainly study this characterization of the Riemann-Hilbert problem here; of course, there is no loss of generality, by our observation of 
gauge invariance. Furthermore, since $\mathfrak{h}({\bf t})$ is invertible, any statement we make about $\hat{\Psi}(\zeta;({\bf t}))$ can be readily transferred 
to a statement about $\Psi(\zeta;{\bf t})$.

We now make a few preliminary transformations to the function 
$\hat{\Psi}(\zeta;{\bf t})$ which introduce a scale parameter, and also bring the problem in to a form which is `cleaner' for exposition's sake. The former transformation is important, whereas the latter are cosmetic, and can be thought of as purely for ease of presentation.

To this end, we introduce the scaling parameter $\hbar > 0$, which will eventually act as a small parameter. Make the change of variables
    \begin{equation}\label{parameter-rescaling}
        \zeta = \hbar^{-3/7}\lambda,\qquad t_{5} = \hbar^{-2/7}\eta, \qquad t_2 = \hbar^{-5/7}\mu, \qquad t_1 = \hbar^{-6/7}\nu,
    \end{equation}
and define a `rescaled' $\hat{\Psi}$-function by
    \begin{equation}
        \boldsymbol{\Psi}(\lambda;\eta,\mu,\nu|\hbar) := \hbar^{\frac{1}{7}\hat{\sigma}}\hat{\Psi}(\hbar^{-3/7}\lambda;\hbar^{-2/7}\eta,\hbar^{-5/7}\mu,\hbar^{-6/7}\nu)
    \end{equation}
where $\hat{\sigma} := \text{diag }(1,0,-1)$. $\boldsymbol{\Psi}$ then satisfies its own RHP with
    \begin{equation}
        \begin{cases}
            \boldsymbol{\Psi}_{+}(\lambda;\eta,\mu,\nu|\hbar) = \boldsymbol{\Psi}_{-}(\lambda;\eta,\mu,\nu|\hbar)J_{\Psi}(\lambda), & \lambda \in \Gamma_{\Psi},\\
            \boldsymbol{\Psi}(\lambda;\eta,\mu,\nu|\hbar) =\left[\mathbb{I} + \OO(\lambda^{-1})\right] f(\lambda)e^{\frac{1}{\hbar}\Theta(\lambda;\eta,\mu,\nu)}, & \lambda \to \infty.
        \end{cases}
    \end{equation}
Here, $J_{\Psi}(\lambda)$ denotes the jumps of the RHP for $\Psi(\zeta;;\eta,\mu,\nu|\hbar)$ \ref{prob:MAIN-RHP}, which is defined on the collection of contours $\Gamma_{\Psi}$, and evaluated on the choice of Stokes data \eqref{STOKES_TRUNCATED}:
    \begin{equation}
        \Gamma_{\Psi} := \bigcup_{k\in\{\pm 2,\pm 6, -3,5\}} \Gamma_k \cup \RR_-,
        \qquad\qquad J_{\Psi}(\lambda) := 
        \begin{cases}
            \mathbb{I} - E_{31}, & \lambda\in \Gamma_6,\\
            \mathbb{I} + E_{32}, & \lambda\in \Gamma_5,\\
            \mathbb{I} - E_{23}, & \lambda\in \Gamma_2,\\
            \mathbb{I} - E_{32}, & \lambda\in \Gamma_{-2},\\
            \mathbb{I} + E_{12}, & \lambda\in \Gamma_{-3},\\
            \mathbb{I} - E_{21}, & \lambda\in \Gamma_{-6},\\
            \mathcal{S}, & \lambda\in \RR_-.
        \end{cases}
    \end{equation}
The main effect of this transformation is the 
introduction of an overall scale parameter $\hbar$ multiplying the exponential asymptotics of the RHP.

Before proceeding to the steepest descent analysis, we perform several simple transformations, which involve multiplication on the right by piecewise constant matrices. 
We emphasize that these transformations serve primarily for ease of exposition later, and should be thought of as technical details.

We now set
    \begin{equation}
        {\bf X}(\lambda;\eta,\mu,\nu|\hbar) := \boldsymbol{\Psi}(\lambda;\eta,\mu,\nu|\hbar)\cdot
            \begin{cases}
                1\oplus \sigma_1, & \text{Im } \lambda >0,\\
                \sigma_3\oplus 1, & \text{Im } \lambda < 0,
            \end{cases}
    \end{equation}
    and then immediately set
        \begin{equation}
        {\bf Y}(\lambda;\eta,\mu,\nu | \hbar) :=
            {\bf X}(\lambda;\eta,\mu,\nu | \hbar) \cdot
            \begin{cases}
                \mathbb{I}-E_{23}, & \lambda\in \left[ \RR_+,\Gamma_2\right],\\
                \mathbb{I}-E_{23}, & \lambda\in \left[ \RR_+,\Gamma_{-2}\right],\\
                \mathbb{I} + E_{31}, & \lambda\in \left[\Gamma_{6},\RR_-\right],\\
                \mathbb{I} - E_{31}, & \lambda\in \left[\Gamma_{-6},\RR_-\right],\\
                \mathbb{I}, & \textit{otherwise}.
            \end{cases}
    \end{equation}
Here, we have introduced the following notation: suppose $L_a := (0,e^{i\varphi_a}\cdot\infty)$ and $L_b :=(0,e^{i\varphi_b}\cdot\infty)$ are rays emanating from the origin, $-\pi<\varphi_a < \varphi_b <\pi$. We denote the (acute) region enclosed by these two rays 
to be
    \begin{equation}
        [L_a,L_b] := \{\lambda\in \CC | \varphi_a<\arg \lambda < \varphi_b\}.
    \end{equation}

One then finds that the matrix ${\bf Y}$ satisfies the following RHP:
    \begin{equation}
        {\bf Y}_+(\lambda;\eta,\mu,\nu | \hbar) = {\bf Y}_-(\lambda;\eta,\mu,\nu | \hbar) \cdot 
            \begin{cases}
                \mathbb{I} -E_{23}, & \lambda \in \RR_+,\\
                \mathbb{I} +E_{23}, & \lambda \in \Gamma_5,\\
                \mathbb{I} +E_{12}, & \lambda \in \Gamma_{-3},\\
                \mathbb{I} -E_{12}, & \lambda \in \RR_-,
            \end{cases}
    \end{equation}
with normalization
    \begin{equation}
         {\bf Y}(\lambda;\eta,\mu,\nu | \hbar) = \left[\mathbb{I} + \OO(\lambda^{-1}) \right]\hat{f}(\lambda)e^{\frac{1}{\hbar}\hat{\Theta}(\lambda;\eta,\mu,\nu)},\qquad \lambda\to \infty.
    \end{equation}
where
    \begin{equation}
        \hat{f}(\lambda) := 
        f(\lambda)\cdot
        \begin{cases}
            1\oplus\sigma_1, & \text{Im } \lambda >0,\\
             \sigma_3\oplus 1, & \text{Im } \lambda <0,
        \end{cases}
    \end{equation}
and
    \begin{equation}
    \hat{\Theta}(\lambda) = \hat{\Theta}(\lambda;t_5,t_2,t_1) = 
        \begin{cases}
            \text{diag }(\vartheta_1(\lambda),\vartheta_3(\lambda),\vartheta_2(\lambda)), & \text{Im }\lambda>0,\\
            \text{diag }(\vartheta_1(\lambda),\vartheta_2(\lambda),\vartheta_3(\lambda)), & \text{Im }\lambda<0.
        \end{cases}
    \end{equation}

We need one final `trivial' transformation. Let $\alpha>0, \beta<\alpha$ be real numbers (these will eventually correspond to the branch points appearing in the spectral curve), and let:
\begin{enumerate}
    \item $\hat{\Gamma}_5$ be a contour emanating from $\lambda =\alpha$ and going to $\infty$ in the same direction as $\Gamma_5$,
     \item $\hat{\Gamma}_{-3}$ be a contour emanating from $\lambda =\beta$ and going to $\infty$ in the same direction as $\Gamma_3$.
\end{enumerate}
Finally, let $\Delta_{\alpha}$ be the triangular region enclosed by the contours $\Gamma_{5},\hat{\Gamma}_5,[0,\alpha]$, and $\Delta_{\beta}$ be the triangular region enclosed by the contours $\Gamma_{-3},\hat{\Gamma}_{-3},[\beta,0]$\footnote{It is also possible that $\beta>0$. In this case one should take $[0,\beta]$ as the base of this triangular region; this does not change any further calculations in an essential way.}. These contours and regions are depicted in Figure \ref{fig:PreliminaryTransformations} (b). We set
    \begin{equation}
        {\bf Z}(\lambda;\eta,\mu,\nu|\hbar) := 
        {\bf Y}(\lambda;\eta,\mu,\nu|\hbar) \cdot 
        \begin{cases}
            \mathbb{I} +E_{23}, & \lambda \in \Delta_{\alpha},\\
            \mathbb{I} +E_{12}, & \lambda \in \Delta_{\beta},\\
            \mathbb{I}, & \textit{otherwise.}
        \end{cases}
    \end{equation}
\begin{figure}[t!]
    \centering
    \begin{subfigure}[t]{0.33\textwidth}
        \centering
        \begin{overpic}[scale=.25]{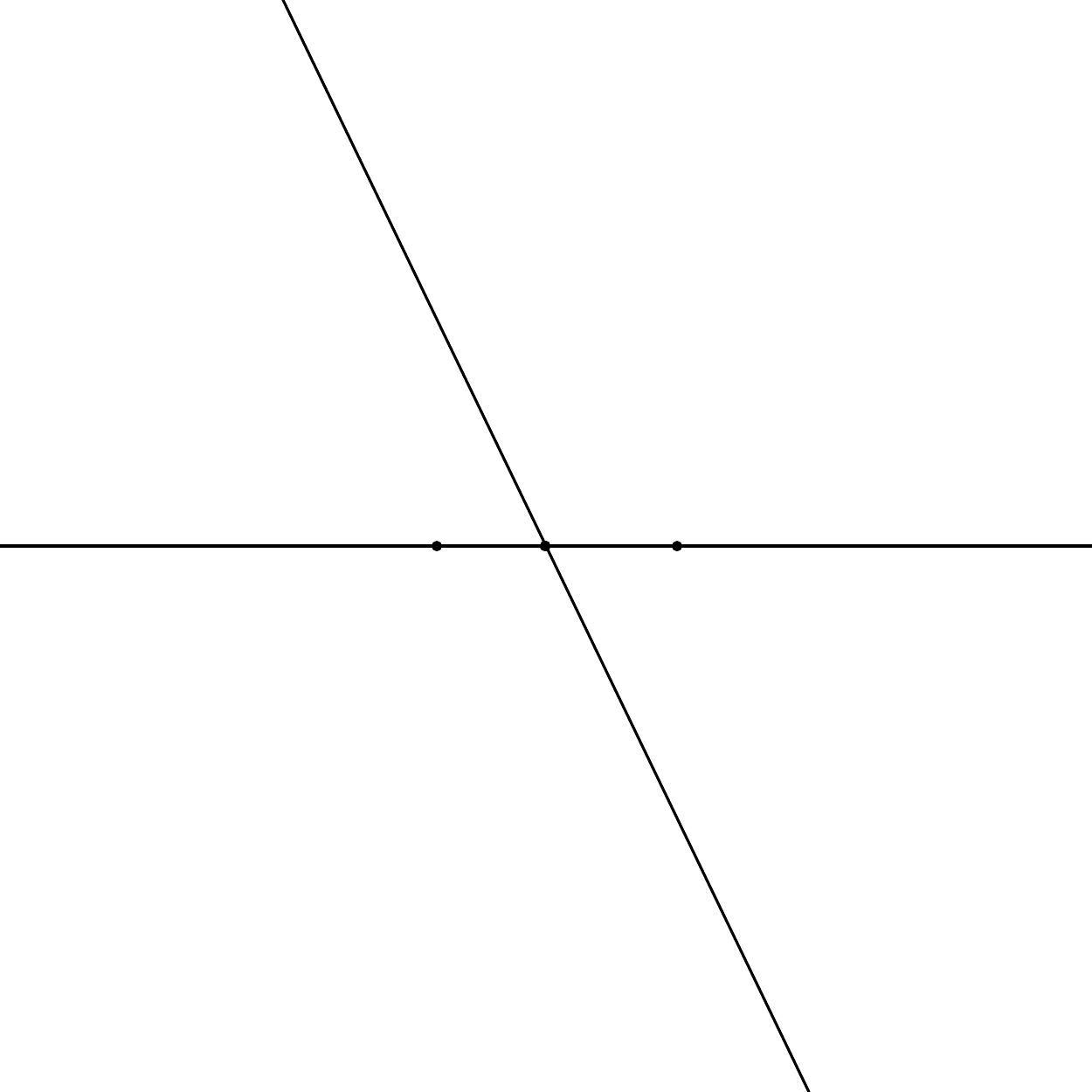}
            \put (66,51) {\footnotesize	 \textcolor{BrickRed}{$\mathbb{I} -E_{23}$}}
            \put (16,51) {\footnotesize	 \textcolor{BrickRed}{$\mathbb{I} -E_{12}$}}
            \put (37,80) {\footnotesize \rotatebox{-60.5}{\textcolor{BrickRed}{$\mathbb{I} + E_{23}$}}}
            \put (21,94) {\footnotesize $\Gamma_{5}$}
            \put (60,32) {\footnotesize \rotatebox{-60.5}{\textcolor{BrickRed}{$\mathbb{I} + E_{12}$}}}
            \put (76,2) {\footnotesize $\Gamma_{-3}$}
            \put (37,44) {\footnotesize $\beta$}
            \put (60,44) {\footnotesize $\alpha$}
        \end{overpic}
        \caption{The jumps of ${\bf Y}$.}
    \end{subfigure}%
    \begin{subfigure}[t]{0.33\textwidth}
        \centering
        \begin{overpic}[scale=.25]{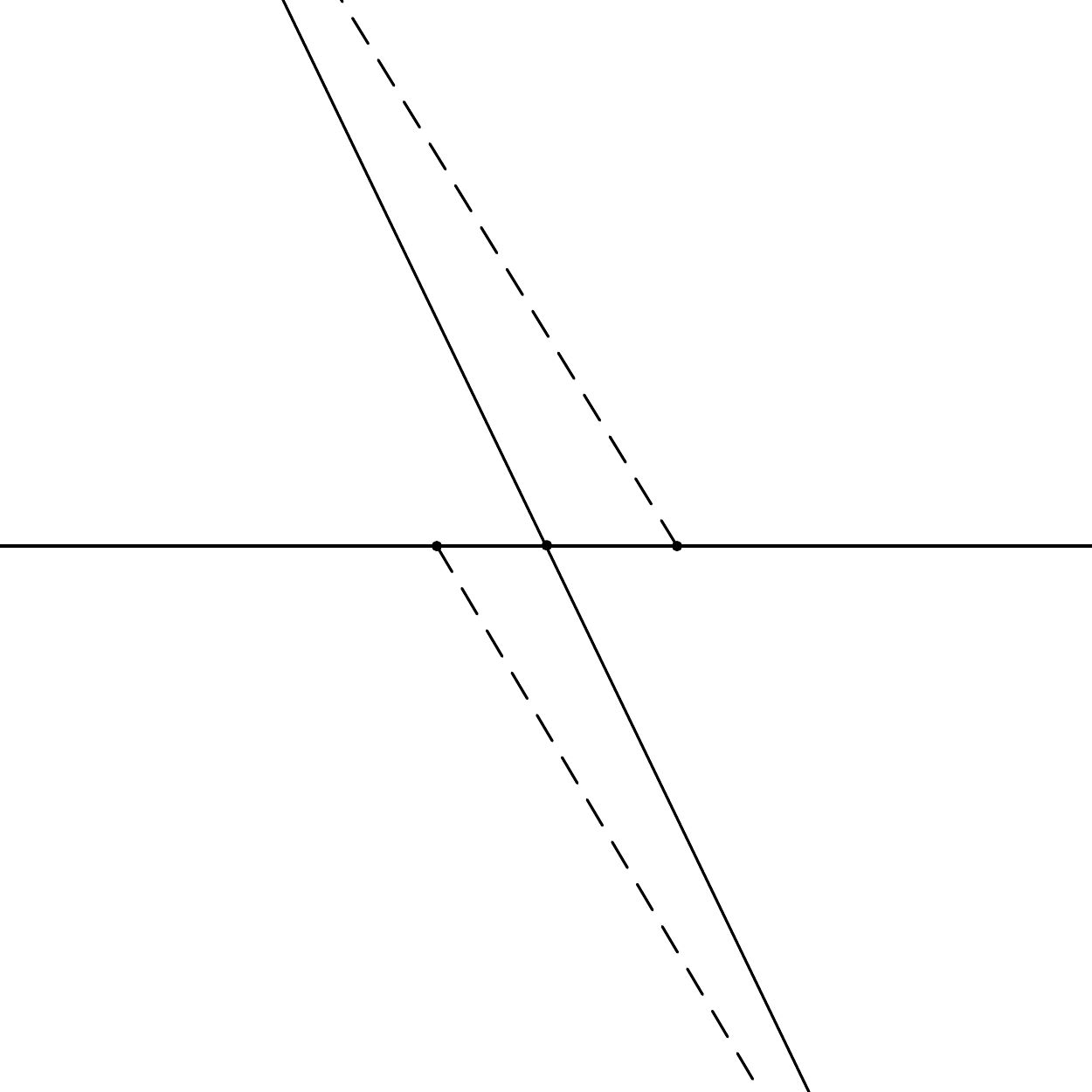}
            \put (21,94) {\footnotesize $\Gamma_{5}$}
            \put (76,2) {\footnotesize $\Gamma_{-3}$}
            \put (37,94) {\footnotesize $\hat{\Gamma}_{5}$}
            \put (56,2) {\footnotesize $\hat{\Gamma}_{-3}$}
            \put (37,44) {\footnotesize $\beta$}
            \put (60,44) {\footnotesize $\alpha$}
            \put (48,56) {\footnotesize $\textcolor{blue}{\Delta_{\alpha}}$}
            \put (45,42) {\footnotesize $\textcolor{blue}{\Delta_{\beta}}$}
        \end{overpic}
        \caption{Definition of domains $\Delta_{\alpha}, \Delta_{\beta}$.}
    \end{subfigure}%
    \begin{subfigure}[t]{0.33\textwidth}
        \centering
        \begin{overpic}[scale=.25]{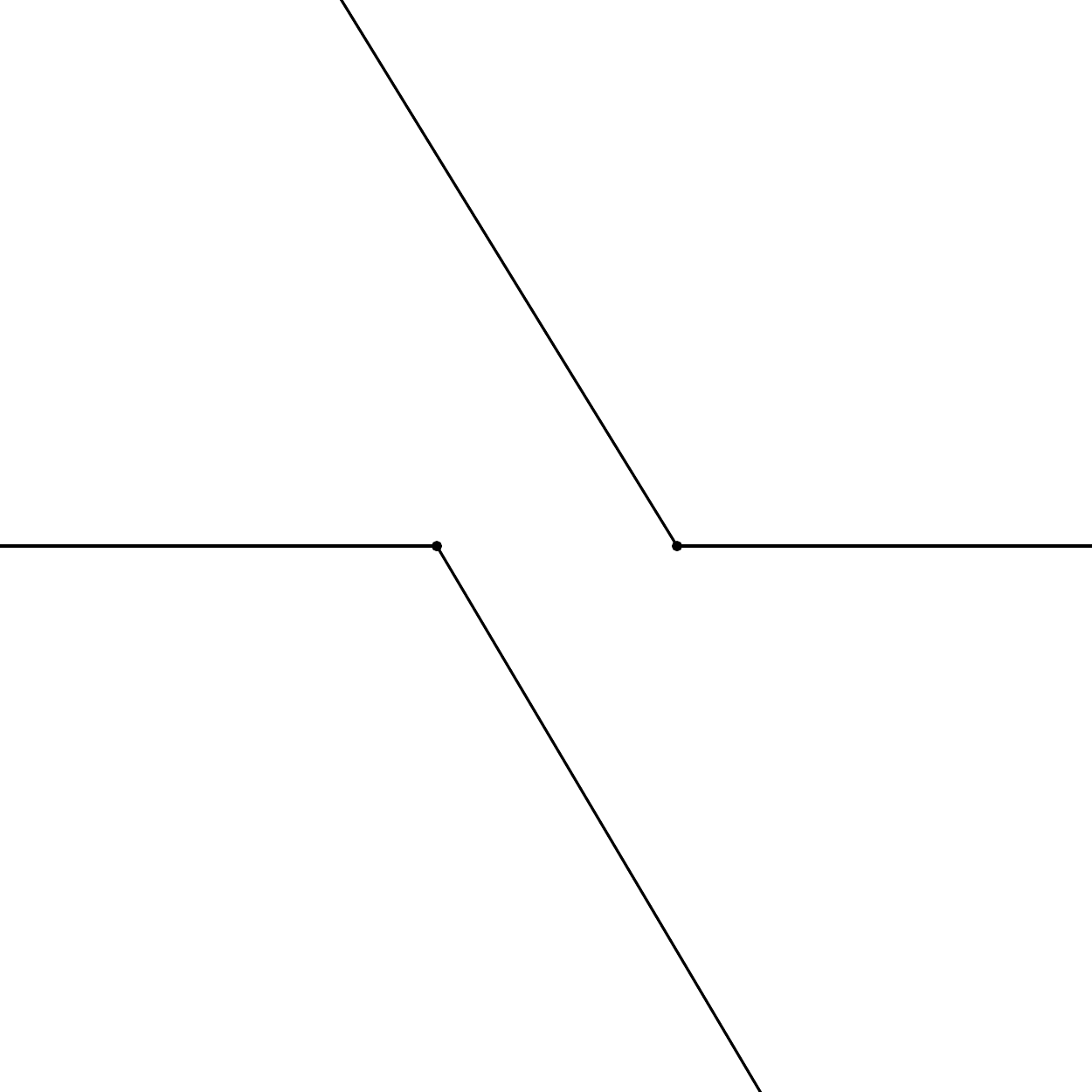}
            \put (66,51) {\footnotesize	 \textcolor{BrickRed}{$\mathbb{I} -E_{23}$}}
            \put (16,51) {\footnotesize	 \textcolor{BrickRed}{$\mathbb{I} -E_{12}$}}
            \put (37,94) {\footnotesize $\hat{\Gamma}_{5}$}
            \put (46,80) {\footnotesize \rotatebox{-58}{\textcolor{BrickRed}{$\mathbb{I} + E_{23}$}}}
            \put (56,2) {\footnotesize $\hat{\Gamma}_{-3}$}
            \put (53,32) {\footnotesize \rotatebox{-58}{\textcolor{BrickRed}{$\mathbb{I} + E_{12}$}}}
            \put (37,44) {\footnotesize $\beta$}
            \put (60,44) {\footnotesize $\alpha$}
        \end{overpic}
        \caption{The jumps of ${\bf Z}$.}
    \end{subfigure}%
    \caption{(a) The jumps of ${\bf Y}(\lambda;\eta,\mu,\nu | \hbar)$. (b) The new contours $\hat{\Gamma}_{5}$ (resp. $\hat{\Gamma}_{-3}$), and the new regions $\Delta_{\alpha}$ (resp. $\Delta_{\beta}$), which are enclosed by 
    and the $\Gamma_5,\hat{\Gamma}_5$, and the real axis (resp. $\Gamma_{-3},\hat{\Gamma}_{-3}$, and the real axis). 
    (c) The jumps of ${\bf Z}(\lambda;\eta,\mu,\nu | \hbar)$. In all figures, rays are oriented \textbf{\textit{outwards}}.}
    \label{fig:PreliminaryTransformations}
\end{figure}
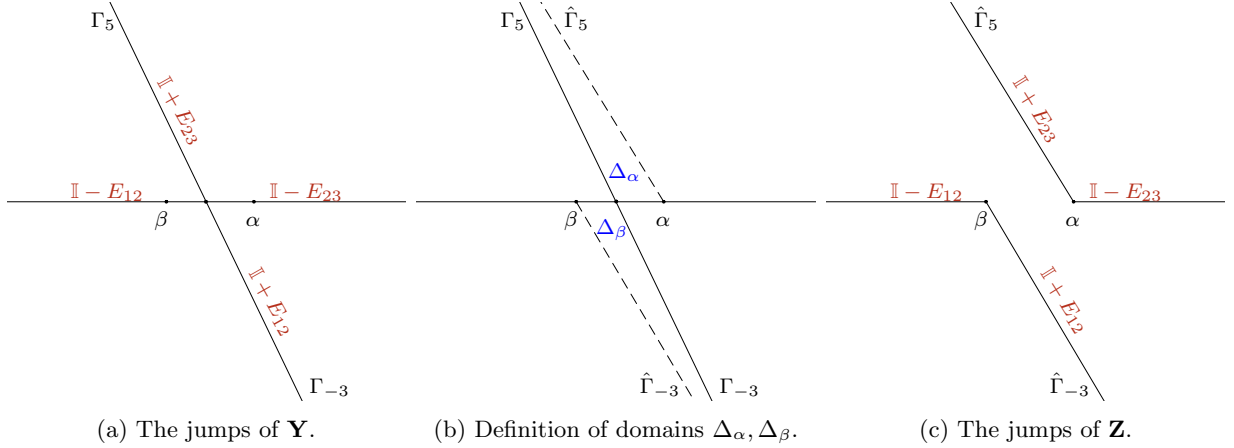
The result is that ${\bf Z}$ has the same asymptotics as ${\bf Y}$, and satisfies a modified jump condition:
    \begin{align}
                {\bf Z}_+(\lambda;\eta,\mu,\nu | \hbar) &= {\bf Z}_-(\lambda;\eta,\mu,\nu | \hbar) \cdot 
            \begin{cases}
                \mathbb{I} -E_{23}, & \lambda \in [\alpha,\infty),\\
                \mathbb{I} +E_{23}, & \lambda \in \hat{\Gamma}_5,\\
                \mathbb{I} +E_{12}, & \lambda \in \hat{\Gamma}_{-3},\\
                \mathbb{I} -E_{12}, & \lambda \in (-\infty,\beta],
            \end{cases}\\
            {\bf Z}(\lambda;\eta,\mu,\nu | \hbar)&= \left[\mathbb{I} + \OO(\lambda^{-1}) \right]\hat{f}(\lambda)e^{\frac{1}{\hbar}\hat{\Theta}(\lambda;\eta,\mu,\nu)},\qquad \lambda\to \infty.
    \end{align}
The jumps of ${\bf Z}$ are depicted in Figure \ref{fig:PreliminaryTransformations} (c).
It is the piecewise analytic function ${\bf Z}(\lambda;\eta,\mu,\nu | \hbar)$ that we will perform a Deift-Zhou analysis of in the below. Since the
transformations of this subsection are all invertible, any result we derive for ${\bf Z}(\lambda;\eta,\mu,\nu | \hbar)$ automatically
applies to our original RHP for $\Psi$.

\section{Existence of solution for sufficiently large values of the parameters.} \label{genus-zero-existence}
In this section, we prove Theorems \ref{MainTheorem1} and \ref{Theorem:topologicalexpansion}.

Our goal is now to perform a Deift-Zhou steepest descent analysis for the RHP ${\bf Z}(\lambda;\eta,\mu,\nu|\hbar)$, with $\hbar$ playing the role of the small parameter. This analysis is by now fairly standard. The first relevant transformation
will be the one which removes the exponential part of the asymptotics; this is typically called the `$g$-function' transformation. We will begin with a discussion of the construction of the $g$-function, and then proceed with the Deift-Zhou analysis.

\subsection{Spectral curve and construction of the $g$-function.}\label{section:g-function}
    We must construct a function $g(\lambda)$, which is defined on a $3$-sheeted Riemann surface, and whose restriction $g_j(\lambda)$ to sheet $j$
    satisfies
        \begin{equation*}
            g_j(\lambda) = \hat{\Theta}_{jj}(\lambda;\eta,\mu,\nu) + \OO(\lambda^{-1/3}) , \qquad \lambda\to \infty,
        \end{equation*}
    and the $g_j(\lambda)$ satisfy certain inequalities on the branch cuts, which we will describe in more detail below.
    The fact that the Riemann surface is $3$-sheeted complicates matters. At the present time, the authors do not know an efficient way to treat this problem in general, and so we settle on searching for solutions to the above problem under the assumption that \textit{the Riemann surface of $g(\lambda)$ is of genus $0$}. In this case, this surface can be resolved in terms of rational functions, and we can solve the above problem concretely, albeit with some restrictions on the values of the parameters $\eta,\mu,\nu$. We have the following proposition.
        \begin{prop}\label{prop:true-g-function}
            Let $a,b,c$ be real parameters. Define polynomials
                \begin{equation}
                    \lambda(u) := u^3-3a^2u+c, \qquad\qquad Y(u) := u^4 - bu^2 + \frac{4}{3}cu - 6a^4+2a^2b,
                \end{equation}
            and put
                \begin{equation}\label{etamunu-abc}
                    \eta := \frac{3}{5}( 4a^2 - b ),\qquad \mu := c(b-2a^2), \qquad \nu := 8a^6 - 3a^4b - \frac{2}{3}c^2,
                \end{equation}
            and
                \begin{align}
                    g(u) &:= \frac{3}{7}u^7 - \frac{3}{5}(b+a^2)u^5 + cu^4 -a^2(6a^2+b)u^3 -2a^2cu^2 - 3a^2(-6a^4-2a^2b)u -2ca^4\\
                        &= \int Y(u)\lambda'(u) du -2ca^4.\nonumber
                \end{align}
            Let $u^{*}(\lambda)$ be a solution to the equation $\lambda(u) = \lambda$ which satisfies $u^{*}(\lambda) = \omega^{j-1}\lambda^{1/3} + \OO(\lambda^{-1/3})$, $\lambda\to \infty$. Then the function $g^{*}(\lambda) := g(u^{*}(\lambda))$
            formally satisfy the asymptotic condition
                \begin{equation}
                    g^{*}(\lambda) = \hat{\Theta}_{jj}(\lambda;\eta,\mu,\nu) + \OO(\lambda^{-1/3}) , \qquad \lambda\to \infty.
                \end{equation}
        \end{prop}
        \begin{proof}
            The proof of the above proposition is a direct calculation, and so we omit it. The only relevant comment to be made is that one must
            calculate $u^{*}(\lambda)$ to order $\lambda^{-7/3}$ in order to check the validity of the proposition.
        \end{proof}
        \begin{remark}
            The coordinates $(a,b,c)$ will be convenient for us in proving that the function $g$ satisfies certain inequalities necessary
            for our analysis. These parameters relate to the parameters $\eta,\mu,\nu$ through the formulae \eqref{etamunu-abc}, and the 
            parameter $\varsigma$ is then
                \begin{equation}
                    \varsigma = 2a^2.
                \end{equation}
            Indeed, it is straightforward to check that, for $(\eta,\mu,\nu)$ given by \eqref{etamunu-abc}, a solution to 
            Equation \eqref{sigma-eq} is given by the above choice of $\varsigma$. 

            Later in this section (see Proposition \ref{prop:bijection}), we shall show that the region $D$ defined in 
            \ref{Domain-D-Definition} is the image of the domain in the $(a,b,c)$-space that we will be working with, and so outside of this
            section, we shall not refer to the coordinates $(a,b,c)$, and only to their counterparts in the region $D$.

            We end this remark by providing a representation of the uniformization coordinates in terms of the 
            parameters $(\eta,\mu,\nu)$ and $\varsigma$. The spectral curve we have defined is parameterized by
                \begin{equation}
                    \lambda(u) = u^3 -\frac{3}{2}\varsigma u - \frac{3\mu}{5\eta-3\varsigma},\qquad\qquad Y(u)=u^4 + \left(\frac{5}{3}\eta -2\varsigma\right)u^2 -\frac{4\mu}{5\eta-3\varsigma}u + \frac{1}{2}\varsigma^2-\frac{5}{3}\eta\varsigma.
                \end{equation}
        \end{remark}
        The previous proposition gives us a candidate $g$-function for our later steepest descent analysis. However, we require more of the function
        $g(\lambda)$: we need that certain inequalities hold on the branch cuts of the associated Riemann surface where $g$ lives. 

        We now carefully define the Riemann surface $\mathcal{R}$ of $g$, as a branched covering over the $\lambda$-coordinate. The three sheets $\mathcal{R} := \mathcal{R}_1 \sqcup \mathcal{R}_2 \sqcup \mathcal{R}_3$ of this surface are defined as follows (here, $\alpha =\lambda(-a), \beta = \lambda(a)$):
        \begin{equation}
        \mathcal{R}_1 := \CC\setminus (-\infty,\beta],\qquad \mathcal{R}_2 := \CC\setminus \left((-\infty,\beta]\cup [\alpha,\infty)\right)\qquad \mathcal{R}_3 := \CC\setminus [\alpha,\infty),
    \end{equation}
    with sheets $1$ and $2$ glued along $(-\infty,\beta]$, and sheets $2$ and $3$ glued along $[\alpha,\infty)$. Note that, for any
    values of $(\eta,\mu,\nu) \in D$, we have the inequalities
        \begin{equation*}
            \alpha >0, \qquad \alpha > \beta.
        \end{equation*}
    The sign of $\beta$ is variable, but is ultimately inessential. The spectral curve in the spectral ($\lambda$) and uniformizing ($u$) planes are shown in Figure \ref{fig:SpectralCurve} (a) and (b), respectively.
    On each sheet, we can define a uniformizing coordinate $u_j(\lambda)$, which are uniquely determined by the property that $\lambda(u_j(\lambda)) = \text{id}_{\mathcal{R}_j}$, for $\lambda\in \mathcal{R}_j$, and their asymptotic behavior at infinity on each sheet, which are given below:
    \begin{align}
        u_1(\lambda) &= \lambda^{1/3}[1 + \OO(\lambda^{-1/3})], \qquad \lambda\to \infty,\label{u1-asymptotics}\\
        u_2(\lambda) &= 
        \begin{cases}
            \omega^2 \lambda^{1/3}[1 + \OO(\lambda^{-1/3})], & \lambda\to \infty,\quad \text{Im } \lambda>0,\\
            \omega \lambda^{1/3}[1 + \OO(\lambda^{-1/3})], & \lambda\to \infty,\quad \text{Im } \lambda<0,
        \end{cases}\label{u2-asymptotics}\\
        u_3(\lambda) &=
        \begin{cases}
            \omega \lambda^{1/3}[1 + \OO(\lambda^{-1/3})], & \lambda\to \infty,\quad \text{Im } \lambda>0,\\
            \omega^2 \lambda^{1/3}[1 + \OO(\lambda^{-1/3})], & \lambda\to \infty,\quad \text{Im } \lambda<0.
        \end{cases}\label{u3-asymptotics}
    \end{align}
    We further set 
        \begin{equation}
            g_j(\lambda) := g(u_j(\lambda)).
        \end{equation}
    In order to guarantee that we will eventually be able to open lenses, it is necessary that the functions $g_j(\lambda)$ satisfy certain inequalities. This will place certain restrictions of the range of the parameters $\eta,\mu,\nu$. Unless one can find a way to circumvent the genus $0$ ansatz, this restriction for now remains insurmountable, and so our existence results will hold only for values of these parameters in the given range.
    \begin{defn}
        We define the region $R$ by
        \begin{equation}
                R:= \left\{ (a,b,c)\in \RR^3 \big| 0 < b < \infty, \qquad 0 \leq c < \frac{b^{3/2}}{\sqrt{6}}, \qquad z_0(b,c) < a < z_+(b,c)\right\},
            \end{equation}
            where $z_-(b,c) < z_0(b,c) < z_+(b,c)$ are the three real solutions to the equation $z^3-\frac{1}{2}bz+\frac{1}{3}c = 0$.
    \end{defn}
    \begin{remark}
        Since this cubic has $3$ real solutions, and so we may use the classical trigonometric Vi\`{e}te formulae to represent these roots:
            \begin{equation*}
                z_0 = \sqrt{\frac{2b}{3}}\sin\left(\frac{1}{3}\arcsin\left(\frac{c\sqrt{6}}{b^{3/2}}\right)\right), \qquad 
                z_{\pm} = \sqrt{\frac{2b}{3}}\sin\left(\frac{1}{3}\arcsin\left(\frac{c\sqrt{6}}{b^{3/2}}\right) \pm \frac{2\pi}{3}\right).
            \end{equation*}
        The above formulae allow us to readily derive a number of inequalities relating to these roots. For instance, we can see immediately that
        for fixed $b>0$, $z_0$ is monotone increasing in $c$, and $z_{\pm}$ are monotone decreasing functions of $c$.
        It is also immediate that $z_- < 0 < z_0 < z_+$, for $0 < b < \infty, 0 < c < \frac{b^{3/2}}{\sqrt{6}}$.
    \end{remark}

            \begin{prop}
            For $(a,b,c) \in R$,
                \begin{equation}
                    \mu(a,b,c) >0,
                \end{equation}
            and furthermore
                \begin{equation}
                    \qquad\qquad \varsigma(\eta,\mu,\nu) >\max\left\{\frac{5}{3}\eta,0\right\}.
                \end{equation}
        \end{prop}
        \begin{proof}
            Note that, in the coordinates $(a,b,c)$, the sign of $\mu(a,b,c)$ is the same as the sign of $b-2a^2$. Similarly,
                \begin{equation*}
                    \varsigma(\eta,\mu,\nu) - \frac{5}{3}\eta = 2a^2 - \frac{5}{3}\eta(a,b,c) = b-2a^2,
                \end{equation*}
            and since $\varsigma = 2a^2 > z_+>0$, it is enough to check that $b-2a^2>0$ for $(a,b,c) \in R$. Indeed, since $a,z_+>0$, and $a<z_+(b,c)$, and $z_+(b,c) < z_+(b,0) = \sqrt{b/2}$,
                \begin{align*}
                    b-2a^2 > b-2z_+^2(b,c) > b-2z_+(b,0)^2 = b-2(\sqrt{b/2})^2 = b-b = 0.
                \end{align*}
        \end{proof}
        \begin{remark}
            In fact, we can also study the above equation for $\mu <0$ purely by symmetry. To see this, we remark that
            for $\mu < 0$ we take the parameter range $\tilde{R}$ to be
                \begin{equation}
                    \tilde{R} := \left\{(a,b,c)\in \RR^3 | 0<b<\infty, \qquad -b^{3/2} < c \leq 0, \qquad z_-(b,c) < a < z_0(b,c)\right\}.
                \end{equation}
            Then all of the above propositions hold by the same argumentation used in this section, by noticing that 
            $z_-(b,c) = -z_+(b,-c)$. 
            Thus, the $\mu <0$ case is a trivial corollary of the $\mu>0$ case. Figure \eqref{fig:43-critical surface} shows the
            critical surface for both $\mu>0$ and $\mu <0$. For ease of exposition, in what follows we will continue only to treat the 
            $\mu \geq 0$ case. Note that $\varsigma(\eta,\mu,\nu) = \varsigma(\eta,-\mu,\nu)$.
        \end{remark}

        We now show that the map from the region $R\cup \tilde{R}$ to the region $D$ defined in the introduction is a bijection.
        \begin{prop}\label{prop:bijection}
            Define the map $\Pi: R\cup \tilde{R} \to \mathbb{R}^3$ by
                \begin{equation}
                    \Pi(a,b,c) = (\eta(a,b,c),\mu(a,b,c),\nu(a,b,c)).
                \end{equation}
            Then, this map is a homeomorphism onto the domain $D$ from Definition \ref{Domain-D-Definition}.
        \end{prop}
        \begin{proof}
            Note first that the domain $R$ indeed contains the ray $\{(\eta,0,0)|\eta>0\}$: the point $\left(\sqrt{\frac{5}{4}\eta},\frac{10}{3}\eta,0\right) \in R$ maps onto the point $(\eta,0,0) \in R$. Since the function 
            $\varsigma  = \varsigma(a,b,c) = 2a^2 = \frac{5}{2}\eta$ at this point, we are indeed in the domain $R$.
            
            On the other hand, by direct calculation, one finds that
                \begin{equation*}
                    \left|\frac{\partial (\eta,\mu,\nu)}{\partial(a,b,c)}\right| = \frac{4}{3}|a(6a^3-3ba+2c)(6a^3-3ba-2c)|.
                \end{equation*}
            Now, $a>0$ on $R$, and $6a^3-3ba+2c = 0$ only if $a\to z_{\pm},z_0$. One can readily check that $6a^3-3ba+2c <0$ on $R$; furthermore,
            since $c<0$,
                \begin{equation*}
                    6a^3-3ba-2c = 6a^3-3ba+2c - 4c <0,
                \end{equation*}
            and thus $\left|\frac{\partial (\eta,\mu,\nu)}{\partial(a,b,c)}\right|\neq 0$ on $R$; an identical calculation shows that the same is true on the domain $\tilde{R}$, and so the Jacobian of the map $\Pi$ is nonvanishing on $R\cup\tilde{R}$, and $\Pi$ is a locally 
            injective function.

            On the other hand, calculating $\frac{\partial \mathcal{P}}{\partial \varsigma}$ in the variables $(a,b,c)$, we obtain
                \begin{equation*}
                    \frac{\partial \mathcal{P}}{\partial \varsigma} = -\frac{(6a^3-3ba+2c)(6a^3-3ba-2c)}{3(2a^2-b)} \neq 0,
                \end{equation*}
            by our previous observations. It follows that on the domain $\Pi(R\cup\tilde{R})$, the discriminant of the equation 
            \eqref{sigma-eq} is non-vanishing, i.e., $\varsigma$ is a simple root of Equation \eqref{sigma-eq}. On $\partial\Pi(R\cup\tilde{R})$ (which is the same as $\Pi(\partial(R\cup\tilde{R}))$, as one can readily check), $\frac{\partial \mathcal{P}}{\partial \varsigma} \to 0$, and so $\sigma$ becomes a multiple root of equation \eqref{sigma-eq}. 
        \end{proof}
    
    We now have the following lemma, which mimics a technique for proving such inequalities developed in \cite{DHL1}:
        \begin{lemma} \label{Lemma:Lensing}
            Let $\Gamma$ denote the preimages of the branch cuts in the uniformization plane:
                \begin{equation*}
                    \Gamma:= \left\{u\in \CC | \text{Im } \lambda(u) = 0, \text{Im }u \neq 0\right\} = \left\{u=x+iy|y^2-3x^2+3a^2 = 0\right\},
                \end{equation*}
            and consider the curve 
                \begin{equation*}
                     \mathcal{C} := \left\{u\in \CC | \text{Im } Y(u) = 0, \text{Im }u \neq 0\right\} = \left\{u=x+iy|4x^3-4xy^2-2bx+\frac{4}{3}c = 0\right\}.
                \end{equation*}
            Provided that $(a,b,c)\in R$, the curves
            $\Gamma$ and $\mathcal{C}$ do not intersect:
                \begin{equation*}
                    \Gamma \cap \mathcal{C} = \emptyset.
                \end{equation*}
        \end{lemma}

        \begin{figure}
            \centering
            \begin{subfigure}[t]{0.48\textwidth}
            \centering
            \begin{overpic}[scale=.25]{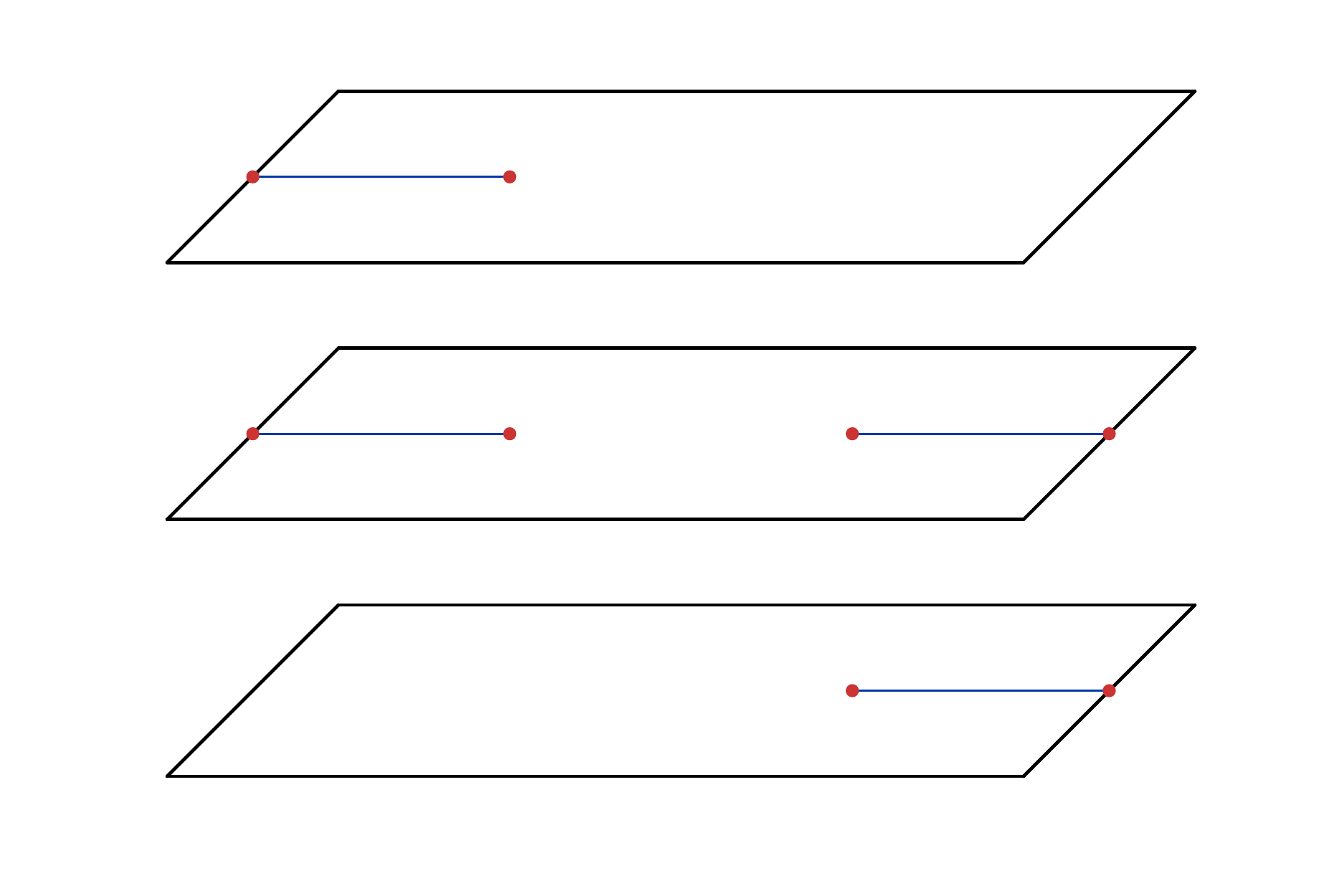}
                \put (90,53) {$\mathcal{R}_1$}
                \put (90,32) {$\mathcal{R}_2$}
                \put (90,14) {$\mathcal{R}_3$}
                \put (37,56) {$\beta$}
                \put (62,36) {$\alpha$}
            \end{overpic}
            \caption{Spectral curve in the $\lambda$-plane.}
            \end{subfigure}
            \begin{subfigure}[t]{0.48\textwidth}
            \centering
            \begin{overpic}[scale=.25]{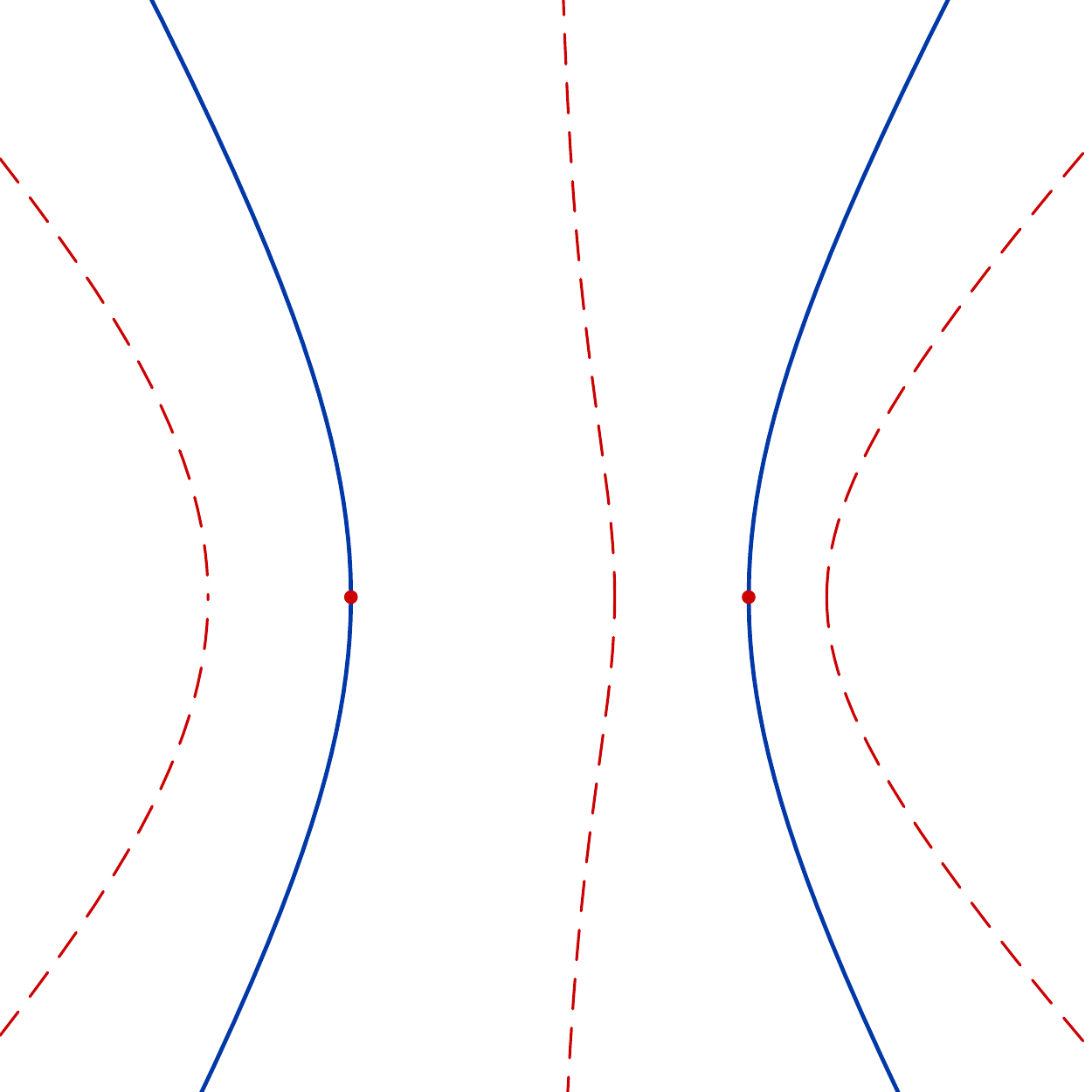}
                \put (90,48) {I}
                \put (45,80) {II}
                \put (20,48) {III}
                \put (35,42) {\footnotesize $-a$}
                \put (62,42) {\footnotesize $a$}
            \end{overpic}
            \caption{Spectral curve in the $u$-plane.}
            \end{subfigure}
            \caption{(a) The sheets of the spectral curve $\mathcal{R}_j$, $j=1,2,3$, glued along the cuts $(-\infty,\beta]$ and $[\alpha,\infty)$. (b) The preimages of the sheets $\mathcal{R}_j$ in the uniformizing plane, with $a = 0.8,b = 3.2,c = 1.2$. The preimage of sheet $\mathcal{R}_j$ is labeled by its corresponding Roman numeral. Furthermore, 
            there is the correspondence $\lambda(-a) = \alpha$, $\lambda(a) = \beta$. Dashed red lines correspond to the places where the $g$-function changes sign; we require that these curves do not intersect the branch cuts (shown in blue).}
            \label{fig:SpectralCurve}
        \end{figure}
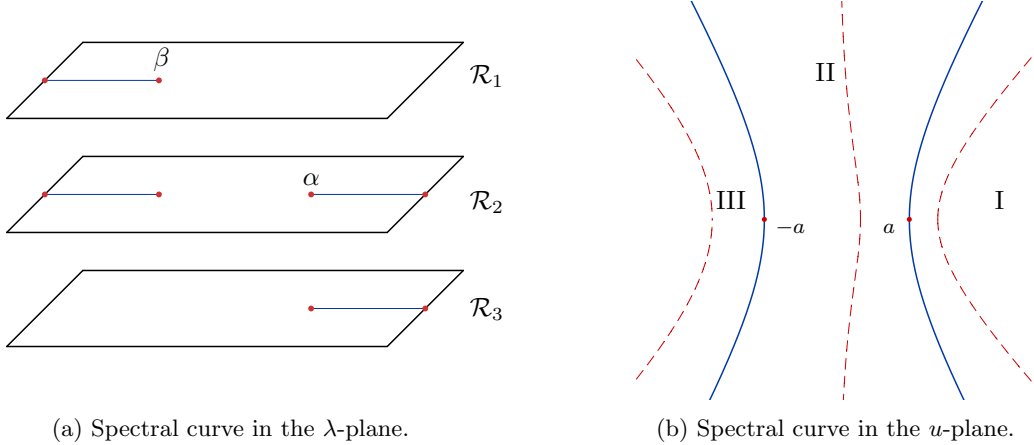

        \begin{proof}
            The fact that the polynomial $z^3-\frac{1}{2}bz+\frac{1}{3}c$ has real roots is immediate from the constraint $|c|< \frac{b^{3/2}}{\sqrt{6}}$. To see that the curves $\Gamma, \mathcal{C}$ indeed do not intersect, note that $\Gamma$ is a hyperbola with vertices at
            $u = \pm a$. The curve $\mathcal{C}$ is quadratic in $y$, and thus can be written as the graph of a pair of functions:
                \begin{equation*}
                    x+iy\in \mathcal{C} \Leftrightarrow y = y(x) = \pm\sqrt{\left[\frac{6x^3-3bx+2c}{6x}\right]_+}.
                \end{equation*}
            Note that the numerator of this expression is precisely the polynomial equation $z^3-\frac{1}{2}bz+\frac{1}{3}c$. Provided $|c|< \frac{b^{3/2}}{\sqrt{6}}$,
            these roots are distinct, and two of these roots are positive if $c>0$. Due to the symmetry $y\to -y$ of both $\Gamma$ and $\mathcal{C}$, it is enough to prove that the graph of $y_+(x)$ does not intersect $\Gamma$ in the upper half plane.

            $y_+(x)$ has three branches emanating from the three roots $z_{\pm},z_0$ (corresponding to the regions where $\frac{6x^3-3bx+2c}{6x}>0$),
            defined on the domains i. $(-\infty,z_-)$, ii. $(0,z_0)$, and iii. $(z_+,\infty)$. Let us analyze the behaviors of each of these branches.
            
            i. Consider the branch of $y_+(x)$ defined on the domain $(-\infty,z_-)$. As $x\to -\infty$, $y_+(x) = -x[1+ \OO(x^{-2}) ]$, and
            decreases monotonically to $0$ as $x\to z_-$ from the left. 
            
            ii. Let us now examine the branch of $y_+(x)$ defined on the domain $(0,z_0)$. As $x\to 0_+$, $y_+(x) \to \sqrt{\frac{c}{3x}}[1+\OO(x)]$, and $y_+(x)$
            decreases monotonically to $0$ as $y_+(x)\to z_0$ from the left.

            iii. Finally, consider the branch of $y_+(x)$ defined on the domain $(z_+,\infty)$. As $x\to +\infty$, $y_+(x) = x[1 + \OO(x^{-2})]$. It is
            easy to check that $y_+(x)$ is again monotonic on $(z_+,\infty)$, and approaches $0$ as $x\to z_+$ from the right.

            We now show that graph of $y_+(x)$ and $\Gamma$ do not intersect in the upper half plane, provided $z_0 < a < z_+$. From the above analysis, it is 
            easy to see that branch ii. of $y_+(x)$ (defined on $(0,z_0)$) does not intersect $\Gamma$ provided $z_0 < a$. Let us now examine the intersection properties
            branch iii. of $y_+(x)$ and $\Gamma$.  Define the function $H(x) :=  [\sqrt{3x^2-3a^2} -y_+(x)] [\sqrt{3x^2-3a^2} +y_+(x)]$, and note that 
            $\sqrt{3x^2-3a^2} -y_+(x) > H(x)$, so if we can show $H(x)>0$, we are done. We have that
                \begin{equation*}
                    \lim_{x\to z_+} H(x) > \lim_{x\to z_+} \left[\sqrt{3x^2-3z_+^2} -y_+(x)\right] \left[\sqrt{3x^2-3a^2} +y_+(x)\right] = 0,
                \end{equation*}
            and so $H(z_+) > 0$. Furthermore,
                \begin{align*}
                    \frac{d}{dx}H(x)  = \frac{d}{dx} \left[\frac{12x^3 + (3b-18a^2)x - 2c}{6x}\right] = \frac{12x^3+c}{3x^2} > 0,
                \end{align*}
            and so $H(x)>0$ on $(z_+,\infty)$, since $(\sqrt{3x^2-3z_+^2} +y_+(x)) > 0$ there trivially.
            So $H(x)$ is monotone increasing on $(z_+,\infty)$, and thus $\Gamma$ does not intersect the graph of $y_+(x)$ in this region. 

            It remains to see that the graph of branch i. of $y_+(x)$ does not intersect $\Gamma$. We again consider the function $H(x)$ defined above;
                \begin{equation*}
                    \lim_{x\to z_+} H(x) > \lim_{x\to z_-} [\sqrt{3x^2-3z_-^2} -y_+(x)] [\sqrt{3x^2-3a^2} +y_+(x)] = 0,
                \end{equation*}
            and so we see that $H(z_-)>0$. For $x<z_-<0$, we have that
                \begin{equation*}
                    \frac{d}{dx}H(x)  = \frac{12x^3+c}{3x^2} < 0,
                \end{equation*}
            so that $H(x)<0$ on $(-\infty,z_-)$, and we have proven the lemma.
            
        \end{proof}
        The following proposition we will need in order to guarantee that the lensing inequalities hold. This lemma also demonstrates in what 
        sense the spectral curve changes when we approach the critical surface. The proof of this lemma follows from straightforward local analysis
        of the functions $g_j(\lambda) = g(u_j(\lambda))$, and so we omit it.
        
        \begin{lemma}\label{lemma:Localexpansions}
    \textit{Behavior of $g_j(\lambda)$ at the branch points.}
    
        \begin{enumerate}
            \item (Generic Case). Let $(\eta,\mu,\nu) \in D$. Then, as $\lambda\to \alpha$,
                \begin{align*}
                    (g_3-g_2)(\lambda) &= \begin{cases}
                        +i\rho_{\alpha}\cdot (\lambda-\alpha)^{3/2}[1 + \OO(\lambda-\alpha)], & \text{Im } \lambda >0,\\
                        -i\rho_{\alpha}\cdot (\lambda-\alpha)^{3/2}[1 + \OO(\lambda-\alpha)], & \text{Im } \lambda <0,
                    \end{cases}
                \end{align*}
             where $\rho_{\alpha} = \rho_{\alpha}(a,b,c) := -\frac{8}{9\sqrt{3a}}(6a^3-3ab-2c) > 0$. Furthermore, as $\lambda\to \beta$,
                \begin{align*}
                    (g_2- g_1)(\lambda) = \rho_{\beta}\cdot (\lambda-\beta)^{3/2}[1 + \OO(\lambda-\beta)],
                \end{align*}
            where $\rho_{\beta} = \rho_{\beta}(a,b,c) := -\frac{8}{9\sqrt{3a}}(6a^3-3ab+2c) > 0$.
            \item (Critical Surface, $\mu \neq 0$). Let $(\eta,\mu,\nu) \in \partial D \setminus \gamma_{\pm}$. Then, as $\lambda\to \alpha$,
                \begin{align*}
                    (g_3-g_2)(\lambda) &= \begin{cases}
                        +i\rho_{\alpha}\cdot (\lambda-\alpha)^{3/2}[1 + \OO(\lambda-\alpha)], & \text{Im } \lambda >0,\\
                        -i\rho_{\alpha}\cdot (\lambda-\alpha)^{3/2}[1 + \OO(\lambda-\alpha)], & \text{Im } \lambda <0,
                    \end{cases}
                \end{align*}
             where $\rho_{\alpha} = \rho_{\alpha}(a,b,c) := -\frac{8}{9\sqrt{3a}}(6a^3-3ab-2c) > 0$. Furthermore, as $\lambda\to \beta$,
                \begin{align*}
                    (g_2-g_1)(\lambda) &= -\hat{\rho}_{\beta} \cdot  (\lambda-\beta)^{5/2}[1 + \OO(\lambda-\beta)],
                \end{align*}
            where $\hat{\rho}_{\beta} := \hat{\rho}_{\beta}(a,b,c) := \frac{8\sqrt{3}}{135a^{7/2}}(2ab-c) > 0$.
            \item (The curve $\gamma_+$). Let $(\eta,\mu,\nu) \in \gamma_+$. Then, as $\lambda\to \alpha$,
                \begin{align*}
                    (g_3-g_2)(\lambda) &= \begin{cases}
                        +i\hat{\rho} \cdot (\lambda-\alpha)^{5/2}[1 + \OO(\lambda-\alpha)], & \text{Im } \lambda >0,\\
                        -i\hat{\rho} \cdot (\lambda-\alpha)^{5/2}[1 + \OO(\lambda-\alpha)], & \text{Im } \lambda <0,
                    \end{cases}
                \end{align*}
            where $\hat{\rho} := \frac{8\sqrt{3}}{135a^{7/2}}ab > 0$. Furthermore, as $\lambda\to \beta$,
                \begin{align*}
                    (g_2-g_1)(\lambda) &= -\hat{\rho} \cdot  (\lambda-\beta)^{5/2}[1 + \OO(\lambda-\beta)],
                \end{align*}
            where $\hat{\rho}> 0$ is as previously defined.
            \item (The curve $\gamma_-$). Let $(\eta,\mu,\nu) \in \gamma_-$. Then, $\alpha = \beta = 0$, and as $\lambda \to 0$,
                \begin{align*}
                    (g_2-g_1)(\lambda) &= 
                    \begin{cases}
                        \frac{3}{5}b(1-\omega)\lambda^{5/3}[1+\OO(\lambda^{2/3})], & \text{Im } \lambda >0,\\
                        \frac{3}{5}b(1-\omega^2)\lambda^{5/3}[1+\OO(\lambda^{2/3})], & \text{Im } \lambda <0,
                    \end{cases}\\
                    (g_3-g_2)(\lambda) &= 
                    \begin{cases}
                        \frac{3}{5}b(\omega-\omega^2)\lambda^{5/3}[1+\OO(\lambda^{2/3})], & \text{Im } \lambda >0,\\
                        \frac{3}{5}b(\omega^2-\omega)\lambda^{5/3}[1+\OO(\lambda^{2/3})], & \text{Im } \lambda <0.
                    \end{cases}
                \end{align*}
        \end{enumerate}
    \end{lemma}
        The above lemmas are enough to guarantee that the inequalities we will need for lensing indeed hold:
        \begin{prop}\label{LensingProposition}
            Let $(\eta,\mu,\nu)\in D$. We have the following inequalities:
                \begin{equation*}
                    \begin{cases}
                        (a.)\qquad \text{Re}[g_3(\lambda) - g_2(\lambda)] > 0, & \lambda \in \hat{\Gamma}_5,\\
                        (b.)\qquad \text{Re}[g_2(\lambda) - g_1(\lambda)] > 0, & \lambda \in \hat{\Gamma}_{-3},\\
                        (c.)\qquad \text{Re}[g_3(\lambda) - g_2(\lambda)] < 0, & \text{in a lens around $[\alpha,\infty)$},\\
                        (d.)\qquad \text{Re}[g_2(\lambda) - g_1(\lambda)] < 0, & \text{in a lens around $(-\infty,\beta]$}.
                    \end{cases}
                \end{equation*}
        \end{prop}
        \begin{proof}
            Lemma \ref{Lemma:Lensing} showed that the branch cuts $\Gamma$ and the curve $\mathcal{C}$ do not intersect. Let us give
            an alternate interpretation of this result. Let 
                \begin{equation*}
                    \hat{{\bf n}} := \frac{\nabla \text{Im } \lambda(u)}{||\nabla \text{Im } \lambda(u)||}.
                \end{equation*}
            We claim that $\nabla \text{Re } g(u) \cdot \hat{{\bf n}} = \frac{\partial }{\partial n} \text{Re } g(u)$ is of constant sign on 
            each connected component of the branch cuts. To see this, observe that $\partial [\text{Re } g(u)] = \frac{1}{2}\partial g(u)$,
            and similarly $\partial [\text{Im } \lambda(u)] = \frac{1}{2 i}\partial \lambda(u)$, where $\partial$ here denotes the holomorphic 
            derivative in $u$. It follows that $\nabla \text{Re } g(u)$ and $\nabla \text{Re } \lambda(u)$ are perpendicular if and only if
            the ratio of these expressions is purely imaginary:
                \begin{equation*}
                    \frac{\partial g(u)}{i\partial \lambda(u)} \in i\RR. 
                \end{equation*}
            What we have shown is that the curve defined by $\frac{\partial }{\partial n} \text{Re } g(u) = 0$ 
            (i.e. where $\text{Re } g(u)$ changes sign) is characterized by
                \begin{equation*}
                    \text{Im } \frac{\partial g(u)}{\partial \lambda(u)} = \text{Im } Y(u) = 0,
                \end{equation*}
            which is precisely the curve $\mathcal{C}$. Thus, $\text{Re }g(u)$ is of constant sign on each connected component of the branch cuts.
            
            With this fact in place, let us proceed to the proof of the inequalities $(a.)$--$(d.)$.

            \textit{Proof of $(a.)$} Let 
                \begin{equation*}
                    E:=\{\lambda \,|\, \text{Re }(g_3-g_2)(\lambda) >0\}.
                \end{equation*}
            For $\lambda$ sufficiently large in the lower half plane, formulae \eqref{u1-asymptotics}--\eqref{u3-asymptotics} and the definition $g_j(\lambda) = g(u_j(\lambda))$ $\text{Re }(g_3-g_2)(\lambda)>0$ for $-\frac{6\pi}{7} <\arg \lambda < -\frac{3\pi}{7}$. Furthermore,
            for near $\lambda = \alpha$ in the lower half plane, Lemma \ref{lemma:Localexpansions} yields that  $\text{Re }(g_3-g_2)(\lambda) \sim - \rho_{\alpha}\cdot \text{Im } (\lambda-\alpha)^{3/2} >0$ for $|\lambda-\alpha|$ sufficiently small and 
            $-\pi < \arg (\lambda-\alpha) < -\frac{2\pi}{3}$. Since Lemma \ref{Lemma:Lensing} guarantees that the boundary of $E$ is bounded away from the branch cuts, and $\text{Re } (g_3-g_2)(\lambda)$ is a non-constant harmonic function, the maximum principle implies that $E$ is necessarily unbounded, and reaches infinity in the sector
            $-\frac{6\pi}{7} <\arg \lambda < -\frac{3\pi}{7}$. Since this sector contains the contour $\hat{\Gamma}_5$ for $\lambda$ sufficiently 
            large, by a possible slight redefinition of $\hat{\Gamma}_5$ near $\lambda=\alpha$, we can conclude that 
            $\text{Re }(g_3-g_2)(\lambda)>0$ on $\hat{\Gamma}_5$.

            \textit{Proof of $(b.)$} The proof of $(b.)$ is almost identical to that of $(a.)$, and so we omit it.
            
            \textit{Proof of $(c.)$} For $\lambda \in (\alpha,\infty)$, we have that $u_{2,+}(\lambda) = \overline{u_{3,-}(\lambda)}$. Since 
            $\overline{g(u)} = g(\overline{u})$, we see that
                \begin{equation*}
                    \overline{g(u_{2,+}(\lambda))} = g(u_{3,-}(\lambda)),
                \end{equation*}
            and so $\text{Re }(g_3-g_2)(\lambda) = 0$ for $z\in (\alpha,\infty)$. By the definition of sheets $2$ and $3$,
                \begin{equation*}
                    \frac{\partial}{\partial n_+}  \text{Re } g_2(\lambda)= -\frac{\partial}{\partial n_-}  \text{Re } g_3(\lambda), \qquad\qquad
                    \frac{\partial }{\partial n_-} \text{Re } g_2(\lambda) = -\frac{\partial}{\partial n_+}  \text{Re } g_3(\lambda),
                \end{equation*}
            where $\frac{\partial}{\partial n_{\pm}}$ denote the normal derivatives in the upper and lower half $\lambda$-planes, respectively
            (note that these normal derivatives differ from the normal derivatives of $g(u)$ only by an overall positive factor of $|\lambda'(u)|>0$). Furthermore, since $\overline{g(u)} = g(\overline{u})$,
                \begin{equation*}
                    \frac{\partial}{\partial n_+} \text{Re }g_2(\lambda) = \frac{\partial}{\partial n_-} \text{Re }g_2(\lambda),\qquad\qquad 
                    \frac{\partial}{\partial n_+} \text{Re }g_3(\lambda) = \frac{\partial}{\partial n_-} \text{Re }g_3(\lambda),
                \end{equation*}
            and so 
                \begin{equation*}
                    \frac{\partial}{\partial n_{\pm}} \left[\text{Re }(g_3-g_2)(\lambda)\right] = 2\frac{\partial}{\partial n_{\pm}} \text{Re } g_3(\lambda). \qquad\qquad (\star)
                \end{equation*}
            Now, Lemma \ref{lemma:Localexpansions} allows us to compute that locally near $\lambda = \alpha$, $\text{Re }(g_3-g_2)(\lambda)<0$
            for $|\arg (\lambda-\alpha)| < \frac{2\pi}{3}$. Since this quantity is negative near $\lambda = \alpha$, Lemma \eqref{Lemma:Lensing} along with the equality $(\star)$ allow us to conclude that $\text{Re }(g_3-g_2)(\lambda) <0$ in a lens-shaped region around $[\alpha,\infty)$.
        
            \textit{Proof of $(d.)$} The proof of $(d.)$ is almost identical to that of $(c.)$, and so we omit it.
        \end{proof}

\begin{remark}
    Actually, all of these inequalities of Proposition \ref{LensingProposition} also extend to the boundary of $D$ as well (i.e., these inequalities also hold for $(\eta,\mu,\nu)$ on the critical surface). The proof of this fact is virtually identical to the proof of Proposition \ref{LensingProposition}, and so we omit it.
\end{remark}

With these inequalities in place, we are ready to proceed to the Deift-Zhou steepest descent analysis for ${\bf Z}(\lambda;\eta,\mu,\nu | \hbar)$.
    
\subsection{Deift-Zhou steepest descent.}
The rest of this section is devoted to the proof of Theorem \eqref{MainTheorem1}; this requires a steepest descent analysis of the Riemann-Hilbert
problem for ${\bf Z}(\lambda;\eta,\mu,\nu|\hbar)$.

We now define the first transformation: let 
    \begin{equation}
        G(\lambda) := \text{diag }(g_1(\lambda),g_2(\lambda),g_3(\lambda)),
    \end{equation}
where $g_k(\lambda)$ are the $g$-functions defined in the previous section, and
        \begin{equation}
                \mathfrak{h}^{(0)} = 
                    \begin{psmallmatrix}
                        1 & -\frac{1}{2\hbar}h_1^{(0)} & \frac{1}{8\hbar^2}(h_1^{(0)})^2 - \frac{1}{4\hbar}h_2^{(0)}\\
                        0 & 1 & -\frac{1}{2\hbar}h_1^{(0)}\\
                        0 & 0 & 1
                    \end{psmallmatrix},
            \end{equation}
where $h_1^{(0)}$, $h_2^{(0)}$ are leading parts of the Hamiltonians $H_1$, $H_2$ (as defined in by formulae \eqref{h1_0} and \eqref{h2_0}).
We then define the first transformation to be
    \begin{equation}
        {\bf U}(\lambda;\eta,\mu,\nu | \hbar):=(\mathfrak{h}^{0})^{-1}{\bf Z}(\lambda;\eta,\mu,\nu | \hbar)e^{-\frac{1}{\hbar}G(\lambda)}.
    \end{equation}
Then, we have the following proposition:
    \begin{prop}
        ${\bf U}(\lambda;\eta,\mu,\nu | \hbar)$ satisfies the following RHP:
            \begin{equation*}
                {\bf U}_+(\lambda;\eta,\mu,\nu | \hbar ) = {\bf U}_-(\lambda;\eta,\mu,\nu | \hbar)\cdot
                \begin{cases}
                    \begin{psmallmatrix}
                        e^{\frac{1}{\hbar}[g_{1,-}(\lambda)-g_{2,-}(\lambda)]} & -1\\
                        0 & e^{-\frac{1}{\hbar}[g_{1,-}(\lambda)-g_{2,-}(\lambda)]}
                    \end{psmallmatrix}\oplus 1, & \lambda \in (-\infty,\beta],\\
                    1 \oplus \begin{psmallmatrix}
                        e^{\frac{1}{\hbar}[g_{2,-}(\lambda)-g_{3,-}(\lambda)]} & -1\\
                        0 & e^{-\frac{1}{\hbar}[g_{2,-}(\lambda)-g_{3,-}(\lambda)]}
                    \end{psmallmatrix}, & \lambda \in [\alpha,\infty),\\
                    \mathbb{I} + E_{23}e^{\frac{1}{\hbar}(g_2-g_3)(\lambda)}, & \lambda \in \hat{\Gamma}_5,\\
                    \mathbb{I} + E_{12}e^{\frac{1}{\hbar}(g_1-g_2)(\lambda)}, & \lambda \in \hat{\Gamma}_{-3},
                \end{cases}
            \end{equation*}
        subject to the normalization condition
            \begin{equation} \label{U-normalization}
                {\bf U}(\lambda;\eta,\mu,\nu | \hbar) = \left[\mathbb{I} + \OO(\lambda^{-1}) \right]\hat{f}(\lambda),
            \end{equation}
    \end{prop}
\begin{proof}
    The form of the jumps of ${\bf U}(\lambda;\eta,\mu,\nu | \hbar)$ follows from direct calculation, using the definitions of $G$, ${\bf Z}$.
    The fact that the asymptotic \eqref{U-normalization} holds can be seen immediately from the fact that
        \begin{equation*}
            \hat{f}(\lambda)e^{\frac{1}{\hbar}[\Theta(\lambda) - G(\lambda)]}\hat{f}^{-1}(\lambda) = \mathfrak{h}^{(0)} + \OO(\lambda^{-1}), 
            \qquad\qquad \lambda\to \infty.
        \end{equation*}
\end{proof}

Note that the jumps on the real axis are in standard form for lens opening; by construction (our choice of the $g$-function), it is possible to open lenses, and after doing so, all jumps off the real axis decay exponentially as $\hbar \to 0$. 
We define two lens-shaped domains, one about $\lambda = \beta$, which opens symmetrically about $\arg (\lambda -\beta) = \pi$, and one about 
$z= \alpha$, which opens symmetrically about $\arg(\lambda - \alpha) = 0$. We label these regions $\Delta_{\beta}$, $\Delta_{\alpha}$, respectively,
and let $\Delta_{\beta}^{\pm}$, $\Delta_{\alpha}^{\pm}$ be the connected components of these domains in the upper (resp. lower) half planes, and let
$\Gamma_{\beta}^{\pm}$, $\Gamma_{\alpha}^{\pm}$ denote the boundary component of $\Delta_{\beta}^{\pm}$, $\Delta_{\alpha}^{\pm}$ which does not include the real line.

Define $2\times 2$ matrices
    \begin{equation*}
        v_{\alpha}(\lambda) := \begin{psmallmatrix}
            1 & 0\\
            -e^{\frac{1}{\hbar}(g_3-g_2)(\lambda)} & 1
        \end{psmallmatrix}, \qquad\qquad 
        v_{\beta}(\lambda) := \begin{psmallmatrix}
            1 & 0\\
            -e^{\frac{1}{\hbar}(g_2-g_1)(\lambda)} & 1
        \end{psmallmatrix}.
    \end{equation*}
Note that these matrices are analytic and invertible in the domains $\Delta_{\alpha}^{\pm}$, $\Delta_{\beta}^{\pm}$, respectively. Now, put
    \begin{equation*}
        {\bf V}(\lambda;\eta,\mu,\nu | \hbar ) :=
        {\bf U}(\lambda;\eta,\mu,\nu | \hbar ) \times
            \begin{cases}
                v_{\beta}(\lambda) \oplus 1, & \lambda \in \Delta_{\beta}^{+},\\
                v_{\beta}^{-1}(\lambda) \oplus 1, & \lambda \in \Delta_{\beta}^{-},\\
                1 \oplus v_{\alpha}^{-1}(\lambda), & \lambda \in \Delta_{\alpha}^{+},\\
                1 \oplus v_{\alpha}(\lambda), & \lambda \in \Delta_{\alpha}^{-},\\
                \mathbb{I}, & \textit{ otherwise.}
            \end{cases}
    \end{equation*}

Then, we have the following Proposition.
\begin{prop}\label{V-prop}
    ${\bf V}(\lambda;\eta,\mu,\nu|\hbar)$ satisfies the following RHP.
        \begin{equation*}
            {\bf V}_+(\lambda;\eta,\mu,\nu|\hbar) = {\bf V}_-(\lambda;\eta,\mu,\nu|\hbar) \times
                \begin{cases}
                    \mathbb{I} + E_{23}e^{-\frac{1}{\hbar}(g_3-g_2)(\lambda)}, & \lambda \in \hat{\Gamma}_5,\\
                    \mathbb{I} + E_{12}e^{-\frac{1}{\hbar}(g_2-g_1)(\lambda)}, & \lambda \in \hat{\Gamma}_{-3},\\
                    \mathbb{I} - E_{32}e^{\frac{1}{\hbar}(g_3-g_2)(\lambda)}, & \lambda \in \Gamma^{\pm}_{\alpha},\\
                    \mathbb{I} - E_{21}e^{\frac{1}{\hbar}(g_2-g_1)(\lambda)}, & \lambda \in \Gamma^{\pm}_{\beta},\\
                    (-i\sigma_2)\oplus 1, & \lambda \in (-\infty,\beta],\\
                    1\oplus(-i\sigma_2), & \lambda \in [\alpha,\infty).
                \end{cases}
        \end{equation*}
    Furthermore, ${\bf V}(\lambda;\eta,\mu,\nu|\hbar)$ is normalized as
        \begin{equation*}
            {\bf V}(\lambda;\eta,\mu,\nu|\hbar) = \left[\mathbb{I} + \OO(\lambda^{-1})\right]\hat{f}(\lambda).
        \end{equation*}
\end{prop}

By Proposition \ref{LensingProposition}, all jumps in the above are exponentially close to the identity matrix, 
except for the jumps on the real axis. As is standard, we search for the solution to a model Riemann-Hilbert problem for a function 
$M(\lambda)$, which solves a RHP identical to ${\bf V}$, but ignores the exponentially small jumps:
\begin{problem}\label{prob:global-parametrix}
\textit{(Global parametrix problem).} Find a $3\times 3$ piecewise analytic function in $\CC\setminus \left((-\infty,\beta]\cup[\alpha,\infty)\right)$, which satisfies the boundary conditions
    \begin{equation*}
        M_+(\lambda) = M_-(\lambda)\cdot
        \begin{cases}
            (-i\sigma_2)\oplus 1, & \lambda \in (-\infty,\beta],\\
            1\oplus(-i\sigma_2), & \lambda \in [\alpha,\infty),
        \end{cases}
    \end{equation*}
with normalization
    \begin{equation*}
        M(\lambda) = 
        \begin{cases}
            \left[\mathbb{I} +\OO(\lambda^{-1})\right]\hat{f}(\lambda),& \lambda\to\infty,\\
            \OO( |\lambda-\alpha|^{-1/4} ), & \lambda\to \alpha,\\
            \OO( |\lambda-\beta|^{-1/4} ), & \lambda\to \beta.
        \end{cases}
    \end{equation*}
\end{problem}
It turns out we can solve problem \ref{prob:global-parametrix} explicitly:
    \begin{prop}\label{prop:global-parametrix-solution}
        Define functions
            \begin{equation}
                \phi_1(u) = \frac{i}{\sqrt{3}}\frac{u^2 -\frac{3}{4}\varsigma}{\sqrt{u^2-\frac{1}{2}\varsigma}} ,\qquad \phi_2(u) = \frac{i}{\sqrt{3}}\frac{u}{\sqrt{u^2-\frac{1}{2}\varsigma}},\qquad \phi_3(u) = \frac{i}{\sqrt{3}}\frac{1}{\sqrt{u^2-\frac{1}{2}\varsigma}},
            \end{equation}
        where the branch cut of the square root is taken to be the interval $[-\sqrt{\varsigma/2},\sqrt{\varsigma/2}]$, and the branch is chosen so that
        \begin{equation*}
            \frac{1}{\sqrt{u^2-\frac{1}{2}\varsigma}} = u^{-1} +\OO(u^{-2}), \qquad\qquad u\to \infty.
        \end{equation*}
        Then, the unique solution to the RHP \ref{prob:global-parametrix}     is given by
            \begin{equation}
                M_{ij}(\lambda) = \phi_i(u_j(\lambda))\cdot\begin{cases}
                \mathbb{I}, & \text{Im } \lambda >0,\\
                \text{diag }(1, -1, 1), & \text{Im } \lambda <0,
            \end{cases} \qquad i,j = 1,2,3,
            \end{equation}
        where $u_j(\lambda)$ are the uniformization coordinates on the spectral curve.
    \end{prop}
    \begin{proof}
    Let
        \begin{equation*}
            \vec{m}(\lambda) =
            \begin{cases}
                \left[r_1(u_1(\lambda)),r_2(u_2(\lambda)),r_3(u_3(\lambda))\right], & \text{Im } \lambda >0,\\
                \left[q_1(u_1(\lambda)),q_2(u_2(\lambda)),q_3(u_3(\lambda))\right], & \text{Im } \lambda <0
            \end{cases}
        \end{equation*}
    be a row of $M(\lambda)$. Equating the analytic continuation of $\vec{m}(\lambda)$ through $(-\infty,\beta]$ using the properties of $u_j(\lambda)$ to the analytic continuation coming from the jump condition, one finds that
        \begin{equation*}
            \left[r_1(u_2(\lambda)),r_2(u_1(\lambda)),r_3(u_3(\lambda))\right]
            = \left[-q_2(u_2(\lambda)),q_1(u_1(\lambda)),q_3(u_3(\lambda))\right].
        \end{equation*}
    A similar calculation on $[\alpha,\infty)$ yields that 
        \begin{equation*}
            \left[r_1(u_1(\lambda)),r_2(u_3(\lambda)),r_3(u_2(\lambda))\right]
            = \left[q_1(u_1(\lambda)),q_3(u_3(\lambda)),-q_2(u_2(\lambda))\right].
        \end{equation*}
    Thus, if we set $r_1(u) := \phi(u)$, one finds that
        \begin{equation*}
            \vec{m}(\lambda) = \left[\phi(u_1(\lambda)),\phi(u_2(\lambda)),\phi(u_3(\lambda))\right]\cdot
            \begin{cases}
                \mathbb{I}, & \text{Im } \lambda >0,\\
                \text{diag }(1, -1, 1), & \text{Im } \lambda <0.
            \end{cases}
        \end{equation*}
    We now must choose an appropriate $\phi$ to match the $\lambda\to \infty$ asymptotics of each row of $M(\lambda)$.
        Take the ansatz that
            \begin{equation*}
                \phi(u) = \frac{Au^2 + Bu + C}{\sqrt{u^2-\frac{1}{2}\varsigma}}.
            \end{equation*}
        The cut of $\phi(u)$ is take to be the segment $[-\sqrt{\varsigma/2},\sqrt{\varsigma/2}]$, so that the function on the Riemann surface of $\lambda(u)$ has a cut on $[\beta,\alpha]$ on the second sheet, and the jumps of $M(\lambda)$ will automatically be satisfied. 
        The form of this ansatz is enough to guarantee we can match the expansion of $M(\lambda)$ at infinity, by requiring
            \begin{equation*}
                M(\lambda)\hat{f}^{-1}(\lambda) = \mathbb{I} + \OO(\lambda^{-1}).
            \end{equation*}
        Expansion of the uniformization coordinate at infinity and performing a matching of parameters yields the result of the proposition.
        \end{proof}

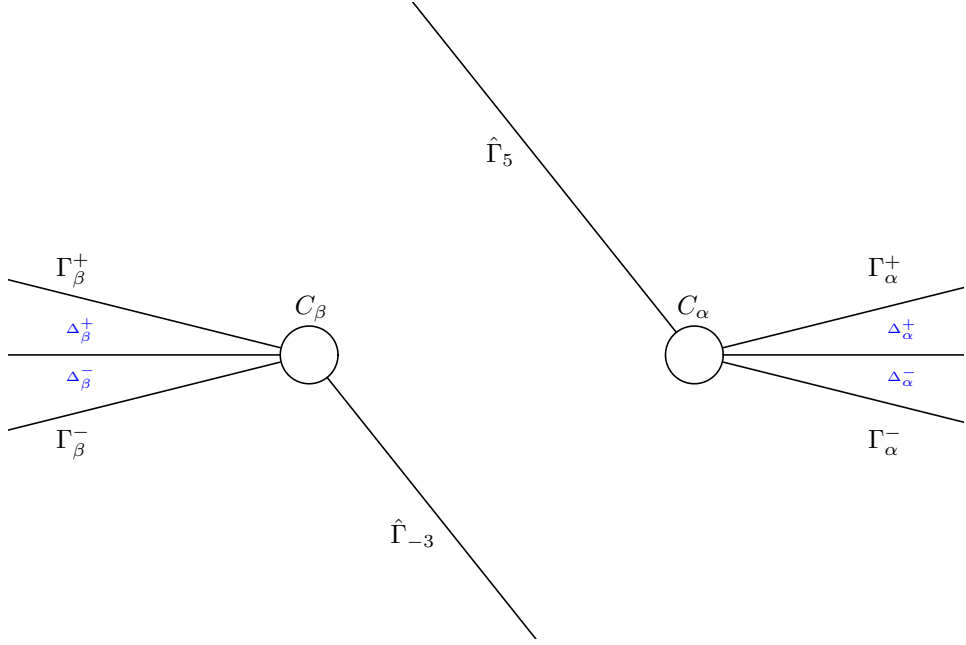
\begin{figure}
    \centering
    \begin{overpic}[scale=0.4]{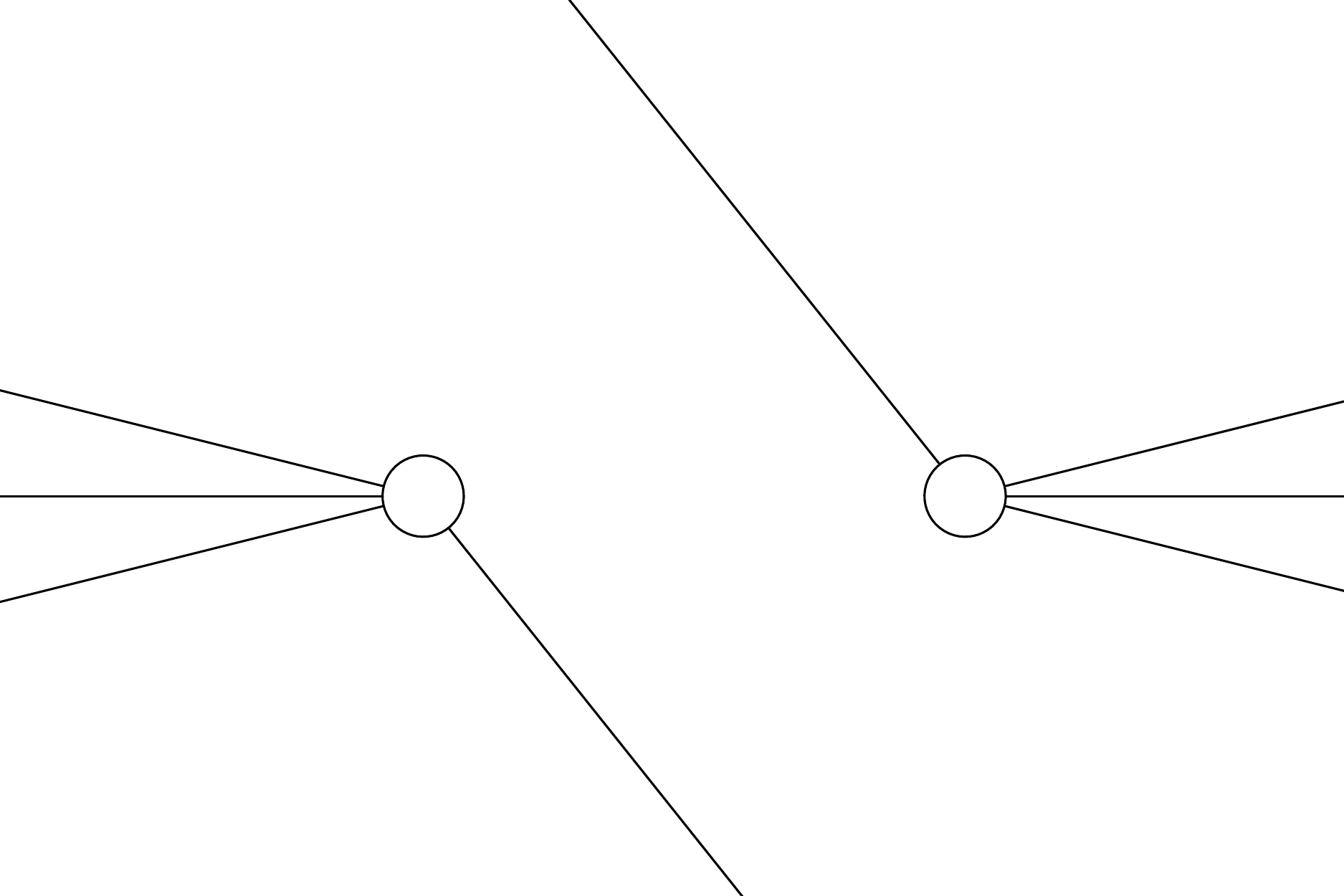}
        \put (50,50) {$\hat{\Gamma}_5$}
        \put (40,10) {$\hat{\Gamma}_{-3}$}
        \put (6,27) {\tiny \textcolor{blue}{$\Delta_{\beta}^{-}$}}
        \put (6,32) {\tiny \textcolor{blue}{$\Delta_{\beta}^{+}$}}
        \put (92,27) {\tiny \textcolor{blue}{$\Delta_{\alpha}^{-}$}}
        \put (92,32) {\tiny \textcolor{blue}{$\Delta_{\alpha}^{+}$}}
        \put (5,38)  {$\Gamma_\beta^{+}$}
        \put (5,20)  {$\Gamma_\beta^{-}$}
        \put (90,38)  {$\Gamma_\alpha^{+}$}
        \put (90,20)  {$\Gamma_\alpha^{-}$}
        \put (30,34) {$C_{\beta}$}
        \put (70,34) {$C_{\alpha}$}
    \end{overpic}
    \caption{Final set of contours after lens opening for the Riemann-Hilbert problem for ${\bf R}$. All rays in the figure are oriented \textbf{\textit{outwards}}; circles are oriented counterclockwise.}
    \label{fig:enter-label}
\end{figure}

\begin{remark}
    As a consequence of the symmetry of the uniformization coordinates
    \begin{align}
        u_1(-\lambda;\eta,-\mu,\nu) &= -u_3(\lambda;\eta,\mu,\nu), \nonumber\\
        u_2(-\lambda;\eta,-\mu,\nu) &= -u_2(\lambda;\eta,\mu,\nu), \label{uniformizer-symmetry}\\
        u_3(-\lambda;\eta,-\mu,\nu) &= -u_1(\lambda;\eta,\mu,\nu). \nonumber
    \end{align}
    and the even/oddness of the functions $\phi_j(u)$, the global parametrix carries the symmetry
        \begin{equation*}
            M(\lambda;\eta,\mu,\nu) = 
            \begin{psmallmatrix}
                1 & 0 & 0\\
                0 & -1 & 0\\
                0 & 0 & 1
            \end{psmallmatrix}M(-\lambda;\eta,-\mu,\nu)
            \begin{psmallmatrix}
                0 & 0 & 1\\
                0 & -1 & 0\\
                1 & 0 & 0
            \end{psmallmatrix}.
        \end{equation*}
\end{remark}

What remains is to open local lenses around the turning points $\lambda = \alpha,\beta$. This is a standard construction involving Airy functions
\cite{DKMVZ}. By the symmetry \eqref{uniformizer-symmetry}, we claim that it is enough to construct the local parametrix at $\lambda = \beta$. If we call this parametrix $P(\lambda)$, then the parametrix at $\lambda = \alpha$ can be constructed as
    \begin{equation*}
        P_{\alpha}(\lambda) = 
        \begin{psmallmatrix}
                1 & 0 & 0\\
                0 & -1 & 0\\
                0 & 0 & 1
            \end{psmallmatrix}P(-\lambda;\eta,-\mu,\nu)
            \begin{psmallmatrix}
                0 & 0 & 1\\
                0 & -1 & 0\\
                1 & 0 & 0
            \end{psmallmatrix}.
    \end{equation*}

Let $D_{\alpha}$, $D_{\beta}$ be local discs about
$\lambda=\alpha,\beta$, respectively, with boundaries $C_{\alpha}$, $C_{\beta}$. In the disc $D_{\beta}$, we define the conformal map
    \begin{align}
        \zeta(\lambda) &:= \left[\frac{3}{4}(g_2-g_1)(\lambda)\right]^{2/3}.
    \end{align}
The fact that this map is conformal near $\lambda =\beta$ follows immediately from Lemma \ref{lemma:Localexpansions}. We consider the
\textit{Airy parametrix} ${\bf A}(\zeta)$, as defined in \cite{DKMVZ}:
        \begin{align*}
            &{\bf A}_+(\zeta) = {\bf A}_-(\zeta) \cdot 
            \begin{cases}
                \begin{psmallmatrix}
                    0 & -1\\
                    1 & 0
                \end{psmallmatrix}, & \arg \zeta = \pi,\\
                \begin{psmallmatrix}
                    1 & 0\\
                    -1 & 1
                \end{psmallmatrix}, & \arg \zeta = \pm\frac{2\pi}{3},\\
                \begin{psmallmatrix}
                    1 & 1\\
                    0 & 1
                \end{psmallmatrix}, &\arg\zeta = 0.
            \end{cases},\\
            &{\bf A}(\zeta) = \frac{1}{\sqrt{2}}\zeta^{-\frac{1}{4}\sigma_3}
            \begin{psmallmatrix}
                1 & i\\
                i & 1
            \end{psmallmatrix}\left[\mathbb{I} + \OO(\zeta^{-3/2})\right]e^{-\frac{2}{3}\zeta^{3/2}\sigma_3},\qquad \zeta\to \infty.
        \end{align*}
In the above, all rays are oriented outwards, away from the origin.
The solution to this parametrix is given explicitly in terms of the Airy function. For example,
    \begin{equation*}
        {\bf A}(\zeta) = 
        \sqrt{2\pi}
        \begin{pmatrix}
            \Ai(\zeta) & -\omega^2\Ai(\omega^2\zeta)\\
            -i\Ai'(\zeta) & i\omega\Ai(\omega^2\zeta)
        \end{pmatrix}, \qquad 0 < \arg \zeta < \frac{2\pi}{3},
    \end{equation*}
and the determination ${\bf A}(\zeta)$ in the other sectors can be inferred using the jump condition for ${\bf A}(\zeta)$ and the connection formula for the Airy function.
W will not need the exact formula here; the only information we require of this parametrix (aside from its jump matrices)
is its large-$\zeta$ expansion (cf. \cite{DKMVZ}, or for more detail \cite{BDY}, formula 7.9, for instance):
    \begin{equation}\label{Airy-Asymptotics}
        {\bf A}(\zeta)e^{\frac{2}{3}\zeta^{3/2}\sigma_3} \sim \frac{\zeta^{-\frac{1}{4}\sigma_3}}{\sqrt{2}}\sum_{k=0}^{\infty}
        \begin{pmatrix}
            s_k & 0\\
            0 & t_k
        \end{pmatrix}
        \begin{pmatrix}
            (-1)^k & i\\
            (-1)^ki & 1
        \end{pmatrix}\left(\frac{2}{3}\zeta^{3/2}\right)^{-k},
    \end{equation}
where
    \begin{equation}
        s_0 = t_0 = 1, \qquad s_k = \frac{\Gamma\left(3k+\frac{1}{2}\right)}{(54)^k k!\Gamma\left(k+\frac{1}{2}\right)},\qquad t_{k} = \frac{1+6k}{1-6k}s_k,\quad k\geq 1.
    \end{equation}
By a possible local redefinition of the lens boundaries and contour $\hat{\Gamma}_{5}$ near $\lambda=\beta$, we define the local parametrix
    \begin{align}
        P(\lambda) &:= E(\lambda)\left[
        {\bf A}\left(\hbar^{-2/3}\zeta(\lambda)\right)e^{\frac{2}{3\hbar}[\zeta(\lambda)]^{3/2}\sigma_3} \oplus 1\right],
    \end{align}
where $E(\lambda)$ is the analytic function (in a neighborhood of $\lambda = \beta$) 
    \begin{align}
        E(\lambda):=  M(\lambda)\left[\frac{1}{\sqrt{2}}
        \begin{psmallmatrix}
            1 &-i\\ -i & 1
        \end{psmallmatrix} [\hbar^{-2/3}\zeta(\lambda)]^{\frac{1}{4}\sigma_3} \oplus 1\right].
    \end{align}
It then follows immediately that 
\begin{lemma}\label{local-parametrix-lemma}
    As $\hbar \to 0$,
        \begin{equation}
            P(\lambda) \sim M(\lambda)\left[\mathbb{I} + \sum_{k=1}^{\infty} \left(P_k(\lambda)\oplus 0\right) \hbar^{k}\right],
        \end{equation}
    where
        \begin{equation}\label{P_k-def}
            P_k(\lambda) = \frac{1}{2}
            \begin{pmatrix}
                1 & -i\\
                -i & 1
            \end{pmatrix}
            \begin{pmatrix}
            s_k & 0\\
            0 & t_k
        \end{pmatrix}
        \begin{pmatrix}
            (-1)^k & i\\
            (-1)^ki & 1
        \end{pmatrix}\left(\frac{2}{3}\zeta(\lambda)^{3/2}\right)^{-k}.
        \end{equation}
\end{lemma}

With this lemma in place, we can then set
    \begin{equation}
        {\bf R}(\lambda;\eta,\mu,\nu|\hbar) :=
        {\bf V}(\lambda;\eta,\mu,\nu|\hbar) \cdot
            \begin{cases}
                M^{-1}(\lambda), & \lambda \in \CC\setminus (D_{\alpha} \cup D_{\beta}),\\
                \begin{psmallmatrix}
                    0 & 0 & 1\\
                    0 & -1 & 0\\
                    1 & 0 & 0
                \end{psmallmatrix}P^{-1}(-\lambda)
                \begin{psmallmatrix}
                    1 & 0 & 0\\
                    0 & -1 & 0\\
                    0 & 0 & 1
                \end{psmallmatrix}, & \lambda \in D_{\alpha},\\
                P^{-1}(\lambda), & \lambda \in D_{\beta}.
            \end{cases}
    \end{equation}

Defining $\Gamma_{\bf R}$ to be the discontinuity set of ${\bf R}(\lambda;\eta,\mu,\nu|\hbar)$, the following Proposition then holds immediately:
\begin{prop}
    ${\bf R}(\lambda;\eta,\mu,\nu|\hbar)$ satisfies the following Riemann-Hilbert problem:
    \begin{equation}
        \begin{cases}
            {\bf R}_+(\lambda;\eta,\mu,\nu|\hbar) = {\bf R}_-(\lambda;\eta,\mu,\nu|\hbar) J_{\bf R}(\lambda), & \lambda \in \Gamma_{\bf R},\\
            {\bf R}(\lambda;\eta,\mu,\nu|\hbar) = \mathbb{I} + \OO(\lambda^{-1}) & \lambda\to \infty,
        \end{cases}
    \end{equation}
    where
        \begin{equation}
            J_{\bf R}(\lambda) = 
            \begin{cases}
                \mathbb{I} +e^{-\frac{1}{\hbar}(g_3-g_2)(\lambda)} M(\lambda)E_{23}M^{-1}(\lambda), & \lambda \in \hat{\Gamma}_5 \setminus D_{\beta},\\
                \mathbb{I} + e^{-\frac{1}{\hbar}(g_2-g_1)(\lambda)}M(\lambda)E_{12}M^{-1}(\lambda), & \lambda \in \hat{\Gamma}_{-3} \setminus D_{\alpha},\\
                \mathbb{I} + e^{\frac{1}{\hbar}(g_3-g_2)(\lambda)}M(\lambda)E_{32}M^{-1}(\lambda), & \lambda \in \hat{\Gamma}^{\pm}_{\beta} \setminus D_{\beta},\\
                \mathbb{I} + e^{\frac{1}{\hbar}(g_2-g_1)(\lambda)}M(\lambda)E_{21}M^{-1}(\lambda), & \lambda \in \hat{\Gamma}^{\pm}_{\alpha} \setminus D_{\alpha},\\
                \begin{psmallmatrix}
                    1 & 0 & 0\\
                    0 & -1 & 0\\
                    0 & 0 & 1
                \end{psmallmatrix}P(-\lambda;\eta,-\mu,\nu)M^{-1}(-\lambda;\eta,-\mu,\nu)
                \begin{psmallmatrix}
                    1 & 0 & 0\\
                    0 & -1 & 0\\
                    0 & 0 & 1
                \end{psmallmatrix}, & \lambda \in C_{\alpha},\\
                P(\lambda)M^{-1}(\lambda), & \lambda \in C_{\beta}.
            \end{cases}
        \end{equation}
\end{prop}
Note that in the above proposition, all jumps are exponentially close to the identity as $\hbar\to 0$, except for the last two. Furthermore,
since these jumps are of the same form, we have only to check that $P(\lambda)M^{-1}(\lambda)$ is sufficiently close to the identity matrix,
and that this estimate does not change upon sending $\mu\to -\mu$. Indeed, we have that
    \begin{prop}\label{R-expansion-prop}
       For $\hbar$ sufficiently small, and $(\eta,\mu,\nu)\in D$, the Riemann-Hilbert problem for ${\bf R}(\lambda;\eta,\mu,\nu|\hbar)$ has a
       solution, and admits the $\hbar\to 0$ asymptotic expansion
        \begin{equation}
            {\bf R}(\lambda;\eta,\mu,\nu|\hbar) \sim \mathbb{I} + \sum_{k=1}^{\infty} {\bf R}_k(\lambda;\eta,\mu,\nu)\hbar^k,
        \end{equation}
    which holds uniformly for $\lambda \in \CC\setminus \Gamma_{{\bf R}}$.
    \end{prop}
\begin{proof}
    Let
        \begin{equation*}
            \boldsymbol{\Delta}(\lambda) := {\bf R}_-^{-1}{\bf R}_+ - \mathbb{I}
        \end{equation*}
    denote the deviation of the jumps of ${\bf R}$ from the identity $\mathbb{I}$. We have already observed that $\boldsymbol{\Delta}(\lambda)$ is exponentially close to 
    the identity as $\hbar\to 0$ on $\Gamma_{{\bf R}} \setminus \left(C_{\alpha} \cup C_{\beta}\right)$. Indeed, by Lemma \ref{local-parametrix-lemma}, we have that
        \begin{equation*}
            \boldsymbol{\Delta}(\lambda) \sim \sum_{k=1}^{\infty} J_k(\lambda) \hbar^{k},
        \end{equation*}
    where
        \begin{align*}
            J_k(\lambda) &= 
            \begin{cases}
                \begin{psmallmatrix}
                    1 & 0 & 0\\
                    0 & -1 & 0\\
                    0 & 0 & 1
                \end{psmallmatrix}
                    M(-\lambda;\eta,-\mu,\nu)\left(P_k(-\lambda;\eta,-\mu,\nu)\oplus 0\right)M^{-1}(-\lambda;\eta,\mu,-\nu)
                \begin{psmallmatrix}
                    1 & 0 & 0\\
                    0 & -1 & 0\\
                    0 & 0 & 1
                \end{psmallmatrix}, & \lambda \in C_{\alpha},\\
                M(\lambda;\eta,\mu,\nu)\left(P_k(\lambda;\eta,\mu,\nu)\oplus 0\right)M^{-1}(\lambda;\eta,\mu,\nu), & \lambda \in C_{\beta},
            \end{cases}
        \end{align*}
    and $P_k(\lambda) = P_k(\lambda;\eta,\mu,\nu)$ are as defined in Equation \eqref{P_k-def}. Using standard results \cite{DKMVZ}, we can therefore 
    expand ${\bf R}$ as an asymptotic series in $\hbar$:
        \begin{equation*}
            {\bf R}(\lambda;\eta,\mu,\nu|\hbar) \sim \mathbb{I} + \sum_{k=1}^{\infty} {\bf R}_k(\lambda;\eta,\mu,\nu)\hbar^k,
        \end{equation*}
    where the functions ${\bf R}_k(\lambda;\eta,\mu,\nu)$ can be determined iteratively as the solutions to certain additive Riemann-Hilbert problems. For $\lambda \in \CC\setminus D_{\alpha}\cup D_{\beta}$, the first
    term ${\bf R}_1$ is given by
        \begin{equation}\label{R1-sol}
            {\bf R}_1(\lambda;\eta,\mu,\nu) = \frac{W_1(\eta,\mu,\nu)}{\lambda - \beta} + \frac{\hat{W}_1(\eta,\mu,\nu)}{\lambda-\alpha},
        \end{equation}
    where
        \begin{align}
            W_1(\eta,\mu,\nu) &= \Res_{\lambda = \beta} \left[M(\lambda;\eta,\mu,\nu)\left(P_1(\lambda;\eta,\mu,\nu)\oplus 0\right)M^{-1}(\lambda;\eta,\mu,\nu)\right],\\
            \hat{W}_1(\eta,\mu,\nu) &=\begin{psmallmatrix}
                    1 & 0 & 0\\
                    0 & -1 & 0\\
                    0 & 0 & 1
                \end{psmallmatrix}W_1(\eta,-\mu,\nu)\begin{psmallmatrix}
                    1 & 0 & 0\\
                    0 & -1 & 0\\
                    0 & 0 & 1
                \end{psmallmatrix}
        \end{align}
\end{proof}

\subsection{Proof of Theorems \ref{MainTheorem1} \& \ref{Theorem:topologicalexpansion}.}
With the steepest descent analysis for ${\bf Z}$ completed, we can now proceed to the proof of Theorem \ref{MainTheorem1}.

\begin{proof}
    We must calculate the first two terms in the expansion of $\log\tau$; we must therefore calculate the error of ${\bf R}(\lambda)$ to order $\hbar$.
    By the results of the previous section, and by choosing an appropriate sector in which $\lambda\to \infty$, we can write 
        \begin{equation*}
            \mathfrak{G}(\lambda) = {\bf Z}(\lambda;{\bf t})e^{-\frac{1}{\hbar}\hat{\Theta}(\lambda;{\bf t})} = \mathfrak{h}^{(0)}{\bf R}(\lambda)\cdot M(\lambda)e^{\frac{1}{\hbar}(G-\hat{\Theta})(\lambda)},
        \end{equation*}
    so that
        \begin{align*}
            \mathfrak{G}^{-1}(\lambda)\mathfrak{G}'(\lambda) &= \underbrace{e^{-\frac{1}{\hbar}(G-\hat{\Theta})(\lambda)}\frac{d}{d\lambda}\left[e^{\frac{1}{\hbar}(G-\hat{\Theta})(\lambda)}\right]}_{K_1(\lambda)} + \underbrace{e^{-\frac{1}{\hbar}(G-\hat{\Theta})(\lambda)}M^{-1}(\lambda)M'(\lambda)e^{\frac{1}{\hbar}(G-\hat{\Theta})(\lambda)} }_{K_2(\lambda)}\\
            &+\underbrace{e^{-\frac{1}{\hbar}(G-\hat{\Theta})(\lambda)}M^{-1}(\lambda){\bf R}^{-1}(\lambda){\bf R}'(\lambda)M(\lambda)e^{\frac{1}{\hbar}(G-\hat{\Theta})(\lambda)}}_{K_3(\lambda)}
        \end{align*}
    One can readily check that
        \begin{equation*}
            \frac{1}{\hbar}\tr \left[K_2(\lambda)\frac{\partial \hat{\Theta}}{\partial \eta}\right] = \frac{1}{\hbar}\tr \left[K_2(\lambda)\frac{\partial \hat{\Theta}}{\partial \mu}\right] = \frac{1}{\hbar}\tr \left[K_2(\lambda)\frac{\partial \hat{\Theta}}{\partial \nu}\right] = \OO(\lambda^{-3}),\qquad \lambda\to \infty,
        \end{equation*}
    and so this term does not contribute to the $\tau$-function, since it is residueless. Furthermore, it is easy to see that
        \begin{equation*}
            \frac{1}{\hbar}\tr \left[K_3(\lambda)\frac{\partial \hat{\Theta}}{\partial \eta}\right] = \frac{1}{\hbar}\tr \left[K_3(\lambda)\frac{\partial \hat{\Theta}}{\partial \mu}\right] = \frac{1}{\hbar}\tr \left[K_3(\lambda)\frac{\partial \hat{\Theta}}{\partial \nu}\right] = \OO(1), \qquad \hbar\to 0.
        \end{equation*}
    So, the leading contribution to the $\tau$-function arises from the term involving $K_1(\lambda)$. By direct calculation, we obtain that
        \begin{align*}
            -\frac{1}{\hbar}\Res_{\lambda=\infty}\tr \left[K_1(\lambda)\frac{\partial \hat{\Theta}}{\partial \eta}\right] &= \frac{1}{2\hbar^2}h_5^{(0)},\\
            -\frac{1}{\hbar}\Res_{\lambda=\infty}\tr \left[K_1(\lambda)\frac{\partial \hat{\Theta}}{\partial \mu}\right] &= \frac{1}{2\hbar^2}h_2^{(0)},\\
            -\frac{1}{\hbar}\Res_{\lambda=\infty}\tr \left[K_1(\lambda)\frac{\partial \hat{\Theta}}{\partial \nu}\right] &= \frac{1}{2\hbar^2}h_1^{(0)}.
        \end{align*}
    All terms of subleading order in $\hbar$ then arise from $K_3(\lambda)$. From our calculations in Proposition \ref{R-expansion-prop} (see in particular Equation 
    \eqref{R1-sol}), we have that
        \begin{align*}
            {\bf R}^{-1}(\lambda){\bf R}'(\lambda) &= \hbar {\bf R}_1'(\lambda) + \OO(\hbar^2), \qquad \hbar\to 0, \qquad \text{and}\\
            \hbar {\bf R}_1'(\lambda) &= -\frac{\hbar (W_1+\hat{W}_1)}{\lambda^2} - \frac{2\hbar(\beta W_1 +\alpha\hat{W_1})}{\lambda^3}+ \OO(\lambda^{-4}),\qquad \lambda\to \infty.
        \end{align*}
    On the other hand, as $\lambda\to \infty$,
        \begin{align*}
            M(\lambda)e^{\frac{1}{\hbar}(G-\hat{\Theta})(\lambda)}\frac{\partial \hat{\Theta}}{\partial \nu}e^{-\frac{1}{\hbar}(G-\hat{\Theta})(\lambda)}M^{-1}(\lambda) &=M(\lambda)\frac{\partial \hat{\Theta}}{\partial \nu}M^{-1}(\lambda)  = E_{13}\lambda + \OO(1),\\
            M(\lambda)e^{\frac{1}{\hbar}(G-\hat{\Theta})(\lambda)}\frac{\partial \hat{\Theta}}{\partial \mu}e^{-\frac{1}{\hbar}(G-\hat{\Theta})(\lambda)}M^{-1}(\lambda) &=M(\lambda)\frac{\partial \hat{\Theta}}{\partial \mu}M^{-1}(\lambda) = (E_{12}+E_{23})\lambda + \OO(1),\\
            M(\lambda)e^{\frac{1}{\hbar}(G-\hat{\Theta})(\lambda)}\frac{\partial \hat{\Theta}}{\partial \eta}e^{-\frac{1}{\hbar}(G-\hat{\Theta})(\lambda)}M^{-1}(\lambda) &=M(\lambda)\frac{\partial \hat{\Theta}}{\partial \eta}M^{-1}(\lambda)\\
            &= (E_{12}+E_{23})\lambda^2 + 
            \begin{psmallmatrix}
                -\frac{1}{4}\varsigma & \frac{\mu}{5\eta-3\varsigma} & \frac{5}{16}\varsigma^2\\
                0 & \frac{1}{2}\varsigma & \frac{\mu}{5\eta-3\varsigma}\\
                1 & 0 & -\frac{1}{4}\varsigma
            \end{psmallmatrix}\lambda
            +\OO(1),\\
        \end{align*}
so that
    \begin{align*}
        \frac{1}{\hbar}\Res_{\lambda=\infty}\tr \left[K_3(\lambda)\frac{\partial \hat{\Theta}}{\partial \nu}\right] &=-\tr \left(E_{13}(W_1+\hat{W}_1)\right) + \OO(\hbar),\\
        \frac{1}{\hbar}\Res_{\lambda=\infty}\tr \left[K_3(\lambda)\frac{\partial \hat{\Theta}}{\partial \mu}\right] &=-\tr \left((E_{12}+E_{23})(W_1+\hat{W}_1)\right)+ \OO(\hbar),\\
        \frac{1}{\hbar}\Res_{\lambda=\infty}\tr \left[K_3(\lambda)\frac{\partial \hat{\Theta}}{\partial \eta}\right] &=-2\tr\left((E_{12}+E_{23})(\beta W_1+\alpha\hat{W}_1)\right)
        - \tr\left((W_1+\hat{W}_1)\begin{psmallmatrix}
                -\frac{1}{4}\varsigma & \frac{\mu}{5\eta-3\varsigma} & \frac{5}{16}\varsigma^2\\
                0 & \frac{1}{2}\varsigma & \frac{\mu}{5\eta-3\varsigma}\\
                1 & 0 & -\frac{1}{4}\varsigma
            \end{psmallmatrix}\right)+ \OO(\hbar).
    \end{align*}
Evaluation of these residues gives that
    \begin{align*}
        -\frac{1}{\hbar}\Res_{\lambda=\infty}\tr \left[K_3(\lambda)\frac{\partial \hat{\Theta}}{\partial \nu}\right] &= \frac{1}{2}h^{(1)}_1 + \OO(\hbar),\\
        -\frac{1}{\hbar}\Res_{\lambda=\infty}\tr \left[K_3(\lambda)\frac{\partial \hat{\Theta}}{\partial \mu}\right] &= \frac{1}{2}h^{(1)}_2 + \OO(\hbar),\\
        -\frac{1}{\hbar}\Res_{\lambda=\infty}\tr \left[K_3(\lambda)\frac{\partial \hat{\Theta}}{\partial \eta}\right] &= \frac{1}{2}h^{(1)}_5 + \OO(\hbar).
    \end{align*}
\end{proof}

An almost immediate consequence of the above results is that the $\tau$-differential admits a topological expansion.
\begin{cor}
    Theorem \ref{Theorem:topologicalexpansion} holds.
\end{cor}
\begin{proof}
    From the proof of Theorem \ref{MainTheorem1}, we see that ${\bf d}\log \tau(\hbar^{-2/7}\eta,\hbar^{-5/7}\mu,\hbar^{-6/7}\nu) \sim \sum_{k=0}^{\infty} \omega_k \hbar^k$, for some differentials $\omega_k$ independent of $\hbar$. Since $u,v$ are expressible as derivatives of $\log \tau$, it follows that
        \begin{equation}\label{initial-ansatz}
             u(\eta,\mu,\nu|\hbar) \sim \sum_{k=0}^{\infty} a_k(\eta,\mu,\nu)\hbar^{k},\qquad\qquad v(\eta,\mu,\nu|\hbar) \sim \sum_{k=0}^{\infty} b_k(\eta,\mu,\nu)\hbar^{k},
        \end{equation}
    for some functions $a_k(\eta,\mu,\nu)$, $b_k(\eta,\mu,\nu)$. Our goal is to show that
        \begin{equation*}
            a_{2\ell+1}(\eta,\mu,\nu) = b_{2\ell+1}(\eta,\mu,\nu) = 0, \qquad \ell\geq 0;
        \end{equation*}
    this would imply that $u,v$ have asymptotic expansions in powers of $\hbar^{2}$, and thus the formal calculations of Appendix \ref{Appendix:perturbative-expansion} are valid.
    Inserting the expansions \eqref{initial-ansatz} into the rescaled string equation \eqref{rescaled-string}, an explicit calculation shows that
        \begin{equation*}
            a_1 = a_3 = 0,\qquad\qquad b_1 = b_3 = 0.
        \end{equation*}
    In general, at order $\hbar^N$ for  $N\geq 4$ we obtain from the string equation that
        \begin{align*}
            0 &=\sum_{\ell = 0}^{N}\left( 18b_{\ell}b_{N-\ell} -  15\eta a_{\ell}a_{N-\ell} + 6\sum_{j=0}^{\ell} a_{N-\ell}a_{\ell-j}a_{j} \right) - \frac{9}{2}\sum_{\ell=0}^{N-2} \left(2a_{\ell}''a_{N-\ell-2} -a'_{\ell}a'_{N-\ell-2}\right) + 5\eta a_{N-2}''+ a^{(4)}_{N-4},\\
            0 &= -3\sum_{k=0}^Nb_ka_{N-k} + 5\eta b_N + b_{N-2}''.
        \end{align*}
    Let $N\geq 4$ be odd, and assume the inductive hypothesis that $a_{2\ell+1} = b_{2\ell+1} = 0$ for all $0\leq \ell<\frac{N-1}{2}$.
    Observe the following facts:
            \begin{itemize}
                \item If $N$ is odd and $\ell\in \ZZ_+$, then either $N-\ell$ is odd or $\ell$ is odd,
                \item If $N$ is odd, $\ell,j\in \ZZ_+$, then one of $N-\ell$, $\ell-j$, $j$ is odd.
            \end{itemize}
   The above facts imply that many terms in the order $\hbar^{N}$ equations vanish, provided the inductive hypothesis holds. 
   The order $\hbar^{N}$ equations then simplify to
        \begin{align*}
            0 &= 3b_0 b_{N} -\frac{1}{2}a_0(5\eta-3a_0)a_N,\\ 
            0 &= -3b_0a_N - 3b_Na_0 + 5\eta b_N.
        \end{align*}
   We can express this system as the matrix equation
    \begin{equation*}
        \begin{pmatrix}
            -\frac{1}{2} a_0(5\eta-3a_0) & 3b_0\\
            -3b_0 & 5\eta-3a_0
        \end{pmatrix}
        \begin{pmatrix}
            a_N\\
            b_N
        \end{pmatrix} = 
        \begin{pmatrix}
            0\\
            0
        \end{pmatrix},
    \end{equation*}
    and so $a_N = b_N = 0$ if the matrix above is invertible. Indeed, since $a_0 =\varsigma(\eta,\mu,\nu), b_0 = -\frac{2\mu}{5\eta-3\varsigma(\eta,\mu,\nu)}$, we see that
        \begin{equation*}
            \det \begin{pmatrix}
            -\frac{1}{2} a_0(5\eta-3a_0) & 3b_0\\
            -3b_0 & 5\eta-3a_0
        \end{pmatrix} = -\frac{1}{2}a_0(5\eta-3a_0)^2 + 9b_0^2 = (5\eta-3\varsigma)\frac{\partial \mathcal{P}}{\partial \varsigma} \neq 0,
        \end{equation*}
    by the definition of the region $D$, and so we find that $a_N = b_N = 0$. This completes the proof.
\end{proof}

\section{Painlev\'{e} I asymptotics.}
As we observed in Section \ref{section:g-function}, as $(\eta,\mu,\nu)$ tend to a point on the critical surface, the spectral curve
degenerates, and we expect to see Painlev\'{e} I-type asymptotics. In this section, we will address the double-scaling limit of the
Riemann-Hilbert problem for $\Psi$ in the vicinity of this Painlev\'{e} I regime. For simplicity, from here on we address criticality
\textit{only in the case when}
    \begin{equation}
        \mu = 0, \qquad\qquad \eta \neq 0.
    \end{equation}
There are thus two interesting critical regimes to consider:
    \begin{equation}
        \eta >0,\quad \nu = \frac{125}{108}\eta^3, \qquad\qquad \text{ and } \qquad\qquad \eta<0, \quad \nu = 0.
    \end{equation}
These cases correspond to the critical curves $\gamma_+$, $\gamma_-$ introduced earlier.
In both cases, as is typical for double-scaling limits of this type (cf. \cite{DK0}, for instance), a modified spectral curve must 
be constructed. From then on the Deift-Zhou analysis is effectively identical to what appears in the previous section, up until the
construction of local parametrices. We calculate the modified spectral curve for the cases $\eta >0$, $\eta < 0$ in subsection \ref{modified-spectral-curve}.

For all intents and purposes, one can think that the essential differences between this section and the previous one are as follows:
    \begin{itemize}
        \item An appearance of the function $g_j(\lambda)$ in a transformation in Section \S2 is replaced by the equivalent transformation here, with $\hat{g}_j(\lambda)$
        substituted for $g_j(\lambda)$,
        \item The local parametrices are no longer of Airy type, and must be constructed by other means.
    \end{itemize}
The modifications necessary are by now well-established in the literature: we must construct a modified spectral curve, and use local 
parametrices which involve the Painlev\'{e} I equation. 

\subsection{Construction of the modified spectral curve(s).}\label{modified-spectral-curve}
Here, we define modified spectral curves for use in the calculation of critical asymptotics of ${\bf Z}$. The modified curves are slightly different in the $\eta > 0$ and $\eta < 0$ regimes, as the local degeneration of the spectral curve to these two regimes is different there,
but both are defined in effectively the same manner as in previous literature \cite{DK0,DG,DHL3}.

Given a point $(\eta_0,\nu_0)$ on one of the critical curves, we define (real) analytic functions $\hat{\eta}(\hbar)$, $\hat{\nu}(\hbar)$ of $\hbar^{4/5}$ such that, for $\hbar$ sufficiently small,
        \begin{equation*}
        \hat{\eta}(\hbar) = \eta_0 + \OO(\hbar^{4/5}),\qquad\qquad \hat{\nu}(\delta) = \nu_0 + \OO(\hbar^{4/5}),
    \end{equation*}

We will later take $\hat{\eta}(\hbar),\hat{\nu}(\hbar)$ to be specific analytic functions in each case. For now however, the above
definition is sufficient.

\begin{defn}
    Let $\hat{\eta}(\hbar)$, $\hat{\nu}(\hbar)$ be as above. We define the \textit{modified spectral curve} $\mathcal{S}$ about the point $(\eta_0,\nu_0)$ to be
        \begin{align}
            \hat{\lambda}(u) &= 
            \begin{cases}
                u^3-\frac{5}{2}\eta_0 u, & \eta_0>0,\\
                u^3, & \eta_0<0,
            \end{cases}\\
            \hat{Y}(u) &= 
            \begin{cases}
                u^4-\frac{5}{3}\left[2\eta_0-\hat{\eta}(\hbar)\right] u^2 - \frac{25}{18}\left[2\hat{\eta}(\hbar)-\eta_0\right] +\frac{1}{3}\frac{\hat{\nu}(\hbar) -\frac{125}{108}\left[3\hat{\eta}(\hbar) - 2\eta_0\right]}{u^2-\frac{5}{6}\eta_0}, & \eta_0 >0,\\
                u^4+\frac{5}{3}\hat{\eta}(\hbar) u^2+\frac{1}{3}\hat{\nu}(\hbar) u^{-2}, & \eta_0 < 0.
            \end{cases}
        \end{align}
\end{defn}

We will consider the modified spectral curve $\mathcal{S}$ as a branched covering of the plane over the $\hat{\lambda}$-coordinate. Set
    \begin{equation}
        \alpha :=\sqrt{5\eta_0/6}.
    \end{equation}
The sheets of the modified curve we label as $\mathcal{S} := \mathcal{S}_1 \sqcup \mathcal{S}_2 \sqcup \mathcal{S}_3$, where
    \begin{align}
        \mathcal{S}_1 &= 
        \begin{cases}
            \CC \setminus (-\infty,-\alpha], & \eta_0 >0,\\
            \CC \setminus (-\infty,0], & \eta_0 <0,
        \end{cases}\\
        \mathcal{S}_2 &= 
        \begin{cases}
            \CC \setminus (-\infty,-\alpha ] \cup [\alpha,\infty), & \eta_0 >0,\\
            \CC \setminus \RR, & \eta_0 <0,
        \end{cases}\\
        \mathcal{S}_3 &= 
        \begin{cases}
            \CC \setminus [\alpha,\infty), & \eta_0 >0,\\
            \CC \setminus [0,\infty), & \eta_0 <0.
        \end{cases}
    \end{align}
On each sheet $j=1,2,3$, we define a uniformization coordinate $u_j(\lambda)$ which resolves the function $\hat{\lambda}(u)$: in other words, $\hat{\lambda}(u_j(z)) = z$ for $z\in \mathcal{S}_j$. These functions are uniquely determined by their asymptotic expansion on each sheet:
    \begin{align}
        u_1(\lambda) &= \lambda^{1/3}[1 + \OO(\lambda^{-1/3})], \qquad \lambda\to \infty,\\
        u_2(\lambda) &= 
        \begin{cases}
            \omega^2 \lambda^{1/3}[1 + \OO(\lambda^{-1/3})], & \lambda\to \infty,\quad \text{Im } \lambda>0,\\
            \omega \lambda^{1/3}[1 + \OO(\lambda^{-1/3})], & \lambda\to \infty,\quad \text{Im } \lambda<0,
        \end{cases}\\
        u_3(\lambda) &=
        \begin{cases}
            \omega \lambda^{1/3}[1 + \OO(\lambda^{-1/3})], & \lambda\to \infty,\quad \text{Im } \lambda>0,\\
            \omega^2 \lambda^{1/3}[1 + \OO(\lambda^{-1/3})], & \lambda\to \infty,\quad \text{Im } \lambda<0.
        \end{cases}
    \end{align}
Note that in the $\eta_0 < 0$ case, the $\OO(\lambda^{-1/3})$ terms in the above are identically zero, and we have an exact expression for the 
uniformizing coordinates.

Finally, we define the modified $g$-function by first setting 
    \begin{equation}
        \hat{g}(u) := \int \hat{Y}(u) \hat{\lambda}'(u) du,
    \end{equation}
and putting
    \begin{equation}
        \hat{g}_j(\lambda;\hbar) \equiv \hat{g}_j(\lambda) := \hat{g}(u_j(\lambda)), \qquad j = 1,2,3.
    \end{equation}

We now have the following proposition, which we state without proof:
    \begin{prop}
        As $\lambda\to \infty$ on each sheet of $\mathcal{S}$, we have the asymptotics
            \begin{equation}
                \hat{g}_j(\lambda) = \hat{\Theta}_{jj}(\lambda;\hat{\eta},0,\hat{\nu}) + \OO(\lambda^{-1/3}).
            \end{equation}
        Furthermore, when $\hbar = 0$,
            $\hat{g}_j(\lambda;0) = g_j(\lambda)$, the corresponding critical $g$-function.
    \end{prop}
We indicate that the details of this computation are effectively identical to those found in Proposition \eqref{prop:true-g-function}.

Crucially, since $\hbar$ will be taken to be sufficiently small, the inequalities that we require for lensing hold in the whole 
complex plane, apart from small discs centered at $\pm \alpha$.
\begin{prop} \textit{Analog of Proposition \ref{LensingProposition}.}
    Fix $\epsilon>0$, and let $\hbar$ sufficiently small so that $|\hat{\eta}(\hbar) - \eta_0| < \epsilon$. Then, there exists $R_{\epsilon} = R_{\epsilon}(\eta_0)>0$
    such that, if we define
        \begin{equation}
           D_{\pm} := \{\lambda \mid |\lambda \mp \alpha| <R_{\epsilon}\},
        \end{equation}
    the following inequalities hold:
    \begin{equation}
                    \begin{cases}
                        (a.)\qquad \text{Re}[\hat{g}_3(\lambda) - \hat{g}_2(\lambda)] > 0, & \lambda \in \hat{\Gamma}_5\setminus D_{+},\\
                        (b.)\qquad \text{Re}[\hat{g}_2(\lambda) - \hat{g}_1(\lambda)] > 0, & \lambda \in \hat{\Gamma}_{-3}\setminus D_{-},\\
                        (c.)\qquad \text{Re}[\hat{g}_3(\lambda) - \hat{g}_2(\lambda)] < 0, & \text{in a lens around $[\alpha,\infty)$}\setminus D_{+},\\
                        (d.)\qquad \text{Re}[\hat{g}_2(\lambda) - \hat{g}_1(\lambda)] < 0, & \text{in a lens around $(-\infty,-\alpha]$}\setminus D_{-}.
                    \end{cases}
                \end{equation}
    Furthermore, if $R_{\epsilon}$ is taken to be the minimal such radius so that the above inequalities hold for given $\epsilon$, then $R_{\epsilon}\to 0$ as $\epsilon\to 0$.
\end{prop}

\subsection{Critical PI asymptotics: $\eta>0$.} \label{subsection:eta-plus}
In this subsection, we perform the steepest descent analysis for ${\bf Z}(\lambda)$ for $(\eta,\mu,\nu)\in \gamma_+$. Much of this analysis is similar to what is found
Section \S2, and so we will omit most details. 

Put 
    \begin{equation}
        \hat{G}(\lambda) := \text{diag }(\hat{g}_1(\lambda),\hat{g}_2(\lambda),\hat{g}_3(\lambda)),
    \end{equation}
let $\mathfrak{a} \equiv \mathfrak{a}(\hat{\eta},\hat{\nu}) := -\frac{5}{1296\hbar}\eta_0\left[125\eta_0^3-250\eta_0^2\hat{\eta}+216\hat{\nu}\right]$, and set
    \begin{equation}
        \mathfrak{h}^{(0)} := 
            \begin{pmatrix}
                1 & \mathfrak{a} & \frac{1}{2}\mathfrak{a}^2\\
                0 & 1 & \mathfrak{a}\\
                0 & 0 & 1
            \end{pmatrix}.
    \end{equation}
We then set
    \begin{equation}
        {\bf U}(\lambda) := (\mathfrak{h}^{(0)})^{-1}{\bf Z}(\lambda)e^{-\frac{1}{\hbar} G(\lambda)}.
    \end{equation}
Next, we must perform a lens-opening.
Define two lens-shaped domains about $\lambda = \pm \alpha$, which open symmetrically about $\arg (\lambda +\alpha) = \pi$ and $\arg (\lambda -\alpha) = 0$, respectively. As we have done in the previous section, we label these regions $\Delta_{\pm \alpha}$,
and let $\Delta_{-\alpha}^{\pm}$, $\Delta_{\alpha}^{\pm}$ be the connected components of these domains in the upper (resp. lower) half planes, and let
$\Gamma_{-\alpha}^{\pm}$, $\Gamma_{\alpha}^{\pm}$ denote the boundary component of $\Delta_{-\alpha}^{\pm}$, $\Delta_{\alpha}^{\pm}$ which does not include the real line.

Define $2\times 2$ matrices
    \begin{equation*}
        v_{\alpha}(\lambda) := \begin{psmallmatrix}
            1 & 0\\
            -e^{\frac{1}{\hbar}(\hat{g}_3-\hat{g}_2)(\lambda)} & 1
        \end{psmallmatrix}, \qquad\qquad 
        v_{-\alpha}(\lambda) := \begin{psmallmatrix}
            1 & 0\\
            -e^{\frac{1}{\hbar}(\hat{g}_2-\hat{g}_1)(\lambda)} & 1
        \end{psmallmatrix}.
    \end{equation*}
Note that these matrices are analytic and invertible in the domains $\Delta_{\alpha}^{\pm}$, $\Delta_{-\alpha}^{\pm}$, respectively. Now, put
    \begin{equation*}
        {\bf V}(\lambda) :=
        {\bf U}(\lambda) \cdot
            \begin{cases}
                v_{-\alpha}(\lambda) \oplus 1, & \lambda \in \Delta_{-\alpha}^{+},\\
                v_{-\alpha}^{-1}(\lambda) \oplus 1, & \lambda \in \Delta_{-\alpha}^{-},\\
                1 \oplus v_{\alpha}^{-1}(\lambda), & \lambda \in \Delta_{\alpha}^{+},\\
                1 \oplus v_{\alpha}(\lambda), & \lambda \in \Delta_{\alpha}^{-},\\
                \mathbb{I}, & \textit{ otherwise.}
            \end{cases}
    \end{equation*}
The following proposition then follows immediately:
\begin{prop} \textit{(Analog of Proposition \ref{V-prop}).}
    ${\bf V}(\lambda)$ satisfies the following RHP.
        \begin{equation*}
            {\bf V}_+(\lambda) = {\bf V}_-(\lambda) \cdot
                \begin{cases}
                    \mathbb{I} + E_{23}e^{-\frac{1}{\hbar}(\hat{g}_3-\hat{g}_2)(\lambda)}, & \lambda \in \hat{\Gamma}_5,\\
                    \mathbb{I} + E_{12}e^{-\frac{1}{\hbar}(\hat{g}_2-\hat{g}_1)(\lambda)}, & \lambda \in \hat{\Gamma}_{-3},\\
                    \mathbb{I} - E_{32}e^{\frac{1}{\hbar}(\hat{g}_3-\hat{g}_2)(\lambda)}, & \lambda \in \Gamma^{\pm}_{\alpha},\\
                    \mathbb{I} - E_{21}e^{\frac{1}{\hbar}(\hat{g}_2-\hat{g}_1)(\lambda)}, & \lambda \in \Gamma^{\pm}_{-\alpha},\\
                    (-i\sigma_2)\oplus 1, & \lambda \in (-\infty,-\alpha],\\
                    1\oplus(-i\sigma_2), & \lambda \in [\alpha,\infty).
                \end{cases}
        \end{equation*}
    Furthermore, ${\bf V}(\lambda)$ is normalized as
        \begin{equation*}
            {\bf V}(\lambda) = \left[\mathbb{I} + \OO(\lambda^{-1})\right]\hat{f}(\lambda).
        \end{equation*}
\end{prop}
From here, we must find an parametrix to bring ${\bf V}(\lambda)$ to the form of a small-norm Riemann-Hilbert problem. Outside of small discs centered as 
$\lambda = \pm \alpha$, we can use as a parametrix the matrix-valued analytic function $M(\lambda)$ which we constructed in Proposition \ref{prop:global-parametrix-solution} as the solution to the Riemann-Hilbert problem \ref{prob:global-parametrix}. Furthermore, if we construct a local parametrix $P(\lambda)$ at $\lambda = -\alpha$,
we can make use of the symmetry of $M(\lambda)$ to construct a local parametrix at $\lambda = +\alpha$:
\begin{equation}
        P_{\alpha}(\lambda) = 
        \begin{psmallmatrix}
                1 & 0 & 0\\
                0 & -1 & 0\\
                0 & 0 & 1
            \end{psmallmatrix}P(-\lambda)
            \begin{psmallmatrix}
                0 & 0 & 1\\
                0 & -1 & 0\\
                1 & 0 & 0
            \end{psmallmatrix}.
    \end{equation}
Note that since $\mu = 0$ here, there is no need to interchange $\mu\leftrightarrow-\mu$. Thus, the task at hand is to construct a local parametrix at $\lambda = -\alpha$.
For this, we will need the following model Riemann-Hilbert problem, which characterizes tronqu\'{e}e solutions to the Painlev\'{e} I equation \cite{Kapaev}:
    \begin{problem}\label{prob:PainleveI}
    \textit{Tronqu\'{e}e Painlev\'{e} I Problem}. Given $x,\varkappa\in \CC$, construct a $2\times 2$-matrix valued piecewise analytic function
    $\boldsymbol{\Phi}(\zeta;x)$ in $\CC\setminus ( \cup_{|j|=1}^2 L_j \cup (-\infty,0])$, where $L_j = \{\zeta | \arg\zeta = \frac{2\pi}{5}j\}$,
    and all rays are oriented away from the origin, such that
        \begin{equation}
            \boldsymbol{\Phi}_+(\zeta;x) = \boldsymbol{\Phi}_-(\zeta;x)\cdot
            \begin{cases}
                \begin{pmatrix}
                    0 & -1\\
                    1 & 0
                \end{pmatrix}, &\zeta \in (-\infty,0],\\
                \begin{pmatrix}
                    1 & 0\\
                    -1 & 1
                \end{pmatrix}, &\zeta \in L_{\pm 2},\\
                \begin{pmatrix}
                    1 & \varkappa\\
                    0 & 1
                \end{pmatrix}, & \zeta \in L_{1},\\
                \begin{pmatrix}
                    1 & 1-\varkappa\\
                    0 & 1
                \end{pmatrix}, & \zeta \in L_{-1},
            \end{cases}
        \end{equation}
        and subject to the normalization condition
            \begin{equation}
                {\bf \Phi}(\zeta;x) = \frac{\zeta^{\frac{1}{4}\sigma_3}}{\sqrt{2}} 
                \begin{pmatrix}
                    1 & -i\\
                    -i & 1
                \end{pmatrix}\left[\mathbb{I} + \sum_{k=1}^{\infty}\frac{\Phi_k(x)}{\zeta^{k/2}}\right]e^{\left[\frac{4}{5}\zeta^{5/2}+x\zeta^{1/2}\right]\sigma_3}.
            \end{equation}
    \end{problem}
It is known that the RHP \ref{prob:PainleveI} characterizes the tronqu\'{e}e solutions $q_{\kappa}(x)$ of the Painlev\'{e} I equation
    \begin{equation}
        q''(x) =6q(x)^2 + x,
    \end{equation}
which satisfy $q_{\varkappa}(x) = \sqrt{-x/6}[1 + \OO((-x)^{5/2})]$, $x\to \infty$. The first few matrices $\Phi_k(x)$ are given in terms of
the solution $q_{\varkappa}$:
    \begin{equation}
        \Phi_1(x) = 
        \begin{pmatrix}
            -\mathcal{H}_{\varkappa}(x) & 0\\
            0 & \mathcal{H}_{\varkappa}(x)
        \end{pmatrix},\qquad\qquad 
        \Phi_2(x) = \frac{1}{2}
        \begin{pmatrix}
            \mathcal{H}_{\varkappa}^2(x) & q_{\varkappa}(x)\\
            q_{\varkappa}(x) & \mathcal{H}_{\varkappa}^2(x)
        \end{pmatrix},
    \end{equation}
where $\mathcal{H}_{\varkappa}(x) := \frac{1}{2}[q_{\varkappa}'(x)]^2-2q_{\varkappa}^3(x) - xq_{\varkappa}(x)$ is the Painlev\'{e} I Hamiltonian.
The parameter $\varkappa$ appears in the exponentially small corrections to the asymptotics 
of $q_{\varkappa}(x)$ \cite{Kapaev}. When the parameter $\varkappa = 0$ or $\varkappa = 1$, the corresponding solution is one of the \textit{tritronqu\'{e}e} 
solution of Painlev\'{e} I. This Riemann-Hilbert problem admits a solution, provided that $x$ is not a pole of the corresponding Painlev\'{e} transcendent.
Because of our particular setup, we will be interested in this parametrix when
    \begin{equation}
        \varkappa = 1,
    \end{equation}
and so from here on we assume that this is the case and do not mention the parameter $\varkappa$ further. With the definition of this parametrix in place, we are ready to construct the local parametrix $P(\lambda)$.

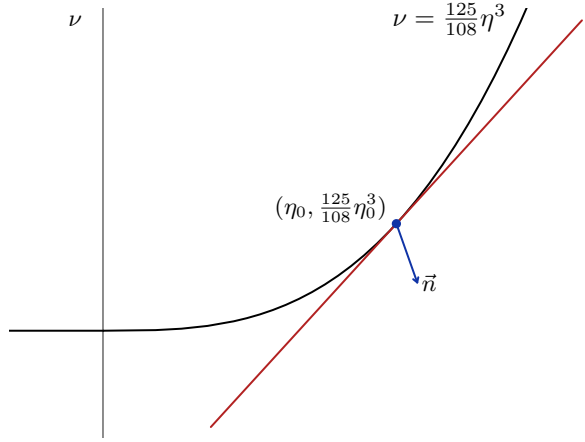
\begin{figure}
    \centering
    \begin{overpic}[scale=0.5]{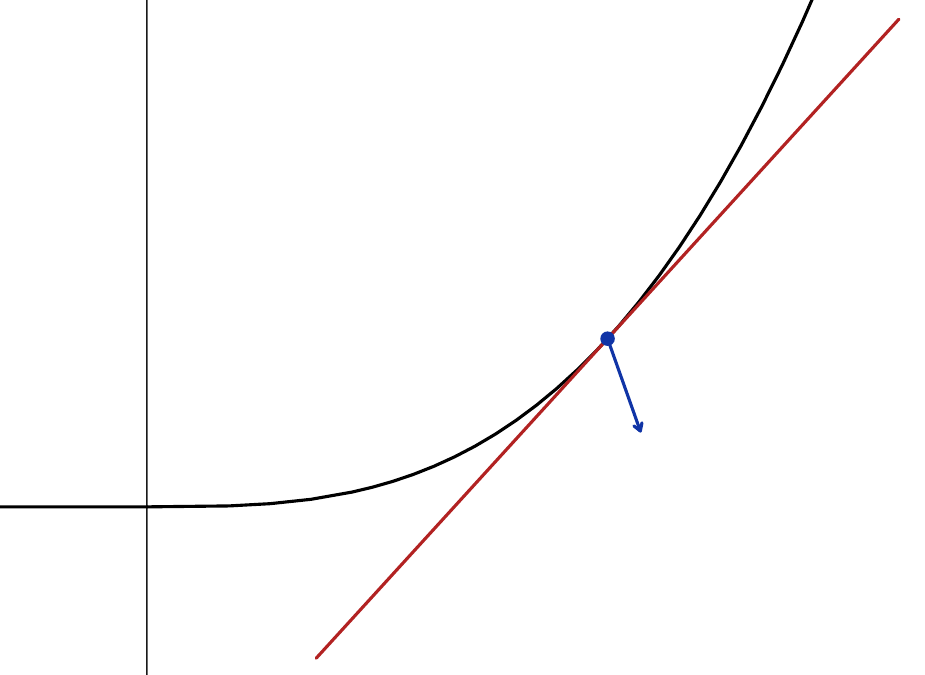}
        \put (65,70) {$\nu = \frac{125}{108}\eta^3$}
        \put (45,38) {\small $(\eta_0,\frac{125}{108}\eta_0^3)$}
        \put (70,25) {\small $\vec{n}$}
        \put (10,70) {\small $\nu$}
    \end{overpic}
    \caption{A point on the curve $\gamma_+$, defined by $\nu = \frac{125}{108}\eta^3$, and its tangent line (shown in red). Any non-tangential
    approach to a point on the critical curve $\gamma_+$, such as the direction $\vec{n}$ shown, leads to Painlev\'{e} I asymptotics.}
    \label{fig:TangentLine}
\end{figure}

Let
    \begin{equation}
        \psi(\lambda;\hbar):= \hat{g}_1(\lambda;\hbar) - \hat{g}_2(\lambda;\hbar),
    \end{equation}
and in the disc $D_-$, we define the conformal map
    \begin{equation}
        \zeta(\lambda) := \left[\frac{5}{8}\psi(\lambda;0)\right]^{2/5}.
    \end{equation}
By construction, $\zeta(\lambda) = \left[\frac{5}{8}\psi(\lambda;0)\right]^{2/5} = \left[\frac{5}{8}\left(g_2(\lambda)-g_1(\lambda)\right)\right]^{2/5}$ (here, $g$ is the \textit{critical} $g$-function),
which by Lemma \ref{lemma:Localexpansions} no. 3, is indeed conformal in a neighborhood of $\lambda = -\alpha$. We also define the
function
    \begin{equation}
        \boldsymbol{x}(\lambda) := \frac{1}{2}\left(\frac{8}{5}\right)^{1/5} \frac{\psi(\lambda;\hbar) - \psi(\lambda;0)}{[\psi(\lambda;0)]^{1/5}},
    \end{equation}
which satisfies the following Proposition:
\begin{prop}
    The function $\boldsymbol{x}(\lambda)$ is analytic in a neighborhood of $\lambda = -\alpha$. Furthermore, if we let $\vec{n} = (n_{\eta},n_{\nu})$ be any any
    unit vector which lies below the tangent line of the curve $\nu = \frac{125}{108}\eta^3$ at $\nu =\frac{125}{108}\eta_0^3$ (see Figure 
    \ref{fig:TangentLine}), and set
        \begin{equation}\label{eta-nu-scalings-plus}
            \hat{\eta}(\hbar) := \eta_0 - C n_{\eta}x\hbar^{4/5}, \qquad\qquad \hat{\nu}(\hbar) := \frac{125}{108}\eta_0^3 - C n_{\nu}x\hbar^{4/5}
        \end{equation}
        for fixed $x\in \CC$, $C>0$. Then, we can choose $C = C(\eta_0,\vec{n})$ so that
        \begin{equation}\label{x-convergence}
            \lim_{\hbar\to 0} \hbar^{-4/5} \boldsymbol{x}(\lambda;\hbar) = f(\lambda;x) = x + \OO(\lambda+\alpha),
        \end{equation}
    where the convergence holds uniformly on compact subsets of $D_-$.
\end{prop}
\begin{proof}
    To see that $\boldsymbol{x}(\lambda)$ is an analytic function, note that, as $\lambda\to \alpha$,
        \begin{align*}
            \psi(\lambda;\hbar) - \psi(\lambda;0) &= k_0(\hbar) (\lambda+\alpha)^{1/2}[1 + \OO(\lambda+\alpha)],\\
            [\psi(\lambda;\vec{0})]^{1/5} &= k_1(\lambda+\alpha)^{1/2}[1 + \OO(\lambda+\alpha)],
        \end{align*}
    where $k_0(\hbar) = \left(\frac{2}{15}\right)^{1/4}\eta_0^2\left[\nu(\hbar)-\frac{125}{108}\eta_0^2\left(3\hat{\eta}(\hbar) - 2\eta_0\right)\right]$, and $k_1 = 2\frac{3^{4/5} 2^{1/20}(15)^{3/4}}{45\eta_0^{1/20}} > 0$. It follows that we can define $\boldsymbol{x}(\lambda;\vec{\delta})$ in a single-valued way.
    Now, making the substitutions of Equation \eqref{eta-nu-scalings-plus}, note that 
        \begin{align*}
            \hbar^{-4/5}\boldsymbol{x}(\lambda) &= \frac{1}{2}\left(\frac{8}{5}\right)^{1/5}\frac{\hbar^{-4/5}k_0(\hbar)}{k_1}[1 + \OO(\lambda+\alpha)]\\
            &= \left[\left(\frac{3}{10}\right)^{1/5}\eta_0^{-1/5}\left(\frac{125}{36}\eta_0^2 n_{\eta}-n_{\nu}\right)\right]x\left[1 + \OO(\lambda+\alpha)\right],
        \end{align*}
   with the remainder term $\OO(\lambda+\alpha)$ being uniformly bounded in $\hbar$ for $|\lambda+\alpha|$ sufficiently small as $\hbar\to 0$.
   Since we took $\vec{n}$ to be a vector approaching the critical point non-tangentially from inside the region of criticality, we can take
    \begin{equation}\label{C-definition}
        C= \left(\frac{10\eta_0}{3}\right)^{1/5}\left(\frac{125}{36}\eta_0^2 n_{\eta}-n_{\nu}\right)^{-1}>0,
    \end{equation}
   and we see that the statement \eqref{x-convergence} holds.
\end{proof}
In other words, we can take any non-tangential approach to the critical point, and obtain the same type of result, provided we rescale constants
in the correct manner.

We then set 
    \begin{equation}
        P(\lambda) := E(\lambda)\left[\boldsymbol{\Phi}(\hbar^{-2/5}\zeta(\lambda),\hbar^{-4/5}\boldsymbol{x}(\lambda)) e^{\frac{1}{2\hbar}[\hat{g}_2(\lambda) - \hat{g}_1(\lambda)]\sigma_3} \oplus 1\right],
    \end{equation}
where $E(\lambda)$ is the holomorphic function
\begin{align}
        E(\lambda):=  M(\lambda)\left[\frac{1}{\sqrt{2}}
        \begin{psmallmatrix}
            1 &i\\ i & 1
        \end{psmallmatrix} [\hbar^{-2/5}\zeta(\lambda)]^{-\frac{1}{4}\sigma_3} \oplus 1\right].
    \end{align}
It is then easy to check that the following Lemma holds:
    \begin{lemma}\label{local-parametrix-lemma-critical1}
        Under the scaling of Equation \eqref{eta-nu-scalings-plus}, as $\hbar \to 0$,
        \begin{equation}
            P(\lambda) \sim M(\lambda)\left[\mathbb{I} + \sum_{k=1}^{\infty} \frac{\left[\Phi_k(\boldsymbol{x}(\lambda)) \oplus 0\right]}{\zeta(\lambda)^{k/2}} \hbar^{-k/5}\right],
        \end{equation}
    where $\Phi_k(x)$ are the matrices appearing in the $\zeta\to \infty$ expansion of $\boldsymbol{\Phi}(\zeta;x)$ in the RHP \ref{prob:PainleveI}.
    \end{lemma}

We then define
    \begin{equation}
        {\bf R}(\lambda) :=
        {\bf V}(\lambda) \cdot
            \begin{cases}
                M^{-1}(\lambda), & \lambda \in \CC\setminus (D_{\pm}),\\
                \begin{psmallmatrix}
                    0 & 0 & 1\\
                    0 & -1 & 0\\
                    1 & 0 & 0
                \end{psmallmatrix}P^{-1}(-\lambda)
                \begin{psmallmatrix}
                    1 & 0 & 0\\
                    0 & -1 & 0\\
                    0 & 0 & 1
                \end{psmallmatrix}, & \lambda \in D_{+},\\
                P^{-1}(\lambda), & \lambda \in D_{-}.
            \end{cases}
    \end{equation}
    
The following proposition then holds:
\begin{prop}
    ${\bf R}(\lambda)$ satisfies the following Riemann-Hilbert problem:
        \begin{equation*}
            {\bf R}_+(\lambda) = {\bf R}_-(\lambda) \cdot
                \begin{cases}
                    \mathbb{I} + M(\lambda)E_{23}M^{-1}(\lambda)e^{-\frac{1}{\hbar}(\hat{g}_3-\hat{g}_2)(\lambda)}, & \lambda \in \hat{\Gamma}_5,\setminus D_{-\alpha}\\
                    \mathbb{I} + M(\lambda)E_{12}M^{-1}(\lambda)e^{-\frac{1}{\hbar}(\hat{g}_2-\hat{g}_1)(\lambda)}, & \lambda \in \hat{\Gamma}_{-3}\setminus D_{+\alpha},\\
                    \mathbb{I} - M(\lambda)E_{32}M^{-1}(\lambda)e^{\frac{1}{\hbar}(\hat{g}_3-\hat{g}_2)(\lambda)}, & \lambda \in \Gamma^{\pm}_{\alpha}\setminus D_{+\alpha},\\
                    \mathbb{I} - M(\lambda)E_{21}M^{-1}(\lambda)e^{\frac{1}{\hbar}(\hat{g}_2-\hat{g}_1)(\lambda)}, & \lambda \in \Gamma^{\pm}_{-\alpha}\setminus D_{-\alpha},\\
                    P(\lambda)M^{-1}(\lambda), & \lambda \in \partial D_{-\alpha},\\
                    \begin{psmallmatrix}
                    1 & 0 & 0\\
                    0 & -1 & 0\\
                    0 & 0 & 1
                \end{psmallmatrix}P(-\lambda)M^{-1}(-\lambda)
                \begin{psmallmatrix}
                    1 & 0 & 0\\
                    0 & -1 & 0\\
                    0 & 0 & 1
                \end{psmallmatrix}, & \lambda \in \partial D_{+\alpha},
                \end{cases}
        \end{equation*}
        Subject to the normalization condition
            \begin{equation}
                {\bf R}(\lambda) = \mathbb{I} + \OO(\lambda^{-1}), \qquad \lambda\to \infty.
            \end{equation}
\end{prop}
Let $\Gamma_{{\bf R}}$ denote the union of the jump contours of ${\bf R}(\lambda)$.
Finally, we can claim that ${\bf R}(\lambda)$ is close to the identity matrix:
\begin{prop}\label{R-expansion-prop-critical-plus}
       For $\hbar$ sufficiently small, $(\hat{\eta},0,\hat{\nu})$ scaled as in Equation \eqref{eta-nu-scalings-plus},
       the Riemann-Hilbert problem for ${\bf R}(\lambda)$ has a solution, and admits the $\hbar\to 0$ asymptotic expansion
        \begin{equation}
            {\bf R}(\lambda) \sim \mathbb{I} + \sum_{k=1}^{\infty} {\bf R}_k(\lambda;x)\hbar^{k/5},
        \end{equation}
    which holds uniformly for $\lambda \in \CC\setminus \Gamma_{{\bf R}}$.
    \end{prop}
\begin{proof}
    As we did in the proof of Proposition \eqref{R-expansion-prop}, we let
        \begin{equation*}
            \boldsymbol{\Delta}(\lambda) := {\bf R}_-^{-1}{\bf R}_+ - \mathbb{I}
        \end{equation*}
    denote the deviation of the jumps of ${\bf R}$ from the identity $\mathbb{I}$. $\boldsymbol{\Delta}(\lambda)$ is exponentially close to 
    the identity as $\hbar\to 0$ on $\Gamma_{{\bf R}} \setminus \left(\partial D_{\pm}\right)$. By Lemma \ref{local-parametrix-lemma-critical1}, we have that
        \begin{equation*}
            \boldsymbol{\Delta}(\lambda) \sim \sum_{k=1}^{\infty} J_k(\lambda) \hbar^{k/5},
        \end{equation*}
    where
        \begin{align*}
            J_k(\lambda) &= 
            \begin{cases}
                \begin{psmallmatrix}
                    1 & 0 & 0\\
                    0 & -1 & 0\\
                    0 & 0 & 1
                \end{psmallmatrix}
                    M(-\lambda)\left[\Phi_k(\boldsymbol{x}(-\lambda)) \oplus 0\right]M^{-1}(-\lambda)
                \begin{psmallmatrix}
                    1 & 0 & 0\\
                    0 & -1 & 0\\
                    0 & 0 & 1
                \end{psmallmatrix}\zeta(-\lambda)^{-k/2}, & \lambda \in \partial D_+,\\
                M(\lambda)\left[\Phi_k(\boldsymbol{x}(\lambda)) \oplus 0\right]M^{-1}(\lambda)\zeta(\lambda)^{-k/2}, & \lambda \in \partial D_-,
            \end{cases}
        \end{align*}
    and $\Phi_k(x)$ are as defined in the asymptotic expansion of the RHP \ref{prob:PainleveI}. From standard theory \cite{DKMVZ} it follows that ${\bf R}$ admits an
    asymptotic series in $\hbar^{1/5}$:
        \begin{equation*}
            {\bf R}(\lambda) \sim \mathbb{I} + \sum_{k=1}^{\infty} {\bf R}_k(\lambda;x)\hbar^{k/5},
        \end{equation*}
    where the functions ${\bf R}_k(\lambda)$ can be determined iteratively as the solutions to certain additive Riemann-Hilbert problems. We
    will need only the expression for the first term for $|\lambda|$ sufficiently large.
    The first term ${\bf R}_1$ is given by
        \begin{equation}\label{R1-sol}
            {\bf R}_1(\lambda;\eta,\mu,\nu) = \frac{W_1(x)}{\lambda +\alpha} + \frac{\hat{W}_1(x)}{\lambda-\alpha}, \qquad \lambda\in \CC\setminus (D_{+}\cup D_{-}),
        \end{equation}
    where
        \begin{align}
            W_1(x) &= \Res_{\lambda = -\alpha} \frac{\left[M(\lambda)\left(\Phi_1(\lambda;x)\oplus 0\right)M^{-1}(\lambda)\right]}{\zeta(\lambda)^{1/2}},\\
            \hat{W}_1(x) &=\begin{psmallmatrix}
                    1 & 0 & 0\\
                    0 & -1 & 0\\
                    0 & 0 & 1
                \end{psmallmatrix}W_1(x)\begin{psmallmatrix}
                    1 & 0 & 0\\
                    0 & -1 & 0\\
                    0 & 0 & 1
                \end{psmallmatrix}.
        \end{align}

\end{proof}

\subsection{Critical PI asymptotics: $\eta<0$.}\label{subsection-eta-minus}
In this subsection, we perform the steepest descent analysis for ${\bf Z}(\lambda)$ for $(\eta,\mu,\nu)\in \gamma_-$. This analysis is again similar to what is found Section \S2 and \S3, and so we will omit most details. 

Put 
    \begin{equation}
        \hat{G}(\lambda) := \text{diag }(\hat{g}_1(\lambda),\hat{g}_2(\lambda),\hat{g}_3(\lambda)),
    \end{equation}
and set
    \begin{equation}
        {\bf U}(\lambda) := {\bf Z}(\lambda)e^{-\frac{1}{\hbar} G(\lambda)}.
    \end{equation}
(Note that the gauge matrix $\mathfrak{h}^{(0)}$ that appeared in the previous sections is just the 
identity matrix here). Next, we must perform a lens-opening.
We define two lens-shaped regions opening from $\lambda = 0$, which open symmetrically about $\arg \lambda= \pi$ and $\arg \lambda = 0$, respectively. We label these regions $\Delta_{\mp 0}$,
and let $\Delta_{-0}^{\pm}$, $\Delta_{+0}^{\pm}$ be the connected components of these domains in the upper (resp. lower) half planes, and let
$\Gamma_{-0}^{\pm}$, $\Gamma_{+0}^{\pm}$ denote the boundary component of $\Delta_{-0}^{\pm}$, $\Delta_{+0}^{\pm}$ which does not include the real line.

Define $2\times 2$ matrices
    \begin{equation*}
        v_{+0}(\lambda) := \begin{psmallmatrix}
            1 & 0\\
            -e^{\frac{1}{\hbar}(\hat{g}_3-\hat{g}_2)(\lambda)} & 1
        \end{psmallmatrix}, \qquad\qquad 
        v_{-0}(\lambda) := \begin{psmallmatrix}
            1 & 0\\
            -e^{\frac{1}{\hbar}(\hat{g}_2-\hat{g}_1)(\lambda)} & 1
        \end{psmallmatrix}.
    \end{equation*}
These matrices are analytic and invertible in $\Delta_{+0}^{\pm}$, $\Delta_{-0}^{\pm}$, respectively. We then set
    \begin{equation*}
        {\bf V}(\lambda) :=
        {\bf U}(\lambda) \cdot
            \begin{cases}
                v_{-0}(\lambda) \oplus 1, & \lambda \in \Delta_{-0}^{+},\\
                v_{-0}^{-1}(\lambda) \oplus 1, & \lambda \in \Delta_{-0}^{-},\\
                1 \oplus v_{+0}^{-1}(\lambda), & \lambda \in \Delta_{+0}^{+},\\
                1 \oplus v_{+0}(\lambda), & \lambda \in \Delta_{+0}^{-},\\
                \mathbb{I}, & \textit{ otherwise.}
            \end{cases}
    \end{equation*}
The following proposition then follows immediately:
\begin{prop} \textit{(Analog of Proposition \ref{V-prop}).}
    ${\bf V}(\lambda)$ satisfies the following RHP.
        \begin{equation*}
            {\bf V}_+(\lambda) = {\bf V}_-(\lambda) \cdot
                \begin{cases}
                    \mathbb{I} + E_{23}e^{-\frac{1}{\hbar}(\hat{g}_3-\hat{g}_2)(\lambda)}, & \lambda \in \hat{\Gamma}_5,\\
                    \mathbb{I} + E_{12}e^{-\frac{1}{\hbar}(\hat{g}_2-\hat{g}_1)(\lambda)}, & \lambda \in \hat{\Gamma}_{-3},\\
                    \mathbb{I} - E_{32}e^{\frac{1}{\hbar}(\hat{g}_3-\hat{g}_2)(\lambda)}, & \lambda \in \Gamma^{\pm}_{+0},\\
                    \mathbb{I} - E_{21}e^{\frac{1}{\hbar}(\hat{g}_2-\hat{g}_1)(\lambda)}, & \lambda \in \Gamma^{\pm}_{-0},\\
                    (-i\sigma_2)\oplus 1, & \lambda \in \RR_-,\\
                    1\oplus(-i\sigma_2), & \lambda \in \RR_+.
                \end{cases}
        \end{equation*}
    Furthermore, ${\bf V}(\lambda)$ is normalized as
        \begin{equation*}
            {\bf V}(\lambda) = \left[\mathbb{I} + \OO(\lambda^{-1})\right]\hat{f}(\lambda).
        \end{equation*}
\end{prop}
From here, one typically searches for a parametrix to bring ${\bf V}(\lambda)$ to the form of a ``small-norm'' Riemann-Hilbert problem. However, we shall
need one more preliminary transformation before proceeding to the search for a parametrix. Let $\Delta_{5}$ denote the sector with acute angle opening which
is bounded by the rays $\hat{\Gamma}_5$ and $\Gamma_{-0}^{+}$, and let $\Delta_{-3}$ denote the sector with acute angle opening which is bounded by 
the rays $\hat{\Gamma}_{-3}$ and $\Gamma_{+0}^{-}$. These domains are depicted in Figure \ref{fig:V-jumps-R-jumps} (a). We set
    \begin{equation}
        \hat{{\bf V}}(\lambda) := 
        {\bf V}(\lambda)\cdot\begin{cases}
            \mathbb{I}-E_{21}e^{\frac{1}{\hbar}(\hat{g}_2-\hat{g}_1)(\lambda)} -E_{23}e^{-\frac{1}{\hbar}(\hat{g}_3-\hat{g}_2)(\lambda)}, & \lambda \in \Delta_{5},\\
            \mathbb{I}-E_{12}e^{-\frac{1}{\hbar}(\hat{g}_2-\hat{g}_1)(\lambda)}-E_{32}e^{\frac{1}{\hbar}(\hat{g}_3-\hat{g}_2)(\lambda)}, & \lambda \in \Delta_{-3},\\
            \mathbb{I}, & \textit{ otherwise}.
        \end{cases}
    \end{equation}
Note that the matrices we have multiplied by are exponentially close to the identity matrix in the sectors they are defined in. This follows from Lemma
\ref{Lemma:Lensing}; see also Figure \ref{fig:g-function-signs}. It then follows that $\hat{{\bf V}}(\lambda)$ satisfies the following Riemann-Hilbert
problem:
    \begin{prop}
        $\hat{{\bf V}}(\lambda)$ satisfies the following RHP.
        \begin{equation*}
            \hat{{\bf V}}_+(\lambda) = \hat{{\bf V}}_-(\lambda) \cdot
                \begin{cases}
                    \mathbb{I} - E_{21}e^{\frac{1}{\hbar}(\hat{g}_2-\hat{g}_1)(\lambda)}, & \lambda \in \hat{\Gamma}_5 \cup \Gamma^{-}_{-0},\\
                    \mathbb{I} - E_{32}e^{\frac{1}{\hbar}(\hat{g}_3-\hat{g}_2)(\lambda)}, & \lambda \in \hat{\Gamma}_{-3} \cup \Gamma^{-}_{+0},\\
                    \mathbb{I} + E_{12}e^{-\frac{1}{\hbar}(\hat{g}_2-\hat{g}_1)(\lambda)}, & \lambda \in \Gamma^{+}_{+0},\\
                    \mathbb{I} + E_{23}e^{-\frac{1}{\hbar}(\hat{g}_3-\hat{g}_2)(\lambda)}, & \lambda \in \Gamma^{-}_{-0},\\
                    (-i\sigma_2)\oplus 1, & \lambda \in \RR_-,\\
                    1\oplus(-i\sigma_2), & \lambda \in \RR_+.
                \end{cases}
        \end{equation*}
    Furthermore, $\hat{{\bf V}}(\lambda)$ is normalized as
        \begin{equation*}
            \hat{{\bf V}}(\lambda) = \left[\mathbb{I} + \OO(\lambda^{-1})\right]\hat{f}(\lambda).
        \end{equation*}
    \end{prop}
\begin{figure}
    \centering
    \begin{subfigure}[t]{0.49\textwidth}
    \centering
        \begin{overpic}[scale=.25]{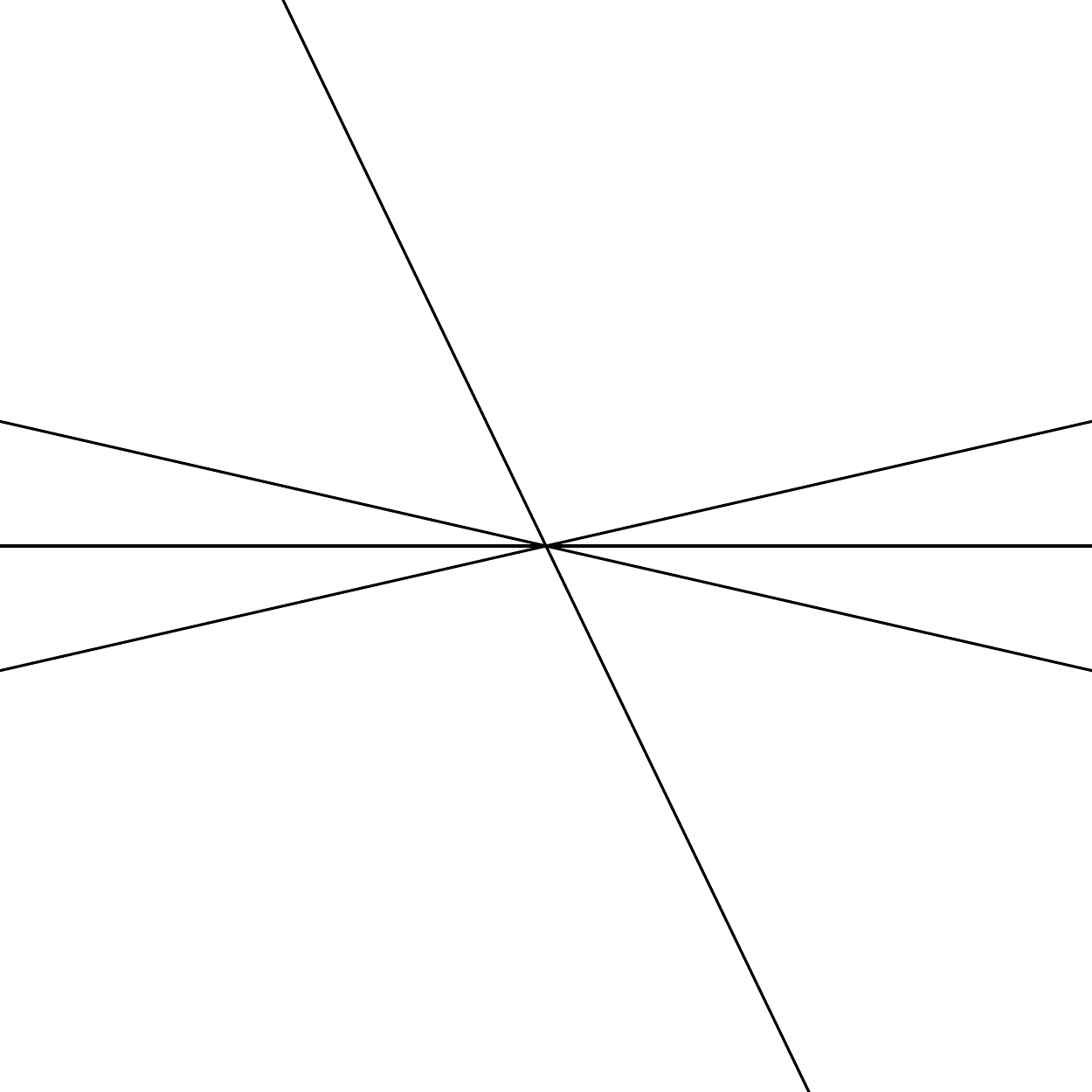}
            \put (20,70) {\small $\textcolor{blue}{\Delta_5}$}
            \put (70,20) {\small $\textcolor{blue}{\Delta_{-3}}$}
            \put (37,85) {\small $\hat{\Gamma}_5$}
            \put (55,15) {\small $\hat{\Gamma}_{-3}$}
            \put (5,65) {\small $\Gamma_{-0}^+$}
            \put (5,35) {\small $\Gamma_{-0}^-$}
            \put (85,65) {\small $\Gamma_{+0}^+$}
            \put (85,35) {\small $\Gamma_{+0}^-$}
            \put (0,52) {\small $\RR_-$}
            \put (96,52) {\small $\RR_+$}
        \end{overpic}
        \caption{Jump contours of $\hat{{\bf V}}(\lambda)$ and the regions $\Delta_{5}$, $\Delta_{-3}$.}
    \end{subfigure}
    \begin{subfigure}[t]{0.49\textwidth}
    \centering
        \begin{overpic}[scale=.25]{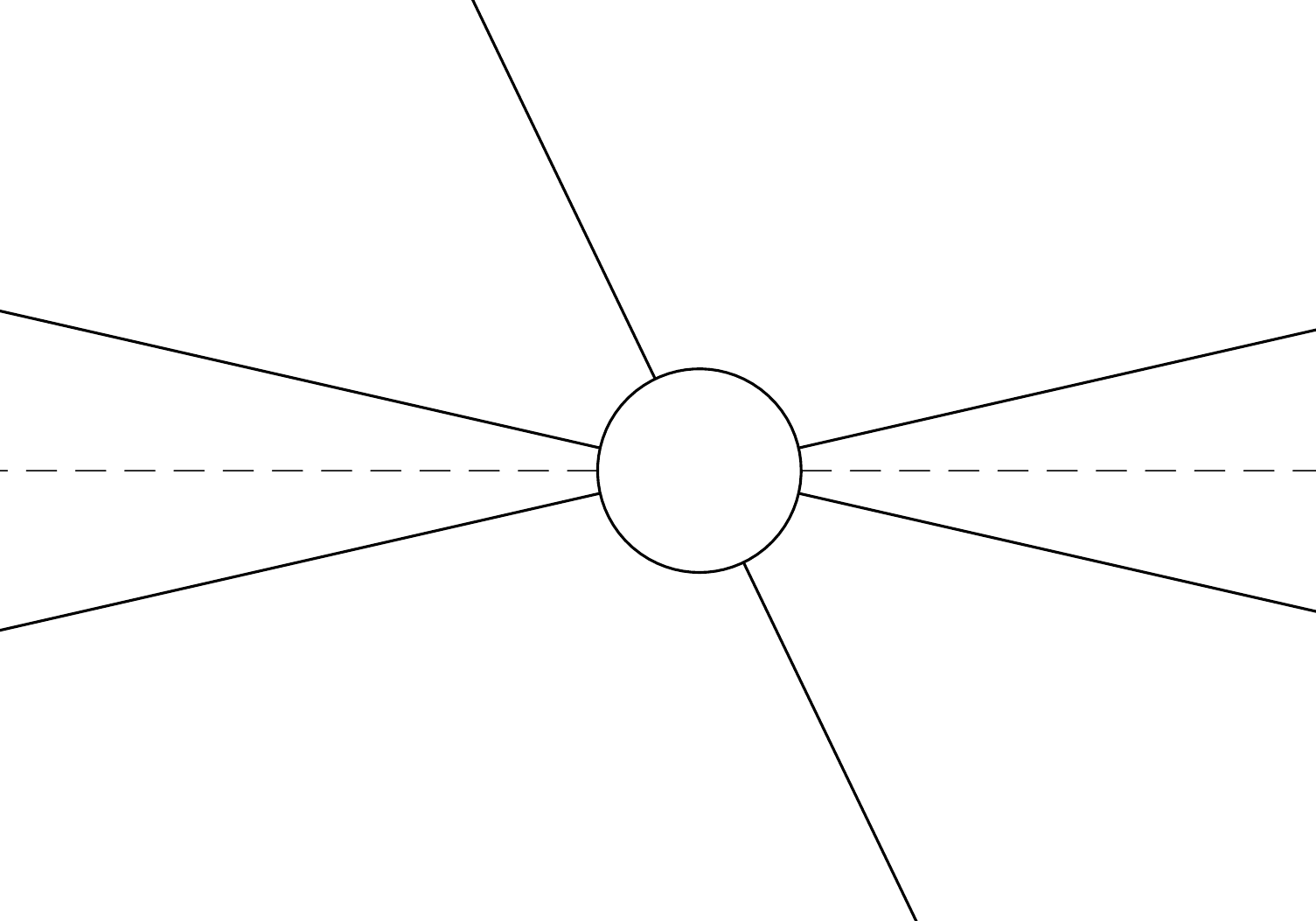}
            \put (27,65) {\small $\hat{\Gamma}_5$}
            \put (55,5) {\small $\hat{\Gamma}_{-3}$}
            \put (5,50) {\small $\Gamma_{-0}^+$}
            \put (5,15) {\small $\Gamma_{-0}^-$}
            \put (85,50) {\small $\Gamma_{+0}^+$}
            \put (85,15) {\small $\Gamma_{+0}^-$}
            \put (55,45) {\small $\partial D_{\epsilon}$}
        \end{overpic}
        \caption{Final set of jump contours for ${\bf R}(\lambda)$.}
    \end{subfigure}
    \caption{(a) Jump contours of $\hat{{\bf V}}(\lambda)$ and (b) final set of jump contours for ${\bf R}(\lambda)$, in the critical case $\eta<0$, and
    $\mu = \nu= 0$.}
    \label{fig:V-jumps-R-jumps}
\end{figure}
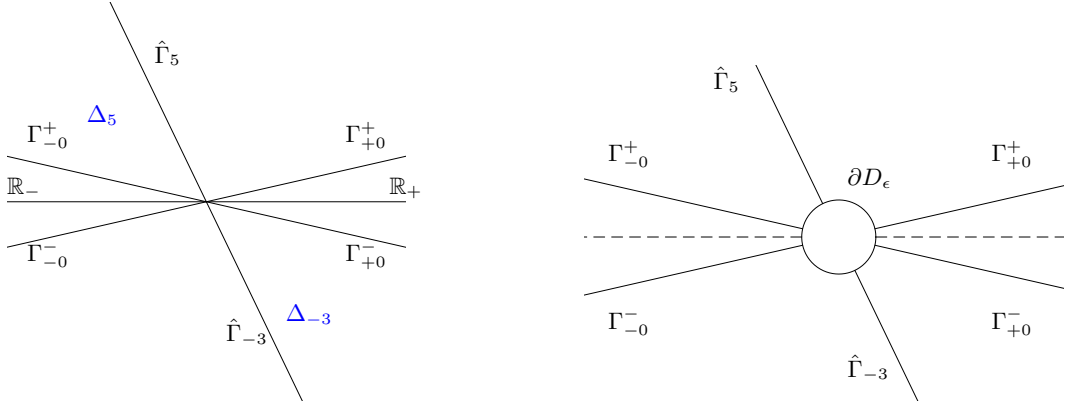

We now search for a parametrix to bring $\hat{{\bf V}}(\lambda)$ to the form of a small-norm Riemann-Hilbert problem. Outside of small discs centered as 
$\lambda = \pm \alpha$, we can again use the matrix-valued analytic function $M(\lambda)$ which we constructed in Proposition \ref{prop:global-parametrix-solution} as the solution to the Riemann-Hilbert problem \ref{prob:global-parametrix}. In fact, in this situation 
(where $\varsigma \equiv 0$), we see that the solution to the global parametrix problem actually coincides with $\hat{f}(\lambda)$:
    \begin{equation}
        M(\lambda) = \hat{f}(\lambda).
    \end{equation}
The construction of local parametrices here is different, however. There is only one local problem, at $\lambda = 0$, and this problem 
is not effectively a $2\times 2$ problem, as it was in the previous cases in Sections \S2,\S3. 

Let
    \begin{equation}
        s = -\eta_0 >0.
    \end{equation}

We define the scaling coordinate and function
    \begin{equation}
        \zeta(\lambda) = \left(\frac{5s}{6}\right)^{3/5}\lambda,\qquad\qquad \boldsymbol{x}(\lambda;\hbar) = \left(\frac{5s}{6}\right)^{-1/5}\left[\nu(\hbar) +\frac{3}{7}\lambda^2\right],
    \end{equation}
so that
    \begin{equation}
        g_j(\lambda) =
            \begin{cases}
                -\frac{6}{5}\omega^{1-j}[\zeta(\lambda)]^{5/3} + \omega^{j-1}\boldsymbol{x}(\lambda;\hbar)[\zeta(\lambda)]^{1/3}, & \text{Im }\lambda>0,\\
                -\frac{6}{5}\omega^{j-1}[\zeta(\lambda)]^{5/3} + \omega^{1-j}\boldsymbol{x}(\lambda;\hbar)[\zeta(\lambda)]^{1/3}, & \text{Im }\lambda<0.
            \end{cases}
    \end{equation}
We now set
    \begin{equation}\label{nu-definition}
        \nu(\hbar) := \left(\frac{5s}{6}\right)^{1/5}\hbar^{4/5}x,
    \end{equation}
so that the following Proposition holds:
    \begin{prop}
        As $\hbar\to 0$,
            \begin{equation}
                \lim_{\hbar\to 0}\hbar^{-4/5}\boldsymbol{x}(\lambda;\hbar) = x,
            \end{equation}
        for $|\lambda|<\hbar^{\frac{2}{5}+\delta}$, $0 < \delta < \frac{1}{5}$.
    \end{prop}
    \begin{proof}
        For $|\lambda|<\hbar^{\frac{2}{5}+\delta}$, using the definition \ref{nu-definition} of $\nu(\hbar)$, we have that
        \begin{equation*}
            \hbar^{-4/5}\boldsymbol{x}(\lambda;\hbar) = x + \frac{3}{7}\hbar^{-4/5}\left(\frac{5s}{6}\right)^{-1/5}\lambda^2 = x+\OO(\hbar^{2\delta}).
        \end{equation*}
        Thus, as $\hbar\to 0$, we obtain the result of the proposition.
    \end{proof}
Here, we will make use of a $3\times 3$ version of the Painlev\'{e} I parametrix. The existence of such a parametrix was first suggested in \cite{JKT}; it 
was proposed in \cite{DHL1,DHL2} that this parametrix is relevant to random matrix theory. The exact form of this $3\times 3$ parametrix $\Xi(\zeta;x)$ is 
given in Appendix \ref{Appendix:P1-parametrix}. This parametrix depends on a parameter $\varkappa$, which we set here to unity:
    \begin{equation}
        \varkappa := 1.
    \end{equation}
    By a possible slight redefinition of the jump contours of $\boldsymbol{\Xi}$, we set
    \begin{equation}
        P(\lambda) := -[\sigma_3\oplus 1]\left(\frac{5s}{6\hbar}\right)^{-\frac{1}{5}\hat{\sigma}}\boldsymbol{\Xi}(\hbar^{-3/5}\zeta(\lambda),\hbar^{-4/5}\boldsymbol{x}(\lambda))e^{-\frac{1}{\hbar} \hat{G}(\lambda)},
    \end{equation}
where $\hat{\sigma} = \text{diag }(1,0,-1)$. Note that the front factor here is a diagonal matrix, and is equal to
    \begin{equation*}
        -[\sigma_3\oplus 1]\left(\frac{5s}{6\hbar}\right)^{-\frac{1}{5}\hat{\sigma}}=M(\lambda)\mathfrak{f}^{-1}(\hbar^{-3/5}\zeta(\lambda)) = \hat{f}(\lambda)\hat{\mathfrak{f}}^{-1}(\hbar^{-3/5}\zeta(\lambda)).
    \end{equation*}
We can make then make the following statement:
    \begin{lemma}\label{P-expansion-negative}
        As $\hbar\to 0$,
            \begin{equation}
                P(\lambda) \sim \hat{f}(\lambda)\left[\mathbb{I} + \sum_{k=1}^{\infty} \frac{\Xi_k(\boldsymbol{x}(\lambda))}{\zeta(\lambda)^{k/3}}\hbar^{k/5}\right],
            \end{equation}
        where $\Xi_k(x)$ are the matrices appearing in Equation \eqref{Xi-asymptotic-expansion}.
    \end{lemma}
Set $D_{\hbar} := \{\lambda\in \CC \mid |\lambda| < \hbar^{\frac{2}{5}+\delta}\}$, and define
    \begin{equation}
        {\bf R}(\lambda) := 
            \hat{{\bf V}}(\lambda)
            \begin{cases}
                \hat{f}^{-1}(\lambda), & \lambda \in \CC \setminus D_{\hbar},\\
                P^{-1}(\lambda), & \lambda \in D_{\hbar}.
            \end{cases}
    \end{equation}
The following Proposition then holds immediately:
    \begin{prop}
        The matrix ${\bf R}(\lambda)$ satisfies the following Riemann-Hilbert problem:
            \begin{equation}
                {\bf R}_{+}(\lambda) = {\bf R}_{-}(\lambda)
                \begin{cases}
                    \mathbb{I} + \hat{f}(\lambda)E_{23}\hat{f}^{-1}(\lambda)e^{-\frac{1}{\hbar}(\hat{g}_3-\hat{g}_2)(\lambda)}, & \lambda \in \hat{\Gamma}_5 \setminus D_{\hbar},\\
                    \mathbb{I} + \hat{f}(\lambda)E_{12}\hat{f}^{-1}(\lambda)e^{-\frac{1}{\hbar}(\hat{g}_2-\hat{g}_1)(\lambda)}, & \lambda \in \hat{\Gamma}_{-3}\setminus D_{\hbar},\\
                    \mathbb{I} - \hat{f}(\lambda)E_{32}\hat{f}^{-1}(\lambda)e^{\frac{1}{\hbar}(\hat{g}_3-\hat{g}_2)(\lambda)}, & \lambda \in \Gamma^{\pm}_{+0}\setminus D_{\hbar},\\
                    \mathbb{I} - \hat{f}(\lambda)E_{21}\hat{f}^{-1}(\lambda)e^{\frac{1}{\hbar}(\hat{g}_2-\hat{g}_1)(\lambda)}, & \lambda \in \Gamma^{\pm}_{-0}\setminus D_{\hbar},\\
                    P(\lambda)\hat{f}^{-1}(\lambda), & \lambda \in \partial D_{\hbar}.
                \end{cases}
            \end{equation}
    \end{prop}

\begin{figure}
    \begin{subfigure}[t]{0.5\textwidth}
        \begin{overpic}[scale=0.25]{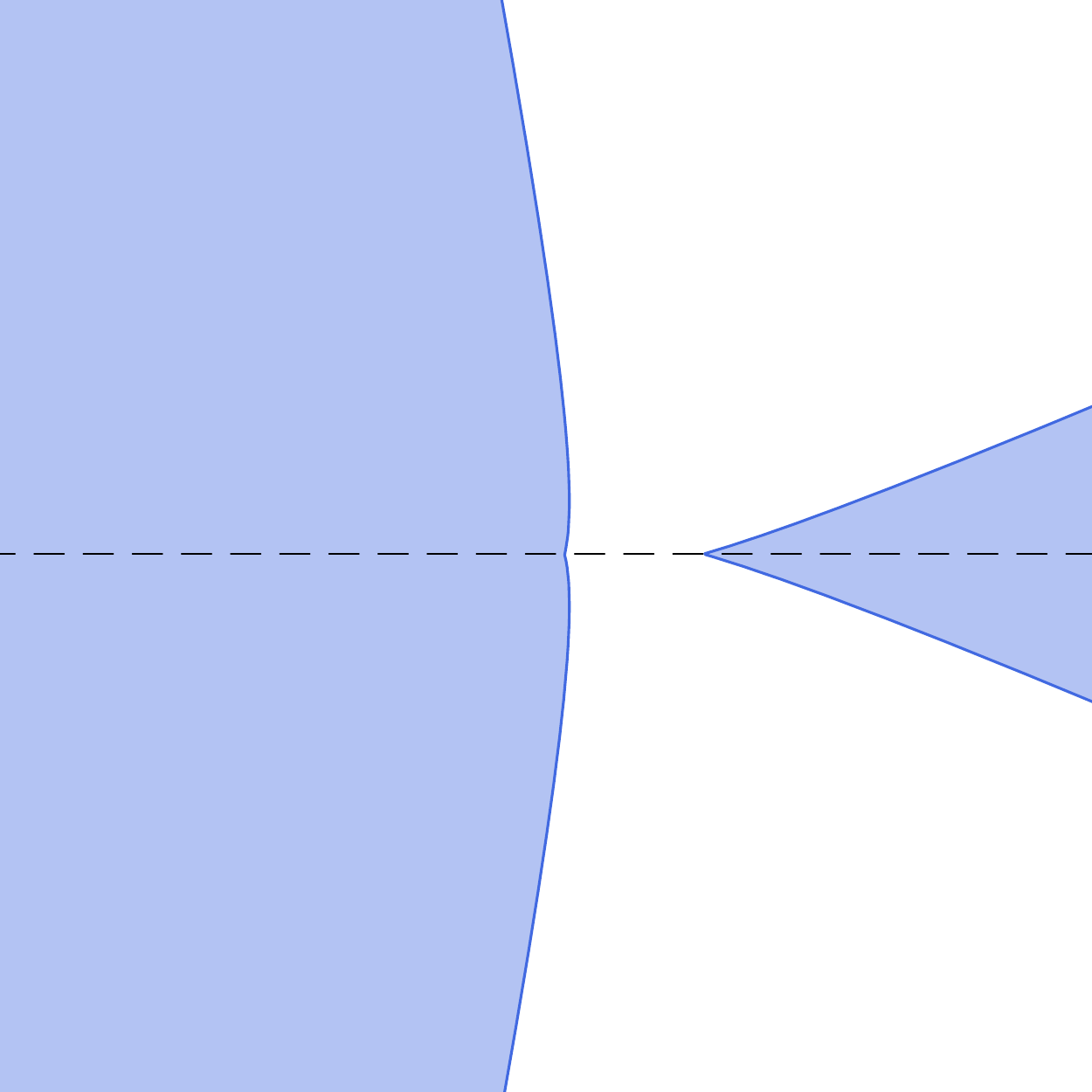}
            \put (47,50) {\tiny $0$}
            \put (0,50) {\tiny $\text{Im }\lambda = 0$}
        \end{overpic}
            \caption{Blue region where $\text{Re} [g_2-g_1](\lambda)<0$.}
    \end{subfigure}
    \begin{subfigure}[t]{0.5\textwidth}
    \centering
    \begin{overpic}[scale=0.25]{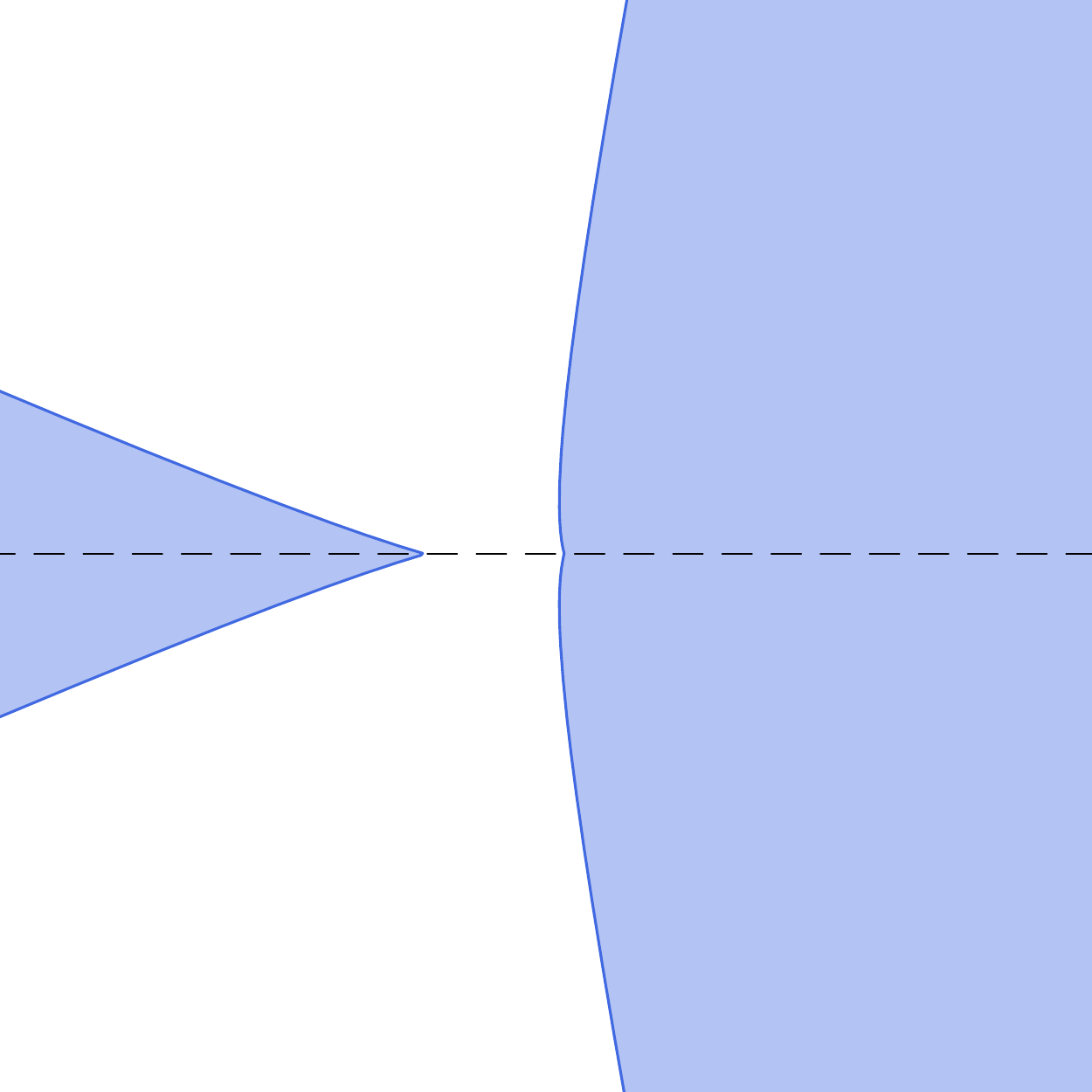}
            \put (53,50) {\tiny $0$}
            \put (75,50) {\tiny $\text{Im }\lambda = 0$}
    \end{overpic}
        \caption{Blue region where $\text{Re} [g_3-g_2](\lambda)<0$.}
    \end{subfigure}
    \caption{Domains where $\text{Re} [g_2-g_1](\lambda)$, $\text{Re} [g_3-g_2](\lambda)$ are negative. In Subfigure (a), the connected component of the negativity domain on the left extends to infinity, and contains the sector $|\arg\lambda| > \frac{4\pi}{7}$. In Subfigure (b), the connected component of the negativity domain on the right extends to infinity, and contains the sector $|\arg \lambda| < \frac{3\pi}{7}$.}
    \label{fig:g-function-signs}
\end{figure}

Let $\Gamma_{\bf R}$ denote the union of the jump contours of ${\bf R}(\lambda)$.    

Unfortunately, the jump of ${\bf R}(\lambda)$ across the circle $\{|\lambda| = \hbar^{2/5+\delta}\}$ is not close to the identity: there are
a finite number of terms in the asymptotic expansion of
    \begin{equation*}
        J_{\bf R}(\lambda) \big|_{|\lambda| = \hbar^{2/5+\delta}} = P(\lambda)\hat{f}^{-1}(\lambda)
    \end{equation*}
which tend to infinity as $\hbar \to 0$. Nowadays this is a more or less common occurrence in Deift-Zhou analysis \cite{Bertola-Lee1,Bertola-Lee2,Molag1,Molag2}, and can be addressed using various methods. We prefer to use the technique of the \textit{partial Schlesinger transformation}. To this end, we define the lower triangular matrix
    \begin{equation}\label{p-def}
        \mathfrak{p}(\lambda):=\mathbb{I} - \hbar^{1/5}\frac{ \mathcal{H}(x) E_{31} }{c\lambda}  + \hbar^{2/5}\frac{(\mathcal{H}^2(x)-q(x)) (E_{32} - E_{21})}{2c^2\lambda} - \hbar^{4/5}\frac{(\mathcal{H}^2(x)-q(x))^2E_{31}}{8c^4\lambda^2},
    \end{equation}
where $c:=\left(5s/6\right)^{1/5}>0$.
Clearly, the above has determinant $1$, and furthermore is analytic in $\CC \setminus D_{\hbar}$. Finally, we define the modified global parametrix
    \begin{equation}\label{M-def}
        M(\lambda) := \mathfrak{p}(\lambda)\hat{f}(\lambda),\qquad  \lambda \in \CC\setminus D_{\hbar}.
    \end{equation}
We then claim that
\begin{prop}\label{shrinking-prop}
    $\boldsymbol{\Delta}(\lambda) := P(\lambda)M^{-1}(\lambda)-\mathbb{I}$ is uniformly bounded on $|\lambda| = \hbar^{2/5+\delta}$, for any sufficiently small $\delta>0$.
\end{prop}
\begin{proof}
    By Lemma \ref{P-expansion-negative}, we have that
        \begin{equation*}
            P(\lambda)M^{-1}(\lambda) = \hat{f}(\lambda)\left[\mathbb{I} + \sum_{k=1}^{\infty} \frac{\Xi_k(\boldsymbol{x}(\lambda))}{\zeta(\lambda)^{k/3}}\hbar^{k/5}\right]\hat{f}^{-1}(\lambda)\mathfrak{p}^{-1}(\lambda), \qquad\hbar\to 0.
        \end{equation*}
    Rewrite this expression as
        \begin{equation*}
            \hat{P}(\lambda)\hat{f}^{-1}(\lambda) = \underbrace{\hat{f}(\lambda)\left[\mathbb{I} + \sum_{k=1}^{4} \frac{\Xi_k(\boldsymbol{x}(\lambda))}{\zeta(\lambda)^{k/3}}\hbar^{k/5}\right]\hat{f}^{-1}(\lambda)\mathfrak{p}^{-1}(\lambda)}_{(A)} + \underbrace{\hat{f}(\lambda)\left[ \sum_{k=5}^{\infty} \frac{\Xi_k(\boldsymbol{x}(\lambda))}{\zeta(\lambda)^{k/3}}\hbar^{k/5}\right]\hat{f}^{-1}(\lambda)\mathfrak{p}^{-1}(\lambda)}_{(B)}.
        \end{equation*}
    We begin by showing that $(B) = \OO(\hbar^{\frac{1}{5}-3\delta})$, so that taking $\delta>0$ sufficiently small, the first term is bounded in $\hbar$.
    By the symmetry properties of the coefficients $\Xi_k$ \eqref{P1-3-symm}, we have that
    \begin{align*}
            \hbar^{(3\ell+1)/5}\frac{\hat{f}(\lambda)\Xi_{3\ell+1}(\boldsymbol{x}(\lambda))\hat{f}^{-1}(\lambda)}{\zeta(\lambda)^{(3\ell+1)/3}} &= 
            \begin{pmatrix}
                0 & \OO(\hbar^{\frac{1}{5}+\ell(\frac{1}{5}-\delta)}) & 0\\
                0 & 0 & \OO(\hbar^{\frac{1}{5}+\ell(\frac{1}{5}-\delta)})\\
                \OO(\hbar^{-\frac{2}{5}+(\ell+1)(\frac{1}{5}-\delta)}) & 0 & 0
            \end{pmatrix},\\
            \hbar^{(3\ell+2)/5}\frac{\hat{f}(\lambda)\Xi_{3\ell+2}(\boldsymbol{x}(\lambda))\hat{f}^{-1}(\lambda)}{\zeta(\lambda)^{(3\ell+2)/3}} &= 
            \begin{pmatrix}
                0 & 0 & \OO(\hbar^{\frac{2}{5}+\ell(\frac{1}{5}-\delta)})\\
                \OO(\hbar^{-\frac{1}{5}+(\ell+1)(\frac{1}{5}-\delta)}) & 0 & 0\\
                0 & \OO(\hbar^{-\frac{1}{5}+(\ell+1)(\frac{1}{5}-\delta)}) & 0
            \end{pmatrix},\\
            \hbar^{3\ell/5}\frac{\hat{f}(\lambda)\Xi_{3\ell}(\boldsymbol{x}(\lambda))\hat{f}^{-1}(\lambda)}{\zeta(\lambda)^{\ell}} &=
            \begin{pmatrix}
                \OO(\hbar^{\ell(\frac{1}{5}-\delta)}) & 0 & 0\\
                0 & \OO(\hbar^{\ell(\frac{1}{5}-\delta)}) & 0\\
                0 & 0 &\OO(\hbar^{\ell(\frac{1}{5}-\delta)})
            \end{pmatrix}.
        \end{align*}
    On the other hand,
        \begin{equation*}
            \mathfrak{p}(\lambda) = \mathbb{I} + \OO(\hbar^{-\frac{1}{5}-\delta})\cdot E_{31} + \OO(\hbar^{-\delta})\cdot(E_{21}+E_{32}) ,\qquad \hbar\to 0.
        \end{equation*}
    It thus follows that 
        \begin{equation*}
            (B) = \hat{f}(\lambda)\left[ \sum_{k=5}^{\infty} \frac{\Xi_k(\boldsymbol{x}(\lambda))}{\zeta(\lambda)^{k/3}}\hbar^{k/5}\right]\hat{f}^{-1}(\lambda)\mathfrak{p}^{-1}(\lambda) = \OO(\hbar^{\frac{1}{5}-3\delta}),
        \end{equation*}
    so for $\delta>0$ sufficiently small, the above quantity is bounded as $\hbar\to 0$. Let us now turn our attention to $(A)$. By construction, we have 
    chosen the entries of $\mathfrak{p}(\lambda)$ to `kill off' the growing terms in this expression. A tedious but straightforward calculation then shows
    that
        \begin{equation*}
            (A) = \hat{f}(\lambda)\left[ \sum_{k=1}^{4} \frac{\Xi_k(\boldsymbol{x}(\lambda))}{\zeta(\lambda)^{k/3}}\hbar^{k/5}\right]\hat{f}^{-1}(\lambda)\mathfrak{p}^{-1}(\lambda) = \mathbb{I} + \OO(\hbar^{\frac{1}{5}-3\delta}),
        \end{equation*}
    and so 
        \begin{equation*}
            \boldsymbol{\Delta}(\lambda) = P(\lambda)M^{-1}(\lambda)-\mathbb{I} = (A) + (B) -\mathbb{I} = \OO(\hbar^{\frac{1}{5}-3\delta}).
        \end{equation*}
\end{proof}

Finally, defining
    \begin{equation}
        \hat{{\bf R}}(\lambda) := 
            \hat{{\bf V}}(\lambda)\begin{cases}
                M^{-1}(\lambda), & \lambda \in \CC \setminus D_{\hbar},\\
                P^{-1}(\lambda), & \lambda \in D_{\hbar},
            \end{cases}
    \end{equation}
where $M(\lambda)$ is as defined in Equation \eqref{M-def}, we can now claim that $\hat{ {\bf R}}(\lambda)$ is indeed close to the identity matrix:
    \begin{prop}\label{eta-negative-R-prop}
        Given $s := -\eta_0>0$, and for $\hbar$ sufficiently small, with $\hat{\nu}(\hbar)$ as defined in Equation \eqref{nu-definition}, 
       the Riemann-Hilbert problem for $\hat{{\bf R}}(\lambda)$ has a solution, and admits the $\hbar\to 0$ asymptotic expansion
        \begin{equation}
            \hat{{\bf R}}(\lambda) \sim \mathbb{I} + \sum_{k=3}^{\infty} {\bf R}_k(\lambda;x)\hbar^{k/5},
        \end{equation}
        which holds uniformly in $\CC\setminus\Gamma_{{\bf R}}$.
    \end{prop}
\begin{proof}
    Let
    \begin{equation*}
            \boldsymbol{\Delta}(\lambda) := \hat{{\bf R}}_-^{-1}\hat{{\bf R}}_+ - \mathbb{I}
        \end{equation*}
    denote the deviation of the jumps of ${\bf R}$ from the identity $\mathbb{I}$. $\boldsymbol{\Delta}(\lambda)$ is exponentially close to the identity matrix
    away from the circle $\partial D_{\hbar}$. By Proposition \ref{shrinking-prop}, we can see that
        \begin{equation*}
            \boldsymbol{\Delta}(\lambda) \sim \sum_{k=1}^{\infty} J_k(\lambda) \hbar^{k/5},
        \end{equation*}
    where the $J_k(\lambda)$ are explicit matrices defined in terms of the function $\mathfrak{p}(\lambda)$
    and coefficients $\Xi_k(x)$ appearing in the asymptotic expansion of $\boldsymbol{\Xi}(\zeta;x)$ (see Equation \eqref{Xi-asymptotic-expansion}).
    From standard theory \cite{DKMVZ} it follows that ${\bf R}$ admits an asymptotic series in $\hbar^{1/5}$:
        \begin{equation*}
            {\bf R}(\lambda) \sim \mathbb{I} + \sum_{k=1}^{\infty} {\bf R}_k(\lambda;x)\hbar^{k/5}.
        \end{equation*}
    The functions ${\bf R}_k(\lambda;x)$ can be determined iteratively as the solutions to certain additive Riemann-Hilbert problems. Since both $J_1(\lambda)$,
    $J_2(\lambda)$ are analytic functions in $D_{\hbar}$, it follows that
        \begin{equation}\label{R1-eta-minus}
            {\bf R}_1(\lambda;x) = {\bf R}_2(\lambda;x) = 0,\qquad \lambda\in \CC\setminus D_{\hbar}.
        \end{equation}
    Thus, ${\bf R}(\lambda) = \mathbb{I} + \OO(\hbar^{3/5})$, for $\lambda\in \CC\setminus D_{\hbar}$.
\end{proof}


\subsection{Proof of Theorems \ref{MainTheorem2} \& \ref{MainTheorem3}.}

\begin{proof} \textit{(Of Theorem \ref{MainTheorem2}.)}
    Using the results of Subsection \ref{subsection:eta-plus}, and by choosing an appropriate sector in which $\lambda\to \infty$, we can write 
        \begin{equation*}
            \mathfrak{G}(\lambda) = {\bf Z}(\lambda;{\bf t})e^{-\frac{1}{\hbar}\hat{\Theta}(\lambda;{\bf t})} = \mathfrak{h}^{(0)}{\bf R}(\lambda)\cdot M(\lambda)e^{\frac{1}{\hbar}(\hat{G}-\hat{\Theta})(\lambda)},
        \end{equation*}
    so that
        \begin{align*}
            \mathfrak{G}^{-1}(\lambda)\mathfrak{G}'(\lambda) &= \underbrace{e^{-\frac{1}{\hbar}(G-\hat{\Theta})(\lambda)}\frac{d}{d\lambda}\left[e^{\frac{1}{\hbar}(G-\hat{\Theta})(\lambda)}\right]}_{K_1(\lambda)} + \underbrace{e^{-\frac{1}{\hbar}(G-\hat{\Theta})(\lambda)}M^{-1}(\lambda)M'(\lambda)e^{\frac{1}{\hbar}(G-\hat{\Theta})(\lambda)} }_{K_2(\lambda)}\\
            &+\underbrace{e^{-\frac{1}{\hbar}(G-\hat{\Theta})(\lambda)}M^{-1}(\lambda){\bf R}^{-1}(\lambda){\bf R}'(\lambda)M(\lambda)e^{\frac{1}{\hbar}(G-\hat{\Theta})(\lambda)}}_{K_3(\lambda)}.
        \end{align*}
    We have that
        \begin{equation*}
            \frac{1}{\hbar}\tr \left[K_2(\lambda)\frac{\partial \hat{\Theta}}{\partial \eta}\right] = \frac{1}{\hbar}\tr \left[K_2(\lambda)\frac{\partial \hat{\Theta}}{\partial \nu}\right] = \OO(\lambda^{-3}),
        \end{equation*}
    so the terms involving $K_2(\lambda)$ do not contribute to the $\tau$-function. Now, calculating the terms involving $K_1(\lambda)$,
        \begin{align*}
            \frac{1}{\hbar}\Res_{\lambda=\infty}\tr \left[K_1(\lambda)\frac{\partial \hat{\Theta}}{\partial \nu}\right] &= \frac{5\eta_0}{6\hbar^2}\left[\hat{\nu}(\hbar) -\frac{125}{108}\eta_0^2\hat{\eta}(\hbar) + \frac{125}{216}\eta_0^3\right],\\
            \frac{1}{\hbar}\Res_{\lambda=\infty}\tr \left[K_1(\lambda)\frac{\partial \hat{\Theta}}{\partial \eta}\right] &= \frac{625}{648\hbar^2}\eta_0^3\left[\hat{\nu}(\hbar)-\frac{25}{12}\eta_0^2\hat{\eta}(\hbar)+\frac{125}{108}\eta_0^3\right].
        \end{align*}
    Comparing with the definition of $\hat{\tau}_0(\eta,0,\nu)$ \eqref{tauhat0-def}, we see that the differential of $\hat{\tau}_0(\eta,0,\nu)$ cancels
    exactly the differential formed by $K_1(\lambda)$:
        \begin{equation*}
            \frac{1}{\hbar}\Res_{\lambda=\infty}\tr \left[K_1(\lambda)\frac{\partial \hat{\Theta}}{\partial \eta}\right]d\hat{\eta}(\hbar) + \frac{1}{\hbar}\Res_{\lambda=\infty}\tr \left[K_1(\lambda)\frac{\partial \hat{\Theta}}{\partial \nu}\right]d\hat{\nu}(\hbar) + {\bf d}\hat{\tau}_0(\hat{\eta},0,\hat{\nu}|\hbar)= 0.
        \end{equation*}
    Terms of subleading order in $\hbar$ then arise from $K_3(\lambda)$. By the results of Proposition \ref{R-expansion-prop-critical-plus},
    we have that
        \begin{align*}
            {\bf R}^{-1}(\lambda){\bf R}'(\lambda) &= \hbar^{1/5} {\bf R}_1'(\lambda) + \OO(\hbar^{2/5}), \qquad \hbar\to 0, \qquad \text{and}\\
            \hbar^{1/5} {\bf R}_1'(\lambda) &= -\frac{\hbar^{1/5} (W_1+\hat{W}_1)}{\lambda^2} - \frac{2\hbar^{1/5}\alpha(\hat{W_1} - W_1)}{\lambda^3}+ \OO(\lambda^{-4}),\qquad \lambda\to \infty.
        \end{align*}
    On the other hand, as $\lambda\to \infty$,
    \begin{align*}
            M(\lambda)e^{\frac{1}{\hbar}(\hat{G}-\hat{\Theta})(\lambda)}\frac{\partial \hat{\Theta}}{\partial \nu}e^{-\frac{1}{\hbar}(\hat{G}-\hat{\Theta})(\lambda)}M^{-1}(\lambda) &=M(\lambda)\frac{\partial \hat{\Theta}}{\partial \nu}M^{-1}(\lambda)  = E_{13}\lambda + \OO(1),\\
            M(\lambda)e^{\frac{1}{\hbar}(\hat{G}-\hat{\Theta})(\lambda)}\frac{\partial \hat{\Theta}}{\partial \eta}e^{-\frac{1}{\hbar}(\hat{G}-\hat{\Theta})(\lambda)}M^{-1}(\lambda) &=M(\lambda)\frac{\partial \hat{\Theta}}{\partial \eta}M^{-1}(\lambda)\\
            &= (E_{12}+E_{23})\lambda^2 + 
            \begin{psmallmatrix}
                -\frac{5}{12}\eta_0 & 0 & \frac{125}{144}\eta_0^2\\
                0 & \frac{5}{6}\eta_0 & 0\\
                1 & 0 & -\frac{5}{12}\eta_0
            \end{psmallmatrix}\lambda
            +\OO(1),\\
        \end{align*}
    and so we compute that
        \begin{align*}
            -\frac{1}{\hbar}\tr \left[K_3(\lambda)\frac{\partial \hat{\Theta}}{\partial \nu}\right]d\hat{\nu}(\hbar) &=
            -Cn_{\nu}\hbar^{-1/5}\tr \left[K_3(\lambda)\frac{\partial \hat{\Theta}}{\partial \nu}\right]dx \\
            &= C n_{\nu}\left(\frac{3}{10\eta_0}\right)^{1/5}\mathcal{H}(x) dx + \OO(\hbar^{1/5}),\\
            -\frac{1}{\hbar}\tr \left[K_3(\lambda)\frac{\partial \hat{\Theta}}{\partial \eta}\right]d\hat{\eta}(\hbar) &=
            -Cn_{\eta}\hbar^{-1/5}\tr \left[K_3(\lambda)\frac{\partial \hat{\Theta}}{\partial \nu}\right]dx\\
            &=-\frac{125}{36}\eta_0^2 C n_{\eta}\left(\frac{3}{10\eta_0}\right)^{1/5}  \mathcal{H}(x)dx + \OO(\hbar^{1/5}).
        \end{align*}
    combining these results, and recalling the value of $C$ from Equation \eqref{C-definition}, we obtain that
        \begin{equation*}
            -\frac{1}{\hbar}\Res_{\lambda=\infty}\tr \left[K_3(\lambda)\frac{\partial \hat{\Theta}}{\partial \eta}\right]d\hat{\eta}(\hbar) - \frac{1}{\hbar}\Res_{\lambda=\infty}\tr \left[K_3(\lambda)\frac{\partial \hat{\Theta}}{\partial \nu}\right]d\hat{\nu}(\hbar) = -\mathcal{H}(x)dx + \OO(\hbar^{1/5}),
        \end{equation*}
    thus proving the theorem.
\end{proof}

We now proceed to the proof of Theorem \ref{MainTheorem3}.

\begin{proof}\textit{(Of Theorem \ref{MainTheorem3}.)} 
    By the results of Subsection \ref{subsection-eta-minus}, in an appropriately chosen sector as $\lambda\to \infty$, we have that
    \begin{equation*}
            \mathfrak{G}(\lambda) = {\bf Z}(\lambda;{\bf t})e^{-\frac{1}{\hbar}\hat{\Theta}(\lambda;{\bf t})} = {\bf R}(\lambda)\cdot M(\lambda)e^{\frac{1}{\hbar}(\hat{G}-\hat{\Theta})(\lambda)} = {\bf R}(\lambda) M(\lambda),
        \end{equation*}
    Therefore,
        \begin{equation*}
            \mathfrak{G}^{-1}(\lambda)\mathfrak{G}'(\lambda) = M^{-1}(\lambda)M'(\lambda) + M^{-1}(\lambda){\bf R}^{-1}(\lambda){\bf R}'(\lambda)M(\lambda).
        \end{equation*}
    By our observations in Proposition \ref{eta-negative-R-prop}, 
    \begin{equation*}
        {\bf R}(\lambda) = \mathbb{I} + \OO(\hbar^{3/5}),\qquad \text{   and so  }\qquad {\bf R}^{-1}(\lambda){\bf R}'(\lambda) = \OO(\hbar^{3/5}),
    \end{equation*}
    and since by \eqref{p-def}, \eqref{M-def} $M(\lambda) = \mathbb{I} + \OO(\hbar^{1/5})$, we see that 
        \begin{equation*}
            M^{-1}(\lambda){\bf R}^{-1}(\lambda){\bf R}'(\lambda)M(\lambda) = \OO(\hbar^{3/5}),
        \end{equation*}
    and so this term does not contribute at leading order to the $\tau$-function.
    On the other hand,
        \begin{equation}
            M^{-1}(\lambda)M'(\lambda) = -\hbar^{1/5}\left(\frac{5s}{6}\right)^{-1/5}\frac{\mathcal{H}(x)E_{31}}{\lambda^2} +\OO(\hbar^{2/5}).
        \end{equation}
    Thus,
        \begin{equation*}
            -\Res_{\lambda=\infty}\tr\left[\mathfrak{G}^{-1}(\lambda)\mathfrak{G}'(\lambda)\frac{\partial \hat{\Theta}}{\partial\nu }\right]d\hat{\nu}(\hbar) = -\mathcal{H}(x)dx + \OO(\hbar^{1/5}),
        \end{equation*}
   and we obtain the theorem in the limit $\hbar\to 0$.
\end{proof}

\section{Concluding Remarks.}
In summary, we have studied some asymptotic properties of a special solution to the $(3,4)$ string equation \eqref{string-equation}, which appears in the study
of the multicritical quartic $2$-matrix model. In particular, we were able to show that this solution admits a ``topological'' expansion, and that the
$\tau$-function for this solution has limits to the Painlev\'{e} I $\tau$-function, confirming a conjecture from \cite{CGM}. 

There are several questions we left unaddressed in this work, but are nevertheless still of interest:
    \begin{enumerate}
        \item The Stokes manifold for the symmetric solutions of the RHP \ref{prob:MAIN-RHP} 
        (i.e., solutions with corresponding Stokes parameters satisfying $s_k = -s_{k+8}$, $k=-7,...,-1$) contains the following $7$ planes, which we parameterize in terms of the remaining $7$ Stokes parameters $(s_1,....,s_7)$:
        \begin{align*}
            \Pi_0 &:= \{(x+1,-1,0,0,1,x,y)\},\\
            \Pi_1 &:= \{(0,-1,x,y,1-x,-1,0)\},\\
            \Pi_2 &:= \{(x,y-1,1,0,0,-1,y)\},\\
            \Pi_3 &:= \{(0,0,1,x,y,-x-1,1)\},\\
            \Pi_4 &:= \{(x,y,1-x,-1,0,0,1)\},\\
            \Pi_5 &:= \{(1,0,0,-1,1-y,x,y)\},\\
            \Pi_6 &:= \{(1,-1-y,x,y,1,0,0)\}.
        \end{align*}
        These planes intersect pairwise: $\Pi_k \cap \Pi_{k\pm 1} \neq \emptyset$ for $k\in \ZZ_7$, and otherwise are completely disjoint. As these planes are embedded in $\CC^7$, their intersection is a point in each case. Note that the Stokes parameters corresponding to the model studied in the present 
        work \eqref{STOKES_TRUNCATED} lies at the intersection of $\Pi_0$ and $\Pi_1$.  This is analogous to the way the tritronqu\'{e}e solutions 
        to Painlev\'{e} I (which lie at the intersection points of pairs of the $5$ lines defining the tronqu\'{e}e solutions in the PI Stokes manifold, cf. \cite{Kapaev-Kitaev}) appear in the $1$-matrix model.
        In particular, the set of solutions to the string equation lying on the plane $\Pi_1$ with parameter $y=0$ can be studied using the same techniques in this work, with little modification: one only has to choose the relevant corresponding Stokes data in the local model problems accordingly. The result is that our main theorems hold for each of these solutions, with the word \textit{tritronqu\'{e}e solution} replaced with 
        \textit{tronqu\'{e}e solution} of Painlev\'{e} I. The role of the Stokes parameter $y$ is unclear. Finding an asymptotic
        formula for solutions to the string equation with generic Stokes data on the plane $\Pi_1$ is still in general an open problem.
        \item A much wider range of limits for the $\tau$-function studied here tend to the Painlev\'{e} I $\tau$-function. Note that we can parameterize the (real) critical surface (see Appendix \ref{Appendix:ImplicitCurves}, Equation \eqref{CRITICAL-SURFACE}) of the aforementioned solutions by
            \begin{equation}
                \left\{(t_1(\varsigma,\eta),t_2^{\pm}(\varsigma,\eta),t_5(\varsigma,\eta)\mid \eta\in \RR, \varsigma > \max\left\{\frac{5}{3}\eta,0\right\}\right\},
            \end{equation}
        where
            \begin{equation}
                t_1(\varsigma,\eta) = -\frac{5}{12}\varsigma\left(5\eta-9\eta\varsigma+3\varsigma^2\right),\qquad t_2^{\pm}(\varsigma,\eta) = \pm\frac{\sqrt{2\varsigma}}{12}(5\eta-3\varsigma)^2,\qquad t_5(\varsigma,\eta) = \eta.
            \end{equation}
        given a point $P_0 = \langle t_1^{(0)},t_2^{\pm (0)},t_5^{(0)}\rangle$ on this surface, consider a unit vector ${\bf u} = \langle u_1,u_2,u_5\rangle$ based at $P_0$ which lies in the region $D$ below the tangent plane of the critical surface. Define rescaled coordinates
            \begin{align}
                X_1(x|T) &:= T^6 t_1^{(0)} + C T^{2/5}u_1 x,\\
                X_2(x|T) &:= T^5 t_2^{\pm (0)} + C T^{-3/5}u_2 x,\\
                X_5(x|T) &:= T^2 t_5^{(0)} + C T^{-18/5}u_5 x.
            \end{align}
        We then claim that
            \begin{conjecture}
                There exists a choice of constant $C = C(P_0,{\bf u })>0$ and a polynomial $Q(t_1,t_2,t_5)$ such that, if we define $\hat{\tau}_0 := e^{Q}$, considered as a differential
                in the variable $x$,
                    \begin{equation}
                        \lim_{T\to \infty} {\bf d} \log \frac{\tau(X_5(x|T),X_2(x|T),X_1(x|T))}{\hat{\tau}_0(X_5(x|T),X_2(x|T),X_1(x|T))} = -\mathcal{H}(x)dx,
                    \end{equation}
                where $\mathcal{H}(x)$ is the Painlev\'{e} I Hamiltonian.
            \end{conjecture}
        The proof of this fact should follow from the results of this work, after some rather tedious calculations. From more involved formal calculations,
        it follows that the functions $U,V$ should behave like
            \begin{align}
                U(X_5(x|T),X_2(x|T),X_1(x|T)) &= \varsigma T^2 + \frac{2}{u_1^2}T^{-4/5}q(x) + o(T^{-4/5}),\\
                V(X_5(x|T),X_2(x|T),X_1(x|T)) &= -\frac{\sqrt{2\varsigma}}{6}(5\eta-3\varsigma) T^3 + \frac{\sqrt{2\varsigma}}{u_1^2}T^{1/5}q(x) + o(T^{1/5}),
            \end{align}
        as $T\to \infty$. This conjecture is consistent with the theorems stated in the introduction. As observed in \cite{DHL2}, for any fixed $t_5$, the function $U$ satisfies the Boussinesq equation. The appearance of the Painlev\'{e} I transcendent on the critical surface here suggests that the above is a specific instance of the Type II Dubrovin Universality conjecture \cite{Dubrovin2,Dubrovin3} for the Boussinesq equation. To the knowledge of the author, this conjecture of the Boussinesq equation is largely unexplored. Although the class of solutions to the Boussinesq equation generated by the $(3,4)$ string equation is rather restricted (it contains only a finite-dimensional manifold of solutions), a proof of this conjecture in this case would shed light on the nature of the Dubrovin conjecture for higher-rank Hamiltonian PDEs.
        \item We showed that a $3\times 3$ Painlev\'{e} I parametrix became relevant in the critical $\eta<0$ limit. This parametrix, to our knowledge, has
        not been implemented in practice in any steepest descent analysis. There should be a 1-1 correspondence between solutions to this $3\times 3$ parametrix
        problem and the standard $2\times 2$ Painlev\'{e} I parametrix that appears more frequently in the literature. In principle, there should exist
        a monodromy map $\mathcal{M}$ which takes a solution to the $2\times 2$ problem with data $\{s_k\}$ to a solution to the $3\times 3$ problem
        with data $\{\nu_k\}$. A promising approach to this problem is suggested in \cite{JKT}, involving a generalized Laplace transform. This approach is 
        applied to a closely related problem for Painlev\'{e} II by K. Liechty and D. Wang in \cite{LW}. In Appendix \ref{Appendix:P1-parametrix}, we
        have already conjectured part of this correspondence. We plan to work out this map in general in a later work.
    \end{enumerate}

\appendix

\section{Hamiltonian structure of the $(3,4)$ string equation.}\label{Appendix:Hamiltonian}
In this appendix, we list the set of Darboux coordinates and corresponding Hamiltonians for the $(3,4)$ string equation \eqref{string-equation}. This is taken 
directly from \cite{DHL2}.

The associated set of Darboux coordinates for the string equation may be taken to be
        \begin{align}
            Q_U :&= U - \frac{4}{3}t_5, \qquad\qquad\qquad\qquad\qquad\qquad\, Q_V := V, \qquad\qquad  Q_W := U', \label{Darboux-Q}\\
            P_U :&= \frac{1}{4}\left( 3UU' - \frac{1}{3}U''' - \frac{7}{3}t_5 U'\right), \qquad\quad\,\, P_V := V', \qquad\quad\,\,\,\,\, P_W := \frac{1}{12}U''-\frac{1}{6}t_5 U + \frac{7}{18}t_5^2.\label{Darboux-P}
        \end{align}
Then, the Hamiltonians are
    \begin{align}
        H_1&= P_{U} Q_{W}+6 P_{W}^{2}-\frac{3}{8} Q_{U} Q_{W}^{2}+\frac{1}{2} P_{V}^{2}-\frac{1}{8} Q_{U}^{4}-\frac{3}{2} Q_{U} Q_{V}^{2}-t_1Q_U + 2t_2Q_V\nonumber\\
    &+ \frac{1}{8}t_5 (16 Q_{U} P_{W}-2 Q_{U}^{3}+4 Q_{V}^{2}-Q_{W}^{2}) - \frac{1}{2}t_5^2(4P_W-Q_U^2)+\frac{19}{27}t_5^3Q_U+\frac{41}{54}t_5^4-\frac{4}{3}t_5t_1\label{Hamiltonian-nu}\\
    H_2&=\frac{1}{2} P_{V} Q_{U} Q_{W}+\frac{1}{4} Q_{V} Q_{W}^{2}-2P_{U} P_{V}-6 P_{W} Q_{U} Q_{V}+Q_{V}^{3}+Q_{U}^{3} Q_{V} +2t_1Q_V\nonumber\\
    &+t_2(4P_W-Q_U^2) + \frac{1}{2}t_5(Q_{V} Q_{U}^{2}-P_{V} Q_{W}+4 Q_{V} P_{W})-2t_5t_2Q_U-\frac{65}{27}t_5^3Q_V-\frac{22}{9}t_5^2t_2\label{Hamiltonian-mu}\\
    H_5&=\frac{1}{2} Q_{W} P_{V} Q_{U} Q_{V}-\frac{3}{4} P_{U} Q_{W} Q_{U}^{2}-P_{U} P_{V} Q_{V}+P_{U} P_{W} Q_{W}+\frac{3}{8} Q_{V}^{4}-\frac{1}{128} Q_{W}^{4}+4 P_{W}^{3}\nonumber\\
    &-\frac{1}{16} Q_{U}^{6}-P_{W} P_{V}^{2}+P_{W} Q_{U}^{4}+P_{U}^{2} Q_{U}-\frac{9}{2} P_{W}^{2} Q_{U}^{2}-\frac{1}{8} Q_{U}^{3} Q_{V}^{2}+\frac{1}{8} Q_{U}^{2} P_{V}^{2}+\frac{3}{32} Q_{W}^{2} Q_{U}^{3}-\frac{1}{16} Q_{W}^{2} Q_{V}^{2}\nonumber\\
    &+t_1\left(2 Q_{U} P_{W}-\frac{1}{8} Q_{W}^{2}-\frac{1}{4} Q_{U}^{3}+\frac{1}{2} Q_{V}^{2}\right) + \frac{1}{2}t_2\left(Q_{V} Q_{U}^{2}-P_{V} Q_{W}+4 Q_{V} P_{W}\right)\nonumber\\
    &+t_5\bigg(\frac{3}{16} Q_{U}^{5}-2 P_{U}^{2}-\frac{1}{16} Q_{U}^{2} Q_{W}^{2}-\frac{1}{4} P_{W} Q_{W}^{2}+5 P_{W}^{2} Q_{U}-2 P_{W} Q_{U}^{3}-5 P_{W} Q_{V}^{2}+\frac{3}{4} Q_{U}^{2} Q_{V}^{2}\nonumber\\
    &-\frac{1}{4} P_{V}^{2} Q_{U}+\frac{1}{2} P_{V} Q_{V} Q_{W}+P_{U} Q_{U} Q_{W}\bigg) -t_2^2Q_U-t_1t_2(4P_W-Q_U^2)\nonumber\\
    &+ t_5^2\bigg(\frac{47}{12} Q_{U} Q_{V}^{2}-\frac{29}{18} P_{U} Q_{W}-\frac{3}{2} P_{W} Q_{U}^{2}+\frac{29}{48} Q_{U} Q_{W}^{2}+\frac{7}{18} Q_{U}^{4}-\frac{14}{9} P_{V}^{2}-\frac{20}{3} P_{W}^{2}\bigg)\nonumber\\
    &-\frac{1}{108}t_5^3(284 Q_{U} P_{W}-49 Q_{U}^{3}+152 Q_{V}^{2}-11 Q_{W}^{2})+\frac{19}{9}t_5^2t_1Q_U-\frac{65}{9}t_5^2t_2Q_V\nonumber\\
    &+\frac{1}{216}t_5^4(1304P_W-299Q_U^2)-\frac{2173}{972}t_5^5Q_U-\frac{2}{3}t_1^2-\frac{22}{9}t_5t_2^2+\frac{82}{27}t_5^3t_1-\frac{556}{243}t_5^6\label{Hamiltonian-eta}
\end{align}
    
One can then show that Hamilton's equations for the above Hamiltonians are equivalent to the $(3,4)$ string equation \eqref{string-equation}, combined with the 
following compatibility conditions (which determine the dependence of $U,V$ on $t_2, t_5$):
     \begin{align}
            \frac{\partial U}{\partial t_2} &= -2V', \label{U_mu}\\
            \frac{\partial V}{\partial t_2} &= \frac{1}{6}U''' - UU', \label{V_mu}\\
            \frac{\partial U}{\partial t_5} &=\frac{\partial}{\partial t_1}\left[-\frac{1}{6}UU'' + \frac{1}{8}(U')^2 + \frac{1}{4}U^3 - 
            \frac{1}{2}V^2 - \frac{5}{9}t_5\left(3U^2-U''\right) + \frac{4}{3}t_1\right],\label{U_eta}\\
            \frac{\partial V}{\partial t_5} &= \frac{\partial}{\partial t_1}\left[ \frac{1}{12}U''V - \frac{1}{4}U'V' + \frac{5}{16}U^2V - \left(\frac{5}{3}t_5 + \frac{1}{4}U\right)^2V - t_2 U\right]\label{V_eta}.
        \end{align}
Evaluated on a solution of the string equation, the Hamiltonians read
    \begin{align*}
        H_1 &:= -\frac{1}{12}U'U''' +\frac{1}{24}(U'')^2 + \frac{3}{8}U(U')^2 +\frac{1}{2}(V')^2 - \frac{1}{8}U^4 - \frac{3}{2}UV^2 -\frac{5}{6}t_5\left(\frac{1}{4}(U')^2 - \frac{1}{2}U^3 - 3 V^2\right) + \frac{t_1}{2}U^2,\\
        H_2 &:= \frac{1}{6}U'''V' -\frac{1}{2}VUU''+\frac{1}{4}V(U')^2 - UU'V'+U^3V+V^3+\frac{5}{6}t_5\left(U''-3U^2\right)V + \frac{1}{3}t_2(U''-3U^2) + 2t_1 V,\\
        H_5 &:= \frac{1}{144}U'U''U'''+\frac{1}{12}V'U'''-\frac{1}{16}U^2U'U'''+\frac{1}{144}U(U''')^2+\frac{1}{16}U(U')^2U''-\frac{1}{4}UVU'V' + \frac{3}{8}V^4\\
        &-\frac{1}{128}(U')^4-\frac{1}{16}U^6+\frac{1}{432}(U'')^3+\frac{1}{12}U^4U''+\frac{3}{32}U^3(U')^2-\frac{1}{8}U^3V^2-\frac{1}{32}U^2(U'')^2+\frac{1}{8}U^2(V')^2\\
        &-\frac{1}{12}U''(V')^2-\frac{1}{16}V^2(U')^2 + 
        \frac{5}{6}t_5\bigg(\frac{3}{2}U^2V^2+\frac{1}{8}UU''-\frac{1}{2}U(V')^2-\frac{1}{12}(U')^2U''-\frac{1}{2}V^2U''-\frac{7}{12}U^3U''\\
        &-\frac{3}{4}(U')^2U^2+\frac{1}{3}UU'U'''+\frac{1}{2}VU'V'+\frac{5}{8}U^5-\frac{1}{36}(U''')^2\bigg) + \frac{1}{2}t_2\left(U^2V+\frac{1}{3}U''V-U'V'\right)\\
        &+ \frac{t_1}{2}\left(V^2-\frac{1}{4}U'^2-\frac{1}{2}U^3+\frac{1}{3}UU''\right) - \frac{5}{3}t_5t_2 UV + \frac{5}{3}t_5t_1\left(U^2-\frac{1}{3}U''\right) + \frac{25}{36}t_5^2\bigg(U^2U''+3UV^2-\frac{3}{2}U^4\\
        &-\frac{1}{6}(U'')^2-2(V')^2-8t_2 V\bigg) - \frac{125}{18}t_5^3V -\frac{10}{9}t_5t_2^2-\frac{2}{3}t_1^2.
    \end{align*}
These Hamiltonians pairwise commute with respect to the canonical Poisson bracket induced by the Darboux coordinates, and furthermore have equal mixed partials:
    \begin{equation}
        \frac{\partial H_k}{\partial t_j} = \frac{\partial H_j}{\partial t_k}, \qquad\qquad k,j \in \{1,2,5\}.
    \end{equation}
Thus, one can define a closed differential $\boldsymbol{\omega}_{Okamoto}$:
    \begin{equation}
        \boldsymbol{\omega}_{Okamoto} := H_1dt_1 + H_2dt_2 + H_5dt_5,
    \end{equation}
which we can use to define the Okamoto $\tau$-function. As was shown in \cite{DHL2}, the above Hamiltonian system can also be realized as the isomonodromic deformations 
of a linear equation with rational coefficients. As such, it admits another $\tau$-function arising from a closed differential $\boldsymbol{\omega}_{JMU}$, in the sense of Jimbo, Miwa, and Ueno \cite{JMU1,JMU2}\footnote{Actually, since this system has a nondiagonalizable leading coefficient, the JMU differential is ill-defined, and the original definition in \cite{JMU1} must be modified \cite{DHL2}. This modified differential is the one we refer to as $\boldsymbol{\omega}_{JMU}$ here.}. As it turns out, these two differentials are proportional:
    \begin{equation}
        \boldsymbol{\omega}_{JMU} = \frac{1}{2}\boldsymbol{\omega}_{Okamoto}.
    \end{equation}
The differential of the JMU $\tau$-function can then be seen to be equivalent to
    \begin{equation}\label{tau-JMU}
        {\bf d}\log \tau({\bf t}) = \frac{1}{2}\boldsymbol{\omega}_{Okamoto}.
    \end{equation}

\section{Formal `topological' expansion of solutions to the string equation.}\label{Appendix:perturbative-expansion}
In this Appendix, we provide a formal method to calculate the $\tau$-function for the $(3,4)$ string equation, perturbatively in a small parameter
$\hbar \to 0_+$. In Section \ref{genus-zero-existence}, we will prove that the formal solution obtained here indeed describes the 
$\tau$-function we are looking for, provided that $t_5,t_2,t_1$ go to infinity in an appropriately chosen sector.

Let us now describe the formal method to construct the $\tau$-function. In the string equation \eqref{string-equation}, we make a 
rescaling of variables
    \begin{equation}\label{rescaling}
        t_{5} = \hbar^{-2/7}\eta, \qquad t_2 = \hbar^{-5/7}\mu, \qquad t_1 = \hbar^{-6/7}\nu, \qquad U = \hbar^{-2/7}u, \qquad V = \hbar^{-3/7}v
    \end{equation}
Note that this is just the prescribed rescaling of the RHP \eqref{prob:MAIN-RHP}, introduced in \eqref{parameter-rescaling}, and that we have also rescaled the functions $U$, $V$. In the new rescaled variables the string equation reads (below, we make the slight abuse of notation that $' = \frac{\partial}{\partial \nu}$)
    \begin{equation}\label{rescaled-string}
        \begin{cases}
            0 &= \frac{5}{2}\eta v -\frac{3}{2} u v + \mu + \hbar^2 v'',\\
            0 &= \frac{1}{2}u^3 + \frac{3}{2} v^2-\frac{5}{4}\eta u^2 + \nu - \hbar^2\left(\frac{3}{8}(u')^2 + \frac{3}{4}uu''- \frac{5}{12}\eta u''\right) + \hbar^{4} \frac{1}{12} u^{(4)}.
        \end{cases}
    \end{equation}
Since the above equation has an expansion in powers of $\hbar^2$, it makes sense to search for solutions to the string equation which admit an expansion in $\hbar^2$ as well:
    \begin{equation}\label{perturbative-uv}
        u(\eta,\mu,\nu|\hbar) = \sum_{k=0}^{\infty} u_k(\eta,\mu,\nu)\hbar^{2k},\qquad\qquad v(\eta,\mu,\nu|\hbar) = \sum_{k=0}^{\infty} v_k(\eta,\mu,\nu)\hbar^{2k}.
    \end{equation}
We prove the following Proposition:
    \begin{prop}
        In the variables \eqref{rescaling}, ${\bf d}\log\tau$ as defined in Equation \eqref{tau-JMU} admits a formal expansion in $\hbar^2$:
            \begin{equation}
                {\bf d}\log\tau(\eta,\mu,\nu|\hbar) = {\bf d}\left(\sum_{k=0}^{\infty}\log \tau_{k}(\eta,\mu,\nu)\hbar^{2k}\right).
            \end{equation}
        Consequentially, we obtain the following formal expansion for the $\tau$-function:
            \begin{equation}\label{formal-tau}
                \tau(\eta,\mu,\nu|\hbar) = \frac{C(\hbar)}{\chi(\eta,\mu,\nu)^{1/24}}e^{\hbar^{-2} \varpi_0(\eta,\mu,\nu)}[1 + \OO(\hbar^{2})].
            \end{equation}
        where $C(\hbar)$ is a constant independent of $\eta,\mu,\nu$,
        \begin{align}
            \varpi_0(\eta,\mu,\nu) &:= -\frac{\varsigma^5}{1344}(54\varsigma^2-245\eta\varsigma + 280\eta^2) - \frac{\mu^2 \varsigma^2(50\eta^2-80\eta\varsigma+27\varsigma^2)}{8(5\eta-3\varsigma)^2}+ \frac{\mu^4(25\eta-24\varsigma)}{(5\eta-3\varsigma)^4},\\
            \chi(\eta,\mu,\nu) &:= \varsigma(5\eta-3\varsigma)^2-\frac{72\mu^2}{(5\eta-3\varsigma)^2} = -2(5\eta-3\varsigma)\frac{\partial \mathcal{P}}{\partial \varsigma},
        \end{align}
        and $\varsigma$ is the solution to the $5^{th}$ order equation \eqref{sigma-eq}.
    \end{prop}
\begin{proof}
Inserting the expansions \eqref{perturbative-uv} into \eqref{rescaled-string}, we obtain at leading order in $\hbar$ the equations
    \begin{equation*}
        \begin{cases}
            0 &= \frac{1}{2}u_0^3-\frac{5}{4}\eta u_0^2 + \frac{3}{2}v_0^2 + \nu,\\
            0 &= \frac{5}{2}\eta v_0 - \frac{3}{2}u_0v_0 + \mu.
        \end{cases}
    \end{equation*}
We obtain as a solution
    \begin{equation*}
        u_0 = \varsigma(\eta,\mu,\nu),\qquad\qquad v_0 = -\frac{2\mu}{5\eta-3\varsigma(\eta,\mu,\nu)},
    \end{equation*}
where $\varsigma$ is the solution to the $5^{th}$ order algebraic equation \eqref{sigma-eq}. Note that the equations at order $\hbar^{2k}$ are
linear in $u_{k},v_{k}$, and depend polynomially on $\{u_{j},v_{j}\}_{j=0}^{k-1}$ and their derivatives, and so one can solve these equations uniquely
for $u_{k}$, $v_{k}$, and in particular the full `topological' expansion is completely determined in terms of the solution to the algebraic equation for 
$\varsigma(\eta,\mu,\nu)$. To obtain an expression for $d\log \tau$, we substitute the rescaled variables \eqref{rescaling} into the expressions for the
Hamiltonians. For instance, for the differential $H_1dt_1$, one finds that:
    \begin{equation*}
        H_1dt_1 = \hbar^2\left(\sum_{k=0}^{\infty}h_1^{(k)}\hbar^{2k}\right)d\nu,
    \end{equation*}
where $h_1^{(k)}$ are differential polynomials in the variables $\{u_k,v_k\}$, and $\eta,\mu, \nu$ (equivalently, they are rational functions of $u_0,\eta,\mu,\nu$). One can derive similar expansions for the differentials $H_2dt_2 = \hbar^2\left(\sum_{k=0}^{\infty}h_2^{(k)}\hbar^{2k}\right)d\mu$ and $H_5dt_5=\hbar^2\left(\sum_{k=0}^{\infty}h_5^{(k)}\hbar^{2k}\right)d\eta$. Note that this differential is closed, and thus all of the coefficients of each power of $\hbar^{2k}$ define closed differentials.
Taking $d\log \tau_k(\eta,\mu,\nu) = \frac{3}{2}\left(h_1^{(k)}d\nu + h_2^{(k)}d\mu + h_5^{(k)}d\eta\right)$ then proves the first claim.

To see that Formula \eqref{formal-tau} holds, we simply apply explicitly the procedure outlined above. To leading order, one finds that the Hamiltonians $h_k^{(0)}$ are
    \begin{align}
        h_1^{(0)} &= -\frac{1}{24}\varsigma^3(20\eta-9\varsigma) + \frac{2\mu^2(6\varsigma-5\eta)}{(5\eta-3\varsigma)^2}, \label{h1_0}\\
        h_2^{(0)} &= -\mu\varsigma^2 + \frac{16\mu^3}{(5\eta-3\varsigma)^3},\label{h2_0}\\
        h_5^{(0)} &= \frac{5}{48}\varsigma^5(2\eta-\varsigma) + \frac{5}{2}\frac{\mu^2\varsigma^2(5\eta-4\varsigma)}{(5\eta-3\varsigma)^2}-\frac{30\mu^4}{(5\eta-3\varsigma)^4},\label{h5_0}
    \end{align}
where in the above we have replace $\nu = \nu(\sigma)$ using Equation \eqref{sigma-eq}.
We have that
    \begin{align*}
        \log \tau_0(\eta,\mu,\nu) = \int h_1^{(0)}(\eta,\mu,\nu) d\nu + c_0(\eta,\mu)
        =-\int \left[\frac{1}{24}\varsigma^3(20\eta-9\varsigma) - \frac{2\mu^2(6\varsigma-5\eta)}{(5\eta-3\varsigma)^2}\right]d\nu + c_0(\eta,\mu)
    \end{align*}
Using equation \eqref{sigma-eq} to make the change of variables $\nu = \nu(\varsigma)$, 
    \begin{equation*}
        d\nu = \frac{d\nu}{d\varsigma}d\varsigma = -\frac{\partial \mathcal{P}}{\partial \varsigma}\bigg/\frac{\partial \mathcal{P}}{\partial \nu} d\varsigma = \left[\frac{1}{2}\varsigma(5\eta-3\varsigma)-\frac{36\mu^2}{(5\eta-3\varsigma)^3}\right] d\varsigma,
    \end{equation*}
and we obtain the following expression for $\tau_0$ (we have replaced here $\nu = \nu(\varsigma)$ using Equation \eqref{sigma-eq}):
\begin{align*}
    \log \tau_0(\eta,\mu,\nu) &= -\int \left[\frac{1}{24}\varsigma^3(20\eta-9\varsigma) - \frac{2\mu^2(6\varsigma-5\eta)}{(5\eta-3\varsigma)^2}\right]\left[\frac{1}{2}\varsigma(5\eta-3\varsigma)-\frac{36\mu^2}{(5\eta-3\varsigma)^3}\right]d\varsigma + c_0(\eta,\mu).
\end{align*}
 The integrand is then a rational function of $\varsigma$, and thus can be evaluated directly:
    \begin{align*}
        \log \tau_0(\eta,\mu,\nu) &= -\frac{\varsigma^5}{672}(54\varsigma^2-245\eta\varsigma + 280\eta^2) - \frac{\mu^2(729\varsigma^4-2160\eta\varsigma^3+1125\eta^2\varsigma^2+750\eta^3\varsigma -625\eta^5)}{108(5\eta-3\varsigma)^2}\\
        &+ \frac{2\mu^4(25\eta-24\varsigma)}{(5\eta-3\varsigma)^4} + c_0(\eta,\mu).
    \end{align*}
Differentiating this result with respect to $\eta,\mu$ and comparing to the expressions
we obtained for $h_2^{(0)}, h_5^{(0)}$ allows one to determine the constant of integration: $c_0(\eta,\nu) = -\frac{25}{108}\eta^2 \mu^2 + C$, where $C$ is independent of
$\eta,\mu,\nu$. Multiplying through by the factor of $1/2$ that converts the Okamoto differential to the Jimbo-Miwa-Ueno one gives the leading term (the exponential part) of the formal asymptotics \eqref{formal-tau}. Continuing in a similar fashion allows one to obtain 
an expression for $\log \tau_1(\eta,\mu,\nu)$; as the procedure is identical, we omit it. This concludes the proof.

\end{proof}

\begin{remark}\label{topological-expansion-remark}
    \textit{Interpretation of the above formal asymptotics.} Within the context of the $2$-matrix model, the above expansion has the following formal
    interpretation. Recall that the free energy $\mathcal{F}(\tau,t,H)$ of the quartic $2$-matrix model:
        \begin{align}
            Z_N(\tau,t,H;N) &:=\iint \exp\left[N\tr\left(\tau XY-V(X,te^{H})-V(Y,te^{H})\right)\right] dXdY,\nonumber\\
            \mathcal{F}(\tau,t,H) &:=\frac{1}{N^2}\log\frac{Z_N(\tau,t,H;N)}{Z_N(\tau,0,0;N)},
        \end{align}
    where $X,Y$ are $N\times N$ Hermitian matrices, $V(X,t) :=\frac{1}{2}X^2 + \frac{t}{4}X^4$, 
    and expression is taken as the analytic continuation of this expression if $t<0$, so that convergence is ensured.  It follows from the results of 
    \cite{DHL1} that this expression\footnote{We take here a regularized version of the free energy; since this regularization is a polynomial in $N^{-2}$,
    this modified expression also admits a topological expansion.} admits a topological expansion
        \begin{equation}\label{regularized-free-energy}
            \mathcal{F}(\tau,t,H) -\mathcal{F}_{reg}(\tau,t,H) \sim \sum_{g=0}^{\infty} \frac{\check{F}_g(\tau,t,H)}{N^{2g}},
        \end{equation}
    where $\mathcal{F}_{reg}(\tau,t,H)$ is an appropriately chosen polynomial in $\tau,t,H$, and $N^{-2}$. Again from \cite{DHL1}, the coefficients in this
    expansion are expressible as rational functions and logarithms in the variables $t,\tau,H$, and a special solution $\sigma = \sigma(\tau,t,H)$ 
    to the algebraic equation (cf. Equation 2.1 in \cite{DHL1})
        \begin{equation}\label{Ideal}
            0 = \mathfrak{I}(\sigma;t,\tau,H) := -t - \frac{1}{9}\tau^2\sigma(\sigma^2-3) -\frac{1}{3}\frac{\sigma}{(1+\sigma)^2} + \frac{2}{3}\left(\frac{\sigma}{1-\sigma^2}\right)^2[\cosh H - 1].
        \end{equation}
    We are interested in the multicritical point 
        \begin{equation}
            \tau = \tau_c := \frac{1}{4}, \qquad t = t_c := -\frac{5}{72},\qquad H = H_c := 0.
        \end{equation}
    We are thus led to define the following multiscaling limit:
        \begin{equation}
            \tau = \tau_c - \frac{C_5\eta}{N^{2/7}} + \frac{C_1\nu }{9N^{6/7}}, \qquad t = t_c - \frac{C_5\eta}{9N^{2/7}} - \frac{C_1\nu}{N^{6/7}} + \frac{2C_5^2\eta^2}{9N^{4/7}} - \frac{8C_5^3\eta^3}{9N^{6/7}},\qquad H = \frac{C_2\mu}{N^{6/7}},
        \end{equation}
    where $C_1,C_2,C_5>0$ are
        \begin{equation}
            C_1 = \frac{9}{164},\quad C_2 = \frac{2}{3},\quad C_5 = \frac{5}{12}.
        \end{equation}
    If we insert these expressions into \eqref{Ideal}, we see that the solution $\sigma(\tau,t,H)$ we are interested in admits an expansion in $N^{-2/7}$:
    \begin{equation}
            \sigma(\tau,t,H) \sim \sum_{k=0}^{\infty} \frac{\sigma_{k}(\eta,\mu,\nu)}{N^{2k/7}},
        \end{equation}
    where the first two terms in this expansion are given by
            \begin{equation}
                \sigma_0(\eta,\mu,\nu) \equiv 1,\qquad\qquad \sigma_1(\eta,\mu,\nu) = \frac{5}{3}\eta-\varsigma(\eta,\mu,\nu),
            \end{equation}
        with $\varsigma(\eta,\mu,\nu)$ the solution to \eqref{sigma-eq}, and the remaining terms can be expressed rationally in terms of $\eta,\mu,\nu$, and $\varsigma$.
    Inserting these formulae into the regularized free energy \eqref{regularized-free-energy}, we obtain that, order by order,
        \begin{equation*}
            \check{F}_g(\tau,t,H)\longrightarrow \log \tau_g(\eta,\mu,\nu), \qquad N\to \infty.
        \end{equation*}
    On the other hand, the asymptotic series
        \begin{equation*}
            \frac{1}{N^2}\log \tau(\eta,\mu,\nu) \sim \sum_{g=0}^{\infty} \frac{\log\tau_g(\eta,\mu,\nu)}{N^{2g}}
        \end{equation*}
    is a formal asymptotic expansion for the $\tau$-function of the rescaled string equation, under the parameter identification $\hbar = N^{-1}$. Thus, we
    can see that the topological expansion of the free energy becomes the topological expansion of the $\tau$-function for the string equation under the 
    multiscaling limit.
\end{remark}

\section{Implicit representation of the spectral curve and critical surface.}\label{Appendix:ImplicitCurves}

The spectral curve is given by
    \begin{equation}
        F(Y,\lambda) := Y^3 - F_2(\lambda;\eta,\mu,\nu)Y - F_4(\lambda;\eta,\mu,\nu) = 0,
    \end{equation}
where
    \begin{align}
        F_2(\lambda;\eta,\mu,\nu) &= 5\eta\lambda^2 + 2\mu \lambda + \frac{1}{2}h_1^{(0)} + \frac{5}{3}\eta \nu,\\
        F_4(\lambda;\eta,\mu,\nu) &= \lambda^4 + \left(\frac{125}{27}\eta^3 + \nu\right) \lambda^2 + \left(\frac{1}{2}h_2^{(0)} + \frac{50}{9}\eta^2\mu\right) + \frac{1}{2}h_5^{(0)} + \frac{25}{18}\eta^2h_1^{(0)} + \frac{20}{9}\eta\mu^2 + \frac{1}{3}\nu^2,
    \end{align}
where $\{h_k^{(0)}$, $k=1,2,5\}$ are the leading order terms in the $\hbar^2$-expansion of the Hamiltonians $H_k$, see Equations 
\eqref{h1_0}--\eqref{h5_0} of the previous appendix. This is consistent with the formula for the spectral curve derived in \cite{DHL2}, Equations 4.3 and 4.3.

The critical surface $\mathfrak{W}$ is a subset of the zero locus of the discriminant of Equation \eqref{sigma-eq}. This discriminant can written in terms of
$\eta,\mu,\nu$ explicitly:
    \begin{align}\label{CRITICAL-SURFACE}
        D(\eta,\mu,\nu) &:= \frac{78125}{93312} \eta^{12} \nu +\frac{3125}{15552} \eta^{10} \mu^{2}-\frac{625}{216} \eta^{9} \nu^{2}-\frac{75}{16} \eta^{7} \mu^{2} \nu -\frac{17}{18} \eta^{5} \mu^{4}+\frac{15}{4} \eta^{6} \nu^{3}+\frac{153}{20} \eta^{4} \mu^{2} \nu^{2}+6 \eta^{2} \mu^{4} \nu \nonumber\\
        &-\frac{54}{25} \eta^{3} \nu^{4}+\mu^{6}-\frac{81}{25} \eta  \,\mu^{2} \nu^{3}+\frac{1458}{3125} \nu^{5}= 0.
    \end{align}
A direct calculation reveals that the curves $\gamma_+$, $\gamma_-$ defined in Equation \eqref{gamma-pm} belong to the surface $\mathfrak{W}$. The part of this surface we are
interested in can be parametrized by
    \begin{equation}
        \nu(\varsigma,\eta) = -\frac{5\varsigma}{12}(5\eta^2-9\eta\varsigma+3\varsigma^2), \qquad \mu(\varsigma,\eta) = \pm \frac{\sqrt{2\varsigma}}{12}(5\eta-3\varsigma)^2,\qquad \eta(\varsigma,\eta) = \eta.
    \end{equation}
where $\eta\in \RR$, $\varsigma>\max\{\frac{5}{3}\eta,0\}$, and the `$\pm$' in the definition of $\mu(\varsigma,\eta)$ accounts for both signs of $\mu$. The curves $\gamma_{-},\gamma_{+}$
then correspond to the specializations $\varsigma = 0$, $\varsigma = \frac{5}{3}\eta$, respectively: note that $\mu(\varsigma,\eta) = 0$ in either case, and so the surface we have 
defined is continuous across these curves. The Gauss map ${\bf N}:\mathfrak{W}\to S^2$ which assigns to each point of $\mathfrak{W}$ its unit normal is a continuous function on $\mathfrak{W}\setminus \gamma_+$. On $\gamma_+$, ${\bf N}$ can take on two values, depending on whether we have approached $\gamma_+$ from the $\mu>0$ or $\mu <0$ side. For a point $(\eta,0,\frac{125}{108}\eta^3)\in \gamma_+$, we denote these two vectors by ${\bf N}^{(\pm)}$. We have that
    \begin{equation}
        {\bf N}^{(+)}\cdot {\bf N}^{(-)} = \frac{15625\eta^4-4320\eta+1296}{15625\eta^4+4320\eta+1296}.
    \end{equation}
It follows that the angle between these two vectors tends to $0$ as $\eta\to 0,\infty$. This angle has a unique maximum on $\gamma_+$ at $\eta = \frac{2}{5\sqrt{5}}3^{3/4}$, where the
angle between ${\bf N}^{(\pm)}$ is $1.580416... = \frac{\pi}{2} + 0.00962...$. Furthermore, this angle vanishes as $\sqrt{\frac{40\eta}{3}}[1+\OO(\eta)]$ as $\eta\to 0$.

\section{$3\times 3$ Painlev\'{e} I parametrix.}\label{Appendix:P1-parametrix}
For simplicity in the below, \textbf{\textit{all rays are oriented outwards towards infinity unless otherwise stated.}}
Define contours $\gamma_k$ by
    \begin{equation}
        \gamma_{\pm k} := \left\{\zeta\in \CC \big| \arg \zeta = \pm \frac{\pi}{10} \pm \frac{\pi}{5}(k-1)\right\}, \qquad k=1,...,5,
    \end{equation}
and define the $3\times 3$ matrix-valued function $J_{\boldsymbol{\Xi}^{(0)}}: \cup_{|k|=1}^5\gamma_k\cup \RR_- \to \CC$
\begin{equation}
        J_{\boldsymbol{\Xi}^{(0)}}(\zeta)= 
        \begin{cases}
            S_{1} := \mathbb{I} + \mathfrak{s}_1 E_{13},& \zeta\in \gamma_1,\\
            S_{2} := \mathbb{I} + \mathfrak{s}_2 E_{23},& \zeta\in \gamma_2,\\
            S_{3} := \mathbb{I} + \mathfrak{s}_3 E_{21},& \zeta\in \gamma_3,\\
            S_{4} := \mathbb{I} + \mathfrak{s}_4 E_{31},& \zeta\in \gamma_4,\\
            S_{5} := \mathbb{I} + \mathfrak{s}_5 E_{32},& \zeta\in \gamma_5,\\
            \mathcal{S} = \begin{psmallmatrix}
                0 & 1 & 0\\
                0 & 0 & 1\\
                1 & 0 & 0
            \end{psmallmatrix},& \zeta\in \RR_-,\\
            S_{-5} := \mathbb{I} - \mathfrak{s}_{1}E_{23},& \zeta\in \gamma_{-5},\\
            S_{-4} := \mathbb{I} - \mathfrak{s}_{2}E_{21},& \zeta\in \gamma_{-4},\\
            S_{-3} := \mathbb{I} - \mathfrak{s}_{3}E_{31},& \zeta\in \gamma_{-3},\\
            S_{-2} := \mathbb{I} - \mathfrak{s}_{4}E_{32},& \zeta\in \gamma_{-2},\\
            S_{-1} := \mathbb{I} - \mathfrak{s}_{5}E_{12},& \zeta\in \gamma_{-1},
        \end{cases}
    \end{equation}
where the $\mathfrak{s}_k\in \CC$ are complex numbers (independent of $\zeta,x$) satisfying the Stokes relation
    \begin{equation}\label{STOKES-p1-3}
        S_{1}\cdots S_{5}\mathcal{S}^T = \mathcal{S}(S_{1}\cdots S_{5})^T.
    \end{equation}
Consider the Riemann-Hilbert problem
    \begin{equation}
                \begin{cases}
                    \boldsymbol{\Xi}^{(0)}_+(\zeta;x) = \boldsymbol{\Xi}^{(0)}_-(\zeta;x)J_{\boldsymbol{\Xi}^{(0)}}(\zeta), & \zeta \in \Gamma_{\boldsymbol{\Xi}},\\
                    \boldsymbol{\Xi}^{(0)}(\zeta;x) = \mathfrak{f}^{(0)}(\zeta)\left[\mathbb{I} + 
                    \frac{\Xi^{(0)}_1(x)}{\zeta^{1/3}} + \frac{\Xi^{(0)}_2(x)}{\zeta^{2/3}} + \OO(\zeta^{-1})\right]e^{\Upsilon^{(0)}(\zeta;x)}, & \zeta\to \infty,
                \end{cases}
            \end{equation}
        Here, 
            \begin{equation}
                \mathfrak{f}^{(0)}(\zeta) := -(\sigma_3\oplus 1)f(\zeta),
            \end{equation}
        where $f(\zeta)$ is as defined in Equation \eqref{gauge-matrix},
        \begin{equation}
            \Upsilon^{(0)}(\zeta;x) = 
        \text{diag}\left(
            -\frac{6}{5}\zeta^{5/3} +x\zeta^{1/3}, -\frac{6}{5}\omega^2\zeta^{5/3} +\omega x\zeta^{1/3}, -\frac{6}{5}\omega\zeta^{5/3} +\omega^2 x\zeta^{1/3}\right),
        \end{equation}
        and
        \begin{equation}
            \Xi^{(0)}_1 := \begin{psmallmatrix}
            -\mathcal{H}(x) & 0 & 0\\
            0 & -\omega^2 \mathcal{H}(x) & 0\\
            0 & 0 & -\omega \mathcal{H}(x) 
            \end{psmallmatrix}, \qquad\qquad 
            \Xi^{(0)}_2 := 
            \begin{psmallmatrix}
            \frac{1}{2}\mathcal{H}(x) & \frac{\omega-1}{6}q(x) & \frac{\omega^2-1}{6}q(x)\\
            \frac{1-\omega}{6}q(x) & \frac{\omega}{2} \mathcal{H}(x) & \frac{\omega^2-\omega}{6}q(x)\\
            \frac{1-\omega^2}{6}q(x) & \frac{\omega-\omega^2}{6}q(x) & \frac{\omega^2}{2} \mathcal{H}(x) 
            \end{psmallmatrix},
        \end{equation}
        where $\mathcal{H}(x) = \frac{1}{2}[q'(x)]^2-2q(x)^3-xq(x)$ is the PI Hamiltonian, and $q(x)$ solves Painlev\'{e} I. Furthermore, the coefficients
        $\Xi^{(0)}_k(x)$ carry the symmetry
            \begin{equation}\label{P1-3-symm}
                \omega^{-k}\mathcal{S}^T\Xi^{(0)}_k(x)\mathcal{S} = \Xi^{(0)}_k(x).
            \end{equation}
        We do not prove the existence of solutions to this model problem, but indicate that one can construct its solution directly from the $2\times 2$
        problem via a generalized Laplace transform procedure, cf. \cite{JKT} for the sketch of this procedure (and the explicit form of the $3\times 3$ Lax pair generating PI), and \cite{LW} for the application of this 
        procedure in the case of Painlev\'{e} II.

        The Stokes manifold for this Riemann-Hilbert problem contains $5$ lines, which can be found by solving Equation \eqref{STOKES-p1-3} under the
        constraint that one of the parameters $\mathfrak{s}_k$ is zero. These lines should correspond to the tronqu\'{e}e solutions of Painlev\'{e} I. 
        The Stokes data relevant to us is
            \begin{equation}
                \mathfrak{s}_1 = 1-\varkappa, \quad \mathfrak{s}_2 = -1, \quad \mathfrak{s}_3 = 0, \quad \mathfrak{s}_4 = -1, \quad \mathfrak{s}_5 = \varkappa.
            \end{equation}
        We will from here on consider the function $\boldsymbol{\Xi}^{(0)}(\zeta;x)$ evaluated on this choice of Stokes data.
We then define
    \begin{equation}
        \boldsymbol{\Xi}(\zeta;x):=
        \boldsymbol{\Xi}^{(0)}(\zeta;x)\cdot
            \begin{cases}
                1\oplus\sigma_1, & \text{Im }\zeta>0,\\
                \sigma_3\oplus 1, & \text{Im }\zeta<0.
            \end{cases}
    \end{equation}
Then, the matrix $\boldsymbol{\Xi}(\zeta;x)$ satisfies the Riemann-Hilbert problem
    \begin{align}
            \boldsymbol{\Xi}_+(\zeta;x) &= \boldsymbol{\Xi}_-(\zeta;x)J_{\boldsymbol{\Xi}}(\zeta), \qquad\qquad \zeta \in \Gamma_{\boldsymbol{\Xi}},\\
                    \boldsymbol{\Xi}(\zeta;x) &= \mathfrak{f}(\zeta)\left[\mathbb{I} + \sum_{k=1}^{\infty} \frac{\Xi_k(x)}{\zeta^{k/3}}\right]e^{\Upsilon(\zeta;x)}, \qquad\qquad \zeta\to \infty,\label{Xi-asymptotic-expansion}
    \end{align}
where
    \begin{equation}
        J_{\boldsymbol{\Xi}}(\zeta) =
        \begin{cases}
            S_{1} := \mathbb{I} + (1-\varkappa) E_{12},& \zeta\in \gamma_1,\\
            S_{2} := \mathbb{I} - E_{32},& \zeta\in \gamma_2,\\
            S_{4} := \mathbb{I} - E_{21},& \zeta\in \gamma_4,\\
            S_{5} := \mathbb{I} +\varkappa E_{23},& \zeta\in \gamma_5,\\
            (-i\sigma_2) \oplus 1,& \zeta\in \RR_-,\\
            1\oplus (-i\sigma_2), & \zeta\in \RR_+,\\
            S_{-5} := \mathbb{I} + (1-\varkappa)E_{23},& \zeta\in \gamma_{-5},\\
            S_{-4} := \mathbb{I} - E_{21},& \zeta\in \gamma_{-4},\\
            S_{-2} := \mathbb{I} - E_{32},& \zeta\in \gamma_{-2},\\
            S_{-1} := \mathbb{I} +\varkappa E_{12},& \zeta\in \gamma_{-1}.
        \end{cases}
    \end{equation}

\printbibliography

\end{document}